\documentclass[paper=A4,oneside,%
pagesize,%
bibliography=totoc]{scrartcl}

\usepackage[english]{babel}

\usepackage{amsmath,amssymb,amsthm,amsfonts,aliascnt}
\usepackage{thmtools} 
\usepackage{mathtools}
\usepackage{exscale}

\usepackage{stmaryrd} 

\usepackage{bbm}
\usepackage{mathrsfs} 
\usepackage[T1]{fontenc}
\usepackage{lmodern}
\usepackage{microtype}
\usepackage{booktabs}

\usepackage{graphicx}
\usepackage[dvipsnames,table]{xcolor}

\usepackage[numbers,sort&compress]{natbib}
\definecolor{links}{rgb}{0,0.3,0}
\definecolor{mylinks}{rgb}{0.8,0.2,0}
\usepackage[colorlinks=true,citecolor=blue,urlcolor=links,breaklinks=true,linkcolor=mylinks]{hyperref}

\usepackage{cleveref}

\newcommand{\orcidicon}[1]{\href{https://orcid.org/#1}{\protect\includegraphics[height=1.5ex]{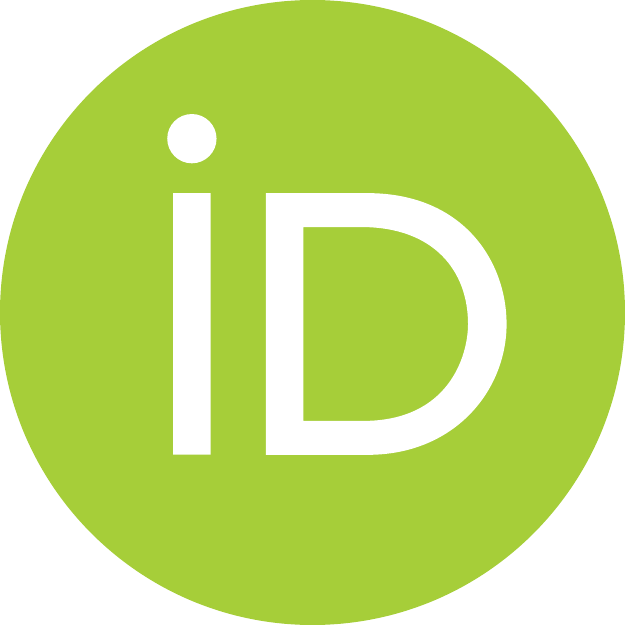}}}
\newcommand{\email}[1]{\href{mailto:#1}{#1}}

\numberwithin{equation}{section}

\newcommand{\Q}{\mathbbm{Q}}

\newcommand{\Z}{\mathbbm{Z}}

\newcommand{\F}{\mathbbm{F}}

\newcommand{\Transpose}{\intercal}
\newcommand{\td}[1][]{\mathrm{d}^{#1}}

\DeclareMathOperator{\Aut}{Aut}

\DeclareMathOperator{\supp}{supp}

\newcommand{\defas}{\mathrel{\mathop:}=}
\newcommand{\set}[1]{
\left\{ #1 \right\}
}

\newcommand{\abs}[1]{
\left\lvert #1 \right\rvert
}
\newcommand{\norm}[1]{
\left\lVert #1 \right\rVert
}

\DeclareMathOperator{\Perm}{Perm}
\DeclareMathOperator{\perm}{perm}
\DeclareMathOperator{\Or}{Or}

\newcommand{\dual}{\star}

\DeclareMathOperator{\Diag}{Diag}

\newcommand{\Period}{\mathcal{P}}
\newcommand{\Hepp}{\mathcal{H}}
\newcommand{\Martin}{\operatorname{\mathsf{M}}}
\newcommand{\MartinPol}{\mathfrak{m}}
\newcommand{\CPP}{J}
\newcommand{\CC}[1]{|#1|}

\newcommand{\Tutte}[1]{\mathsf{T}_{#1}}

\newcommand{\PsiPol}{\Psi}
\newcommand{\KirchPol}{\widetilde{\PsiPol}}

\newcommand{\trans}{\mathcal{T}}

\newcommand{\Graph}[2][1.0]{%
\vcenter{\hbox{\includegraphics[scale=#1]{graphs/#2}}}%
}

\newcommand{\takecoeff}[1]{\left[ #1 \right]}
\newcommand{\takecoefff}[2]{#1[ #2 #1]}

\newcommand{\PointCount}[2]{%
\left\llbracket #1 \right\rrbracket_{#2} %
}

\newcommand{\MapleNote}{\footnote{Maple is a trademark of Waterloo Maple Inc.}}
\newcommand{\MapleTM}{\href{http://www.maplesoft.com/products/Maple/}{\textsf{\textup{Maple}}\texttrademark}}
\newcommand{\nauty}{\href{http://pallini.di.uniroma1.it/}{\texttt{\textup{nauty}}}}
\newcommand{\FORM}{\href{https://www.nikhef.nl/~form/}{\textsc{Form}}}
\newcommand{\HoG}[1]{\href{https://houseofgraphs.org/graphs/#1}{#1}}

\theoremstyle{plain}
\newtheorem{theorem}{Theorem}[section]
\newtheorem{lemma}[theorem]{Lemma}
\newtheorem{corollary}[theorem]{Corollary}
\newtheorem{proposition}[theorem]{Proposition}
\newtheorem{conjecture}[theorem]{Conjecture}

\theoremstyle{definition}
\newtheorem{example}[theorem]{Example}
\newtheorem{definition}[theorem]{Definition}

\theoremstyle{remark}
\newtheorem{remark}[theorem]{Remark}

\newcommand{\Filename}[1]{{\upshape\ttfamily #1}}
\newcommand{\JaxoDraw}{\texttt{\textup{JaxoDraw}}}

\title{Feynman symmetries of the Martin and $c_2$ invariants of regular graphs}

\author{
\thanks{%
	University of Oxford, UK,
	\email{erik.panzer@maths.ox.ac.uk}}
	Erik Panzer
	\orcidicon{0000-0002-9897-5812}
	\and
\thanks{%
    University of Waterloo, ON, Canada,
    \email{kayeats@uwaterloo.ca}}
	Karen Yeats
	\orcidicon{0000-0003-0347-6525}
}
\date{\today}

\begin{document}
\maketitle
\begin{abstract}
    For every regular graph, we define a sequence of integers, using the recursion of the Martin polynomial. This sequence counts spanning tree partitions and constitutes the diagonal coefficients of powers of the Kirchhoff polynomial. We prove that this sequence respects all known symmetries of Feynman period integrals in quantum field theory. We show that other quantities with this property, the $c_2$ invariant and the extended graph permanent, are essentially determined by our new sequence. This proves the completion conjecture for the $c_2$ invariant at all primes, and also that it is fixed under twists. We conjecture that our invariant is perfect: Two Feynman periods are equal, if and only if, their Martin sequences are equal.
\end{abstract}

\tableofcontents

\section{Introduction}\label{sec intro}
In this paper, graphs are undirected and allowed to have self-loops (edges that connect a vertex to itself) and multiedges (several edges connecting the same pair of vertices).

Let $G$ be a $2k$-regular graph, that is, every vertex has the same, even degree $2k>0$. Following \cite{Jaeger:Transition4,FleischnerGenestJackson:CC4}, a \emph{transition} at a vertex $v$ is a partition $\tau(1)\sqcup\ldots\sqcup\tau(k)$ of the $2k$ half-edges $e_i=vw_i$ at $v$ into pairs $\tau(i)=\{e_{\tau'(i)},e_{\tau''(i)}\}$ (made precise in \cref{def transition}).
Given a transition, we can transform $G$ into a smaller graph, which is again $2k$-regular:
\begin{equation*}\label{eq:transition}
    G_{\tau} = (G\setminus v) 
    + w_{\tau'(1)} w_{\tau''(1)}+\cdots+w_{\tau'(k)}w_{\tau''(k)}
\end{equation*}
is obtained by removing $v$ together with its edges $e_i$, and then adding $k$ new edges to match the neighbours $w_i$ of $v$. Let $\trans(v)$ denote the set of all $(2k-1)!!$ transitions at $v$.

In this paper, we study invariants of graphs that solve the recursion
\begin{equation}\label{eq:Martin-recursion}
    \Martin(G) = \sum_{\tau\in\trans(v)} \Martin(G_{\tau}).
\end{equation}
Such invariants are interesting both as pure combinatorics and because they provide a unifying perspective on previously disparate invariants that are interesting in quantum field theory on account of them having the key symmetries of Feynman period integrals.

This recursion was introduced by Martin \cite{Martin:EnumerationsEuleriennes} to define a polynomial that counts circuit decompositions. He considered in particular the 3-term recurrence
\begin{equation}\label{eq:Martin-recursion-4}
    \Martin\Bigg(\Graph[0.5]{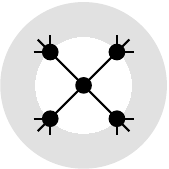}\Bigg)
    =\Martin\Bigg(\Graph[0.5]{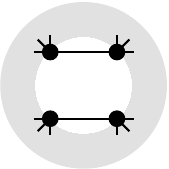}\Bigg)
    +\Martin\Bigg(\Graph[0.5]{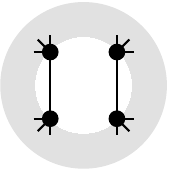}\Bigg)
    +\Martin\Bigg(\Graph[0.5]{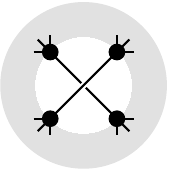}\Bigg)
\end{equation}
for 4-regular graphs, and Las Vergnas \cite{LasVergnas:Martin} extended Martin's considerations to arbitrary even degrees. For matroid theorists, these polynomials can be interpreted as a kind of chromatic polynomial for the transition matroid \cite{Traldi:TransitionMatroidIntroduction}, instead of the usual cycle matroid.
Our focus lies on a certain derivative of Martin's polynomial (see \cref{def:Martin-as-diff}) which, for graphs with 3 or more vertices, can also be defined as follows:
\begin{definition}\label{def:Martin-intro}
The \emph{Martin invariant} $\Martin(G)$ of a regular graph $G$ with even degree and at least 3 vertices, is the non-negative integer defined by the rules:
\begin{enumerate}
    \item If $G$ has a self-loop, then $\Martin(G)=0$.
    \item If $G$ has three vertices (and no self-loop), then $\Martin(G)=1$.
    \item Otherwise, pick any vertex $v$, and define $\Martin(G)$ recursively by \eqref{eq:Martin-recursion}.
\end{enumerate}
\end{definition}
It follows from \cite{LasVergnas:Martin} that all choices of $v$ in rule 3.\ produce the same result, so $\Martin(G)$ is well-defined.  The situation for graphs with fewer than 3 vertices and justification for our normalization is explained in \cref{def:Martin-as-diff} and subsequent discussion; it is not crucial at present.
To illustrate the definition, the Martin invariant of the complete graph $K_5$ is
\begin{equation*}
    \Martin(K_5)
    = 3 \Martin\left(\Graph[0.6]{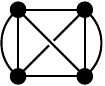}\right)
    = 3 \Martin\left(\Graph[0.6]{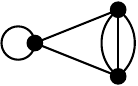}\right) + 6 \Martin\left(\Graph[0.6]{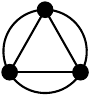}\right)
    = 3\cdot 0 + 6\cdot 1
    = 6.
\end{equation*}
In the first expansion, all 3 transitions produce isomorphic graphs: a perfect matching added to $K_4=K_5\setminus v$. The second expansion produces a self-loop when the transition pairs the parallel edges with each other; the other 2 transitions create $\Graph[0.25]{K3-222}$.

For 4-regular graphs, the Martin invariant was introduced in \cite{BouchetGhier:BetaIso4}. It was shown in particular that $\Martin(G)>0$ if and only if $G$ is 4-edge connected, and that $\Martin(G)$ simplifies into a product if $G$ has a 4-edge cut. We generalize these properties to higher degrees:
\begin{theorem}\label{prop:edge-cuts}
    Let $G$ be a $2k$-regular graph. Then:
    \begin{enumerate}
        \item If $G$ has an edge cut of size less than $2k$, then $\Martin(G)=0$.
        \item If $G$ has no edge cut of size less than $2k$, then $\Martin(G)>0$.
        \item If $G$ has an edge cut of size equal to $2k$, then $\Martin(G)=k!\cdot \Martin(G_1)\cdot\Martin(G_2)$ for the graphs $G_1$ and $G_2$ obtained by replacing one side of the cut with a single vertex:
    \begin{equation}\label{eq:Martin-edge-product}
        \Martin\Bigg(\Graph[0.5]{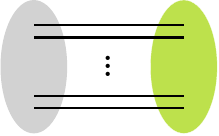}\Bigg)=k!\cdot \Martin\Bigg(\Graph[0.5]{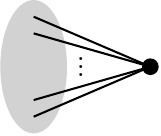}\Bigg)\cdot\Martin\Bigg(\Graph[0.5]{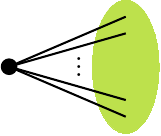}\Bigg).
    \end{equation}
    \end{enumerate}
\end{theorem}
These results are proved in \cref{lem:not-2kcon-zero}, \cref{cor:2kcon-nonzero}, and \cref{lem:2k-cut-product}. In \cref{sec decomp and bound} we prove further results, including a lower bound on non-zero Martin invariants and a characterization of precisely when the lower bound is achieved in terms of total decomposability (\cref{prop:tot-decomp}) and we prove uniqueness of these decompositions (\cref{thm:decomp-unique}).

Note that if one side of the cut in case 3 is just a single vertex, then $G_1=G$.  This is consistent with the full definition of $\Martin(G)$, \cref{def:Martin-as-diff}, where we find that $\Martin(G_2)=1/k!$ for the graph $G_2$ consisting of two vertices with $2k$ edges between them.

The computation of $\Martin(G)$ therefore reduces to graphs $G$ that are \emph{cyclically $(2k+2)$-edge connected}.\footnote{In our context of $2k$-regular graphs, the notions of cyclically $(2k+2)$-edge connected, essentially $(2k+2)$-edge connected, and internally $(2k+2)$-edge connected are all equivalent since the properties of each side of the cut having at least one cycle, having at least one edge, and having at least two vertices are equivalent for $2k$-edge cuts of $2k$ regular graphs.} This means that $G$ is $2k$-edge connected, and requires in addition that the only edge cuts of size $2k$ are the trivial ones that separate a single vertex from the rest of the graph. It is easy to see that such graphs cannot have any 1- or 2-vertex cuts. Our following two results (proven in \cref{sec:product+twist}) give relations from 3- and 4-vertex cuts:
\begin{itemize}
    \item Let $G$ be $2k$-regular and $2k$-edge connected, with a 3-vertex cut. Then the two sides of the cut can be turned into $2k$-regular graphs $G_1$ and $G_2$, by adding edges (but no self-loops) between the cut vertices. Furthermore, we have 
    \begin{equation}\label{eq:martin vertex product}
    \Martin(G)=\Martin(G_1)\cdot\Martin(G_2).
    \end{equation}
    \item  If two regular graphs of even degree are obtained from each other by a double transposition on one side of a 4-vertex cut, then their Martin invariants agree:
    \begin{equation}\label{eq:twist}
        \Martin\Bigg(\ \Graph[0.3]{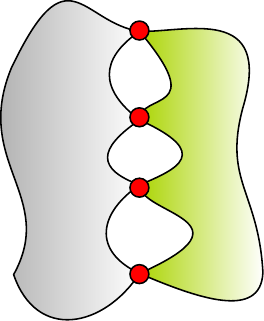}\ \Bigg)=\Martin\Bigg(\ \Graph[0.3]{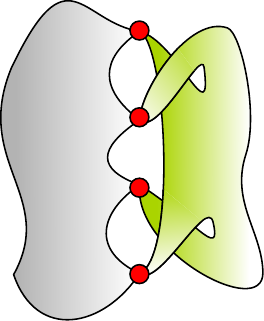}\ \Bigg).
    \end{equation}
\end{itemize}
These operations on graphs are called \emph{product} and \emph{twist} in \cite{Schnetz:Census}. For an example of the product identity, consider the 3-vertex cut highlighted in red (the 3 vertices stacked vertically) in the following graph:
\begin{equation*}
    \Martin\Bigg(\ \Graph[0.4]{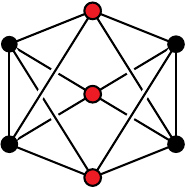}\ \Bigg)
    = \Martin\Bigg(\ \Graph[0.4]{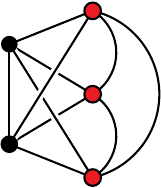}\ \Bigg) \cdot \Martin\Bigg(\ \Graph[0.4]{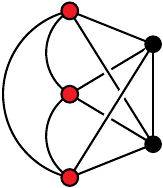}\ \Bigg)
    = \Martin(K_5)\cdot \Martin(K_5)
    = 36.
\end{equation*}

\subsection{Spanning trees and diagonals}
The Martin invariant has another combinatorial description, in terms of spanning trees. Given a graph $G$ with $n$ vertices, a \emph{spanning tree} $T$ of $G$ is a subset of the edges of $G$ such that $T$ has $n-1$ elements and forms a connected subgraph. We prove (see \cref{sec diags}):
\begin{theorem}\label{thm:Martin=ST-partitions}
    For every vertex $v$ in a $2k$-regular graph $G$ with at least $3$ vertices, the number of partitions of the edge set of $G\setminus v$ into $k$ spanning trees is equal to $\Martin(G)$.
    In particular, the number of such partitions is independent of $v$.
\end{theorem}
It is worth emphasizing that this is quite surprising: the Martin polynomial, by its nature, counts circuit partitions, but it is not at all apparent at the outset that it can also be used to count spanning tree partitions.\footnote{Thanks to a referee for drawing attention to the unexpected nature of this connection.}

As an example of the theorem, consider the $4$-regular complete graph $G=K_5$. Then the 6 edges of $K_5\setminus v\cong K_4$ allow precisely $\Martin(K_5)=6$ partitions into pairs of spanning trees:
\begin{equation}\label{eq:STP-K4}\begin{aligned}
    \Graph[0.55]{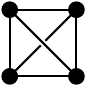}&
    \ =\ \Graph[0.55]{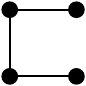}\,\cup\,\Graph[0.55]{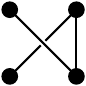}
    \ =\ \Graph[0.55]{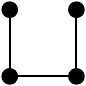}\,\cup\,\Graph[0.55]{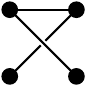}
    \ =\ \Graph[0.55]{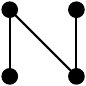}\,\cup\,\Graph[0.55]{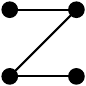}
    \\ &
    \ =\ \Graph[0.55]{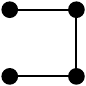}\,\cup\,\Graph[0.55]{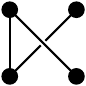}
    \ =\ \Graph[0.55]{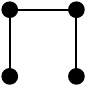}\,\cup\,\Graph[0.55]{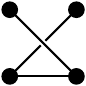}
    \ =\ \Graph[0.55]{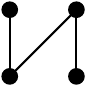}\,\cup\,\Graph[0.55]{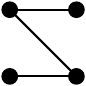}
\end{aligned}\end{equation}
The theorem implies that $\Martin(G)$ is determined by the cycle matroid of $G\setminus v$.\footnote{
    However, the Martin invariant is \emph{not} determined by the Tutte polynomial (see \cref{sec:phi4}).
} We can also view the theorem as a relation between $\Martin(G)$ and certain coefficients of a polynomial:
\begin{definition}\label{def coeff extr}
    Given a formal power series $F=\sum_{n} c_{n} x^n$ in several variables $x_i$, we write $\takecoeff{x^n}F=c_n$ for the coefficient of any monomial $x^n=\prod_i x_i^{n_i}$.
\end{definition}
The \emph{Symanzik} or \emph{dual Kirchhoff} polynomial\footnote{This is sometimes called \emph{graph polynomial}, e.g.\ in \cite{BrownSchnetz:K3phi4,BlochEsnaultKreimer:MotivesGraphPolynomials}.} $\PsiPol_G$ of any graph $G$ is the sum
\begin{equation}\label{eq:psipol}
    \PsiPol_G = \sum_{T}
    \prod_{e\notin T} x_e
    \in \Z[x_1,\ldots,x_m]
\end{equation}
over all spanning trees $T$ of $G$. The variables $x_e$ are labelled by the $m$ edges of $G$. Applied to the complement of a vertex in a $2k$-regular graph $G$, we can restate \cref{thm:Martin=ST-partitions} as
\begin{equation*}
    \takecoefff{\big}{x_1^{k-1}\cdots x_m^{k-1}} \PsiPol_{G\setminus v}^k
    = k!\cdot \Martin(G)
\end{equation*}
(see \cref{thm martin as coeff} for the proof).
We will use this identity to show the following symmetry under planar duality (see \cref{sec duality}):
\begin{theorem}\label{thm:Martin-duality}
    Suppose that $G_1$ and $G_2$ are 4-regular graphs that have vertices $v_1,v_2$ such that $G_1\setminus v_1$ is a planar dual of $G_2\setminus v_2$. Then $\Martin(G_1)=\Martin(G_2)$.
\end{theorem}

\subsection{Residues}\label{sec residues intro}
By linearity, the residues $\Martin(G)\mod q$, for a fixed integer $q$, also solve the recursion \eqref{eq:Martin-recursion}. Two such residues coincide with invariants of graphs that have been studied previously: the \emph{graph permanent} and the \emph{$c_2$-invariant}.  These two invariants are quite different in character but both were introduced because they have key symmetries of Feynman period integrals and so should carry some quantum field theoretic information despite their purely combinatorial nature.

Let $G$ be a $2k$-regular graph and choose two vertices $v\neq w$ such that $v$ has no self-loop. Call the remaining vertices $1,\ldots,n$ and pick an orientation of the $kn$ edges of $G\setminus v$. Their $n\times kn$ incidence matrix $A$ has entries $A_{ie}=\pm 1$ if vertex $i\in\set{1,\ldots,n}$ is the head or tail of edge $e$, and $A_{ie}=0$ otherwise. Stacking $k$ copies of $A$, we get a square $kn\times kn$ matrix, denoted $A^{[k]}$. It was shown in \cite{CrumpDeVosYeats:Permanent} that the square of the permanent
\begin{equation*}
    \Perm(G)=\perm A^{[k]} \in \Z
\end{equation*}
has a well-defined residue modulo $k+1$: this residue does not depend on the edge orientations or the choices of $v$ and $w$. When $k+1$ is composite, then $\Perm(G)\equiv 0 \mod (k+1)$, so the permanent is interesting only when $k+1$ is prime. We prove (see \cref{thm:Martin-Perm2})
\begin{theorem}\label{intro:thm-Martin-Perm}
    Let $G$ be a $2k$-regular graph $G$ with $n$ vertices. If $k+1$ is prime, then
    \begin{equation}
        \Martin(G)\equiv (-1)^{n-1}\Perm(G)^2 \mod (k+1).
    \end{equation}
\end{theorem}
This implies that the permanent inherits all identities of the Martin invariant mentioned above, explaining all properties of the permanent obtained in \cite{CrumpDeVosYeats:Permanent} using different methods.

The permanent of a graph was enriched in \cite{Crump:ExtendedPermanent} to an infinite sequence of residues. Let $G^{[r]}$ denote the $2kr$-regular graph obtained from $G$ by replacing each edge with a bundle of $r$ parallel edges. The \emph{extended graph permanent} \cite{Crump:ExtendedPermanent} consists of the permanents of $G^{[r]}$ where $kr+1$ is prime. For example, for a 4-regular graph $G$, it is the sequence
\begin{equation*}
    \left(\Perm(G),\Perm(G^{[2]}),\Perm(G^{[3]}),\Perm(G^{[5]}),\ldots\right) \in 
    \F_3\times \F_5\times \F_7\times \F_{11} \times \cdots
\end{equation*}
of residues in $\F_p=\Z/p\Z$ of the permanents of $G^{[(p-1)/2]}$ for all odd primes $p$.
\begin{definition}
    The \emph{Martin sequence} of a regular graph $G$ is the list of integers $\Martin(G^{[r]})$, indexed by all positive integers $r$. We denote this sequence as
\begin{equation}\label{eq:Martin-seq}
    \Martin(G^{\bullet}) = \Big(\Martin(G^{[1]}),\Martin(G^{[2]}),\Martin(G^{[3]}),\ldots\Big).
\end{equation}
\end{definition}
By \cref{thm:Martin-Perm2}, the Martin sequence determines the extended graph permanent.\footnote{Without squaring, the residues $\Perm(G^{[r]})\mod (rk+1)$ are only defined up to sign. These signs depend on the choices of $v$ and $w$ and the orientation of the edges.}

The $c_2$-invariant was introduced in \cite{Schnetz:Fq}. It is a sequence of residues $c_2^{(q)}(G)\in\Z/q\Z$ indexed by all prime powers $q$ and defined by
\begin{equation*}
    \PointCount{\PsiPol_G}{q} \equiv q^2\cdot c_2^{(q)}(G) \mod q^3
\end{equation*}
where $\PointCount{\PsiPol_G}{q}=\big|\big\{x\in\F_q^N\colon \PsiPol_G(x)=0\big\}\big|$ denotes the number of points on the hypersurface $\set{\PsiPol_G=0}$ over the finite field $\F_q$ with $q$ elements. In particular, the restriction of $c_2$ to primes ($q=p$) determines a sequence of residues in $\F_2\times \F_3\times \F_5\times\F_7\times\cdots$.  We will only consider $c_2$ restricted to primes as this is where our techniques are effective.

In many examples, the $c_2$-invariant at primes ($q=p$) appears to be congruent to coefficients of modular forms \cite{BrownSchnetz:ModularForms,Schnetz:GeometriesPQFT}, which has been established rigorously in a few cases \cite{BrownSchnetz:K3phi4,Logan:NewCYphi4}.  This connection hints at the deep geometric and number theoretic nature of the $c_2$-invariant.

We prove (see \cref{thm:c2-martin}):
\begin{theorem}\label{thm:c2-Martin-intro}
    For every $4$-regular graph $G$ with at least 6 vertices, and every prime $p$, and any vertex $v$,
    \begin{equation*}
        c_2^{(p)}(G\setminus v)\equiv \frac{\Martin(G^{[p-1]})}{3p} \mod p.
    \end{equation*}
\end{theorem}
This shows that the $c_2$-invariant at primes is also determined by the Martin sequence. 
Putting these two relations side by side we have that the extended graph permanent and the $c_2$-invariant capture different residues of the Martin invariants: If $G$ is 4-regular, then
\begin{align*}
    \Martin(G^{[r]})&\equiv (-1)^{n-1} \Perm(G^{[r]})^2 \mod (2r+1)
    \quad\text{if $2r+1$ is prime, and}\\
    \Martin(G^{[r]})&\equiv 3p\cdot c_2^{(p)}(G\setminus v)\mod p^2
    \quad\text{if $p=r+1$ is prime.}
\end{align*}

For the $c_2$ invariant this operation of removing a vertex from a 4-regular graph is quite important.  If $H$ is obtained from a 4-regular graph $G$ by removing a vertex we say that $H$ is a \emph{decompletion} of $G$ and $G$ is the \emph{completion} of $H$.  Note that the completion is unique but a graph may have many non-isomorphic decompletions.  The extended graph permanent also implicitly involves the decompletion since it works with the incidence matrix of $G\setminus v$ rather than with $G$ directly.

The $c_2$ invariant, like the extended graph permanent, inherits the symmetries of the Martin invariant from \cref{thm:c2-martin}, some of which were not previously known (e.g.\ the twist). Most interestingly, although the left-hand side of the theorem involves $G\setminus v$, note that the right-hand side depends only on $G$ but not $v$.  Thus \cref{thm:c2-Martin-intro} implies:
\begin{corollary}[Completion invariance at primes]\label{cor c2 completion}
For all primes $p$ and any two vertices $v,w$ of a 4-regular graph $G$, we have
$c_2^{(p)}(G\setminus v)\equiv c_2^{(p)}(G\setminus w) \mod p$.
\end{corollary}
This was conjectured by Brown and Schnetz in 2010, see \cite[Conjecture~4]{BrownSchnetz:K3phi4}. It is indicative of a relation between the highest weight parts of the cohomology groups of the two graph hypersurfaces $\{\PsiPol_{G\setminus v}=0\}$ and $\{\PsiPol_{G\setminus w}=0\}$, as studied in \cite{BlochEsnaultKreimer:MotivesGraphPolynomials,BrownDoryn:FramingsForGraphHypersurfaces}.

Over the last thirteen years, this conjecture was attacked by techniques from algebraic geometry, combinatorics and physics, but with very little to show for these efforts until the proof of the conjecture for $p=2$, first only for graphs $G$ with an odd number of vertices \cite{Yeats:SpecialCaseCompletion} and then for all graphs \cite{HuYeats:c2p2} by the second author with Hu.  These proofs used the combinatorial interpretation of diagonal coefficients that will figure prominently in \cref{sec diags} and \cref{sec c2} along with some intricate involutions.  The connection with the Martin invariant proved in the present paper provides the missing piece to finally settle this conjecture from over a decade ago, for all primes.
In full, \cite[Conjecture~4]{BrownSchnetz:K3phi4} also covers $c_2$ at prime powers $q=p^n$ with $n>1$, and these cases remain open.

\subsection{Feynman integrals}

The product and twist identities discussed above (see also \cref{prop:3vertex-cut,prop:twist} for precise statements) were introduced for 4-regular graphs by Schnetz \cite{Schnetz:Census}. However, these identities were discovered not for the Martin invariant, but for a completely different function of graphs:  The \emph{period} of $G$ is defined as the (Feynman) integral
\begin{equation}\label{eq:period}
    \Period(G) = \left( \prod_{e=2}^m \int_0^{\infty} \td x_e \right)
    \left. \frac{1}{\PsiPol_{G}^2}\right|_{x_1=1}
\end{equation}
whenever this converges \cite{BlochEsnaultKreimer:MotivesGraphPolynomials}. The motivation to study these numbers comes from quantum field theory, where they give contributions to the beta function that drives the running of the coupling constant.
Systematic calculations of these periods go back to \cite{BroadhurstKreimer:KnotsNumbers,Broadhurst:5loopsbeyond}.

Schnetz showed in \cite{Schnetz:Census} that if $G$ is 4-regular and cyclically 6-connected, then for any vertex of $v$, the period $\Period(G\setminus v)$ is well-defined (convergent) and independent of the choice of $v$. The period therefore defines a function
\begin{equation*}
    G \mapsto \Period(G\setminus v)
\end{equation*}
on the set of cyclically 6-connected 4-regular graphs. It satisfies the product, twist and duality identities \cite{Schnetz:Census}, which suggests a relation between the period and the Martin invariant. Indeed, we find two such connections.

\begin{figure}
    \centering
    \includegraphics[width=0.65\textwidth]{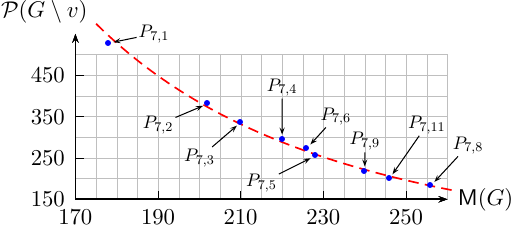}
    \caption{The periods and Martin invariants of primitive $\phi^4$ graphs \cite{Schnetz:Census} with $9$ vertices (7 loops decompleted). The dashed curve is a fit $\Period(G\setminus v)\approx 32150\cdot \Martin(G)^{-3.015}$.}%
    \label{fig:phi4-martin-period}%
\end{figure}
Firstly, the Martin invariant and the period are correlated numerically. Graphs with smaller $\Martin(G)$ tend to have larger periods. This approximate relation is illustrated in \cref{fig:phi4-martin-period} for the 4-regular graphs with $9$ vertices. As the plot shows, the period can be estimated surprisingly well from $\Martin(G)$ and a simple power law. It is remarkable that the period integral \eqref{eq:period} is so strongly correlated with the number of partitions into spanning trees (\cref{thm:Martin=ST-partitions}). Furthermore, the period, and hence the Martin invariant, are also correlated to the Hepp bound \cite{Panzer:HeppBound}, which is a tropical version of the period.

Secondly, even though the correlation just mentioned is only approximate, we find that the Martin \emph{sequence} can detect precisely when two graphs have the same period. We calculated Martin invariants for a large number of 3- and 4-regular graphs (\cref{sec:data}). Comparing with the known periods \cite{PanzerSchnetz:Phi4Coaction,BorinskySchnetz:RecursivePhi3} and the Hepp bound, supports the following conjecture (we state a 3-regular version in \cref{con:Martin-perfect-phi3}).
\begin{conjecture}\label{con:Martin-perfect}
    Cyclically 6-connected 4-regular graphs $G_1$ and $G_2$ have equal period $\Period(G_1\setminus v_1)=\Period(G_2\setminus v_2)$ if and only if they have equal Martin sequences $\Martin(G_1^{\bullet})=\Martin(G_2^{\bullet})$.
\end{conjecture}
Via \cref{thm:c2-Martin-intro}, our conjecture implies that the $c_2$ invariants of graphs with equal periods agree at all primes. This is expected, also at prime powers, according to \cite[Remark~2.11]{Schnetz:Fq}. Our conjecture thus implies the case $q=p$ of \cite[Conjecture~5]{BrownSchnetz:K3phi4}.

\Cref{con:Martin-perfect} also suggests that the period (and Hepp bound) of a graph should be computable from the Martin sequence. Indeed, this is proved in forthcoming work with Francis Brown. One important observation towards this is that Martin sequences are P-recursive (equivalently, their generating series are D-finite), as we prove in \cref{cor p recursive}.

There is another, in fact classical, connection between the Martin \emph{polynomial} and quantum field theory: the Martin polynomial appears as the symmetry factor $S_N$ in $O(N)$ vector models. Only very recently, however, also the Martin \emph{invariant} was noted to have a physical meaning in this context, namely it emerges at $N=-2$ for loop-erased random walks \cite{WieseFedorenko:FieldErasedRW,ShapiraWiese:ExactLERW}.  We summarize this connection in \cref{sec:O(n)}.

\subsection{Origin}
In fact, it was in this physics context that we discovered the Martin invariant:

The need to compile a table of symmetry factors for the calculation \cite{KompanietsPanzer:phi4eps6} triggered an, at first experimental, study of structural properties of the polynomial $S_N$ and its correlation with the period $\Period(G)$. Trying different values for $N$ produced the remarkable discovery in \cref{fig:phi4-martin-period}, namely that for $N=-2$, not only is the correlation strong, but the pairs $\set{P_{7,4},P_{7,7}}$ (see \cref{fig:duality}) and $\set{P_{7,5},P_{7,10}}$ of graphs with the same periods also have exactly (and not just approximately) equal symmetry factors $S_N$. Even graphs with unknown periods supported this connection, via the Hepp bounds computed in \cite{Panzer:HeppBound} which are a very good proxy for periods. The paper \cite{BouchetGhier:BetaIso4} suggested focusing on the recursive \cref{def:Martin-intro} which we used to prove the relations between $\Martin(G)$, $\Perm(G) \mod 3$, and $c_2^{(2)} \mod 2$ (the argument after \cref{ex:c2(2)}). We then applied the idea of duplicating edges from \cite{Crump:ExtendedPermanent} to upgrade the Martin invariant to the Martin sequence, and found the proofs of \cref{intro:thm-Martin-Perm,thm:c2-Martin-intro} given in \cref{sec:Perm2-recurrence,sec 3 inv recurrence}.
Only in hindsight did we note that a simpler version of our proof technique relates the Martin invariant to spanning tree partitions (\cref{thm:Martin=ST-partitions}), triggered by Francis Brown's suggestion to relate $c_2$ to the diagonal coefficients of the Symanzik polynomial---rather than the third denominator.

\subsection{Conclusions}

We unified the theory of several invariants of graphs that were motivated by Feynman integrals in particle physics, by relating them to a single, combinatorial invariant counting spanning tree partitions:
\begin{itemize}
    \item All known invariants of graphs that respect the identities of the Feynman period (i.e.\ the $c_2$ invariant, the extended graph permanent, the Hepp bound, and the period integral itself) are determined by the Martin sequence.
\end{itemize}
In this paper, we prove this for the extended permanent and for the $c_2$-invariant at primes. The determination of $c_2$ at prime powers, the Hepp bound, and the period integral, all in terms of only the Martin sequence, are work in progress with Francis Brown.

Secondly, we have learned that this integer sequence has a rich theory and satisfies a number of interesting relations. Most importantly:
\begin{itemize}
    \item The number of spanning tree partitions fulfils Martin's recursion at a vertex.
\end{itemize}
Guided by applications in physics, we focused on regular graphs, where this observation can be interpreted as a relation between spanning tree partitions and circuit partitions (via the Martin polynomial). However, note that the sequence of spanning tree partitions, i.e.\ the diagonal of the Symanzik polynomial, is defined for arbitrary graphs---not just for decompletions of regular graphs. Our proof of \cref{thm martin as coeff} shows directly a recurrence relation for this diagonal, without any assumptions on regularity. Hence in hindsight, we could \emph{define} the Martin sequence as this diagonal, and much will generalize to arbitrary graphs, or even matroids (see \cref{rem:perm-non-regular,rem:duality-matroids}).

Apart from generalizing degrees, the Martin polynomial is also defined for \emph{directed} Eulerian graphs, where transitions are restricted to pair incoming half-edges with outgoing half-edges. For regular digraphs with $k$ incoming and $k$ outgoing edges at every vertex, this polynomial is divisible by $x(x+1)\cdots(x+k-2)$, see \cite[T\'eor\`eme~3.2]{LasVergnas:Martin}. By analogy, the derivative at $x=2-k$ will have similar structural properties as our unoriented case (derivative at $x=4-2k$). It is tempting to wonder if this invariant of digraphs is related to counting partitions into arborescences, i.e.\ the diagonal of an oriented version of the Symanzik polynomial, and hence oriented versions of permanents and period integrals.

Another natural question is to identify the minimal and maximal values that $\Martin(G)$ can take, for primitive regular graphs with a fixed degree and loop order; and to characterize the graphs that achieve those values. \Cref{thm:zigzag-min} for 4-regular graphs and \cref{conj:prism-min} for 3-regular graphs suggest that the minimizers are highly structured and could perhaps be identified for all degrees. There also appears to be a unique graph at each degree and loop order that \emph{maximizes} the Martin invariant, see \cref{tab:Martin3-max,tab:Martin4-max}. Those are harder to describe, but the same graphs appear to also maximize the number of connected sets \cite{CambieGoedgebeurJokken:MaxConReg}, and they are also the unique minimizers of the Hepp bound.

The unexplained identities $\Martin(P_{8,30}^{[r]})=\Martin(P_{8,36}^{[r]})$ and $\Martin(P_{8,31}^{[r]})=\Martin(P_{8,35}^{[r]})$ 
of the graphs in \cref{fig:830-836} remain an intriguing open problem. We verified these equalities for $r\leq 8$, but we cannot explain why these pairs of graphs share the same number of spanning tree partitions. The currently known set of Martin sequence identities (twist, product, duality, and Fourier-split) does not connect these graphs. This strongly suggests that there exist further mechanisms that relate Martin sequences (and thus Feynman periods), and it would be interesting to discover such new symmetries of Feynman integrals.

\medskip

The sections of this paper are largely independent of each other.

\subsection{Acknowledgments}
Erik Panzer dedicates this work to his father, Michael Erhard Panzer, 1955--2022. In admiration, sorrow, gratitude, and loving memory.

We thank Simone Hu for discussions throughout this project and for confirming at an early stage that the available data for the $c_2$-invariant at prime $p=2$ is consistent with the relation to the Martin invariant. We further thank Simone for her work in refining the approach to the $c_2$-invariant via counting partitions of edges into trees and forests \cite{Hu:Master}, which is also a key part of our work here. We also thank Francis Brown for suggesting the connection between $c_2$ and the diagonal coefficient, which allows for a simplification (\cref{lem:c2-from-diag}) of the proof of \cref{thm:c2-Martin-intro} for primes $p\neq 3$.  Finally, we thank the two referees for their very careful reading and thoughtful suggestions.

Erik Panzer is funded as a Royal Society University Research Fellow through grant {URF{\textbackslash}R1{\textbackslash}201473}.
He would like to thank the Isaac Newton Institute for Mathematical Sciences, Cambridge, for support and hospitality during the programme \href{https://www.newton.ac.uk/event/ncn2/}{\emph{New connections in number theory and physics}} where work on this paper was undertaken. This event was supported by EPSRC grant no.\ EP/R014604/1.
Karen Yeats is supported by an NSERC Discovery grant and the Canada Research Chairs program.
Both authors are grateful for the hospitality of Perimeter Institute where part of this work was carried out.
Research at Perimeter Institute is supported in part by the Government of Canada through the
Department of Innovation, Science and Economic Development Canada and by the province of
Ontario through the Ministry of Economic Development, Job Creation and Trade. This research
was also supported in part by the Simons Foundation through the Simons Foundation Emmy
Noether Fellows Program at Perimeter Institute.
Both authors are grateful to Max Planck Institute for Mathematics in Bonn for its hospitality and financial support when final minor edits were implemented.

Many figures in this paper were created using {\JaxoDraw} \cite{BinosiTheussl:JaxoDraw}.

\section{The Martin polynomial}\label{sec martin poly}

Let $G$ be an undirected graph where every vertex has even degree. Multiple edges between the same pair of vertices are allowed. At each vertex, we may also have any number of self-loops, that is, edges connecting the vertex to itself.

It will be helpful to think of graphs in terms of half-edges. The edges of the graph define a perfect matching of the half-edges of the graph, each half-edge is incident with exactly one vertex, and every vertex $v$ has a \emph{corolla} of incident half-edges. The set of the other halves of the edges incident to $v$ will be called the \emph{$h$-neighbourhood} of $v$. The size of this set is always equal to the degree of $v$; in particular, a self-loop at $v$ contributes both of its half-edges to the $h$-neighbourhood of $v$. The neighbourhood in the traditional graph theory sense, namely the set of vertices adjacent to $v$, can be recovered from the $h$-neighbourhood by replacing each half-edge in the $h$-neighbourhood with the vertex incident to that half-edge. When helpful for clarity we will call the graph theory notion of neighbourhood the \emph{$v$-neighbourhood}. Whenever there are self-loops or multiedges at $v$, then the size of the $v$-neighbourhood is less than the degree of $v$.

\begin{definition}\label{def transition}
    A \emph{transition} at a vertex $v$ is a perfect matching of the half-edges of the $h$-neighbourhood of $v$. We write $\trans_G(v)$ or simply $\trans(v)$ for the set of all transitions at $v$.
    A \emph{transition system} $P$ is a choice of one transition $P(v)\in\trans(v)$ for every vertex of $G$.
\end{definition}
Note that the notion of transition could equally well be defined as a perfect matching of the corolla at $v$.  The benefit of working with the $h$-neighbourhood is that for $v$ with no self-loops, and $\tau$ a transition at $v$, the graph $G_\tau$ (see the beginning of \cref{sec intro}) can be defined very naturally to be the graph obtained from $G$ by removing $v$ and its corolla and matching the half-edges of the $h$-neighbourhood of $v$ into new edges according to $\tau$. 

Since the degree of every vertex $v$ is even, the sets $\trans(v)$ are non-empty and transition systems exist. Concretely, a $v$ of degree $d$ has precisely
\begin{equation}\label{eq:trans(v)}
    \abs{\trans(v)}
    =
    (d-1)!!=\frac{d!}{(d/2)!\cdot 2^{d/2}}
\end{equation}
transitions, and the graph $G$ has
$\prod_v (d_v-1)!!$
transition systems. A transition $\tau\in\trans(v)$ can be interpreted as a traffic sign that forces anyone who enters $v$ from half-edge $e$, to leave by half-edge $e'$ where $e'$ is the partner of $e$ in the matching $\tau$. Traversed in this way, a transition system $P$ determines a partition of the edges of $G$ into circuits (closed trails). We write $\CC{P}$ for the number of these circuits. For an example see \cref{fig:transitions}. A detailed description of circuits in terms of half-edges can be found on \cite[page~181]{Traldi:TransitionMatroidIntroduction}.
\begin{figure}
    \centering
    $ P\in\set{
    \Graph[0.55]{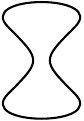},\quad
    \Graph[0.55]{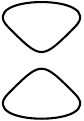},\quad
    \Graph[0.55]{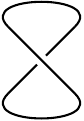}
    }
    $
    \qquad
    $ \MartinPol\left(\Graph[0.6]{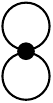},x\right) = 1 + (x-2) + 1 = x$
    \caption{The three circuit partitions of a 2-rose and its Martin polynomial.}%
    \label{fig:transitions}%
\end{figure}
\begin{definition}\label{def:Martin-pol}
    The \emph{Martin polynomial}\footnote{In the very degenerate case where $G$ has no edges, $\MartinPol(G,x)=1/(x-2)$ is not a polynomial.} is a sum over all transition systems, given by
    \begin{equation}\label{eq:martin-pol}
        \MartinPol(G,x) = \sum_{P} (x-2)^{\CC{P}-1}
        \in\Z[x].
    \end{equation}
\end{definition}
See \cref{tab:MartinPol} for some examples. This polynomial was introduced by Martin \cite{Martin:EnumerationsEuleriennes} and extended by Las Vergnas \cite{LasVergnas:Martin}.\footnote{Our notation follows \cite{LasVergnas:Martin}; Martin writes $P(G,x)$.} In more recent literature, results for the Martin polynomial are often stated in terms of the \emph{circuit partition polynomial}, defined as\footnote{In \cite{Bollobas:EvalCPP,ArratiaBollobasSorkin:InterlaceNew} this is denoted $r_G(x)$, we use $\CPP(x)$ following \cite{EllisMonaghan:NewMartin,EllisMonaghan:idCPP,EllisMonaghan:MartinMisc}.}
\begin{equation}\label{eq:circuit-pol}
    \CPP(G,x) = x \cdot \MartinPol(G,x+2)
    = \sum_{P} x^{\CC{P}}
    \in\Z[x].
\end{equation}
The coefficients of $\CPP(G,x)=\sum_k r_k(G) x^k$ count the edge partitions into $k$ circuits. For example, the linear coefficient $r_1(G)=\CPP'(G,0)=\MartinPol(G,2)$ counts the Eulerian circuits of $G$. The precise notion of circuit here is based on half-edges, see \cite[p.~262]{Bollobas:EvalCPP} or \cite{EllisMonaghan:MartinMisc} for details. Further evaluations of these polynomials can be found in \cite{Bollobas:EvalCPP,EllisMonaghan:NewMartin}.
There is also an integral representation \cite{MooreRussel:GraphIntegral} for the values of $J(G,x)$ at positive integers $x$.

It is clear from \cref{def:Martin-pol} that $\MartinPol(G,x)$ solves the recursion \eqref{eq:Martin-recursion}. Note however that if one applies a transition $\tau$ that pairs both ends of a self-loop, then $G_{\tau}$ will have a \emph{free loop}, that is, a loop not connected to any vertex, which must still be counted. This can be avoided with the following lemma.
\begin{lemma}\label{lem:martin-self-loop-rec}
    If $e$ is a self-loop at a vertex $v$ which has even degree $d\geq 2$, then
\begin{equation}\label{eq:martin-self-loop-rec}
    \MartinPol(G,x)=(x+d-4)\cdot\MartinPol(G\setminus e, x).
\end{equation}
\end{lemma}
\begin{proof}
    Given any transition $\tau'\in\trans_{G'}(v)$ of $G'=G\setminus e$, we can pair the half-edges $e',e''$ of $e$ to obtain a transition $\tau=\tau'\cup\set{e',e''}\in\trans_G(v)$ in $G$. We can also pick one of the $d/2-1$ pairs $\set{i,j}\in\tau'$ and interlace $e$ in either direction, to obtain two further transitions $\tau_{ij},\tau_{ji}\in\trans_G(v)$ of the form $\tau_{ij}=\tau'\setminus \set{\set{i,j}}\cup\set{\set{i,e'},\set{j,e''}}$. These constructions produce every transition in $\trans_G(v)$ exactly once.

    We can thus group the sum \eqref{eq:martin-pol} by transition systems of $P'$ of $G'$. For each $P'$, we find one transition system $P$ of $G$ with $\CC{P}=\CC{P'}+1$, and $d-2$ with $\CC{P}=\CC{P'}$. Hence we conclude the claim, with the prefactor $(x-2)+(d-2)$.
\end{proof}
\begin{table}
    \centering
    \begin{tabular}{lccccccc}
    \toprule
        $G$ & $\Graph[0.6]{2rosev}$ & $\Graph[0.6]{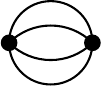}$ & $\Graph[0.6]{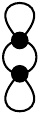}$ & $\Graph[0.6]{K3-222}$ & $\Graph[0.6]{K3-113}$ & $\Graph[0.6]{K4-112}$ & $\Graph[0.6]{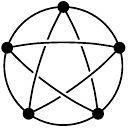}$ \\
    \midrule
        $\MartinPol(G,x)$ & $x$ & $3x$ & $x^2$ & $x^2+6x$ & $3x^2$ & $5x^2+12x$ & $15x^2+36x$ \\
    \bottomrule
    \end{tabular}
    \caption{Martin polynomials of some 4-regular graphs.}%
    \label{tab:MartinPol}%
\end{table}
We can iterate \cref{lem:martin-self-loop-rec} to remove all self-loops \cite[Proposition~4.1]{LasVergnas:Martin}. Hence we can compute $\MartinPol(G,x)$ without ever creating free loops.
As a special case, we find
\begin{equation}\label{eq:Martin-rose}
    \MartinPol(C_1^{[k]}, x)=x(x+2)\cdots(x+2k-4)
\end{equation}
for the graph $C_1^{[k]}$ consisting of a single vertex with $k>1$ self-loops. This graph is called the \emph{$k$-rose}. Induction via \eqref{eq:Martin-recursion} with a rose as the base case shows that $\MartinPol(G,x)$ is divisible by the right-hand side of \eqref{eq:Martin-rose} whenever $G$ has a vertex of degree $2k$ or more.
\begin{corollary}\label{cor:mpol-loop-zero}
    If $G$ is $2k$-regular, has at least two vertices, and at least one self-loop, then $\MartinPol(G,x)$ is divisible by $(x+2k-4)^2$.
\end{corollary}
\begin{proof}
    The self-loop $e$ gives one factor of $x+2k-4$ directly in \eqref{eq:martin-self-loop-rec}. Since $G\setminus e$ still has at least one vertex of degree $2k$, we get a second factor from $\MartinPol(G\setminus e,x)$.
\end{proof}
The Martin polynomial has numerous further interesting properties, however for our purposes we only use the above. Mainly, we conclude that the Martin invariant as defined in the introduction, \cref{def:Martin-intro}, makes sense: the number $\Martin(G)$ is uniquely determined, because we can obtain it from $\MartinPol(G,x)$.  Using this insight we give a new direct definition of the Martin invariant and then observe its equivalence with \cref{def:Martin-intro}. 
\begin{definition}\label{def:Martin-as-diff}
    The \emph{Martin invariant} of a $2k$-regular graph ($k\geq1$) is the derivative
\begin{equation}\label{eq:Martin-as-diff}
    \Martin(G)=\frac{4\cdot(-1)^{k}}{(k-2)!\cdot (2k)!}\MartinPol'(G, 4-2k)
    \in \Q
    .
\end{equation}
\end{definition}
Since taking a derivative is a linear map, any derivative of $\MartinPol$ inherits the recursion \eqref{eq:Martin-recursion}. Note also that $\MartinPol'(G,4-2k)$ vanishes for every $2k$-regular graph with a self-loop and more than one vertex, by \cref{cor:mpol-loop-zero}. Therefore, this derivative fulfils 1.\ and 3.\ of \cref{def:Martin-intro}. Our choice of normalization (part 2.\ of \cref{def:Martin-intro}) determines the prefactor in \eqref{eq:Martin-as-diff}. Indeed, \eqref{eq:Martin-as-diff} gives
\begin{equation}\label{eq:Martin-rose-dipole}
    \Martin(C_1^{[k]}) = \frac{2^{k}}{(2k)!}
    \quad\text{and}\quad
    \Martin(C_2^{[k]})
    = \frac{1}{k!}
\end{equation}
for the rose and dipole graphs. The dipole $C_2^{[k]}=K_2^{[2k]}$ is a $2k$-fold edge, and its Martin invariant reduces to $(2k-1)!!$ times that of the rose by \eqref{eq:Martin-recursion} and \eqref{eq:trans(v)}. For the rose itself, use \eqref{eq:Martin-rose}.
At three vertices, the only $2k$-regular graph without self-loops is the $k$-fold triangle, and expanding it at a vertex shows that, as claimed in part 2.\ of \cref{def:Martin-intro},
\begin{equation}\label{eq:Martin-triangle}
    \Martin(C_3^{[k]})
    = k!\cdot \Martin(C_2^{[k]})
    = 1.
\end{equation}
Note that only $k!$ transitions produce a dipole; all other transitions produce self-loops and can hence be discarded for the computation of the Martin invariant.
\begin{remark}
Our normalization \eqref{eq:Martin-triangle} simplifies the 3-vertex cut product of \cref{prop:3vertex-cut}. This maximizes the analogy between $\Martin$ and the Feynman period $\Period$, because the latter also fulfils this product without extra factors \cite[Theorem~2.1]{Schnetz:Census}. The downside is that we move outside of integers for the Martin invariants of $C_1^{[k]}$ and $C_2^{[k]}$ (though all graphs with at least 3 vertices have integer Martin invariant) and have an explicit factor $k!$ in the $2k$-edge cut product \eqref{eq:Martin-edge-product}, which however has no direct analogue for Feynman periods, because graphs with such a cut have divergent Feynman integrals.
\end{remark}
The Martin invariant was introduced for 4-regular graphs in \cite{BouchetGhier:BetaIso4}. In that paper, it is called $\beta_4$, and a different normalization is used. If $G$ has $n$ vertices, then
\begin{equation}
    \Martin(G)=\beta_4(G)\cdot 2^{n-3}.
\end{equation}
The normalization of $\beta_4$ ensures that the totally decomposable graphs (see \cref{prop:tot-decomp}) are precisely the graphs with $\beta_4(G)=1$, in analogy to how the series-parallel graphs are precisely the graphs with $\beta(G)=1$ for Crapo's $\beta$ invariant \cite[Proposition~8]{Crapo:HigherInvariant}. But in contrast to $\Martin(G)$, for most graphs the number $\beta_4(G)$ is not an integer.

Apart from the paper \cite{BouchetGhier:BetaIso4}, the only other result on the Martin invariant that we could find is \cite[Theorem~4.6]{EllisMonaghan:DiffMartin}. For 4-regular graphs, it states that
\begin{equation}\label{eq:Monaghan}
    6\Martin(G) = \MartinPol'(G,0) = \frac{\CPP'(G,-2)}{-2}
    = \sum_C r_1(G\setminus C) \cdot (-2)^{\kappa(C)-1}
\end{equation}
is a sum over all subsets $C$ of edges such that every vertex of $G$ has degree 0 or 2 in $C$. In other words, $C$ is a vertex-disjoint union of cycles. We write $\kappa(C)$ for the number of these cycles, and $r_1(G)=\MartinPol(G,2)=\CPP'(G,0)$ denotes the number of Eulerian circuits.

\subsection{Vector model}\label{sec:O(n)}
In theoretical physics, the Martin polynomial appears as the \emph{symmetry factor} $S_G$ in the $O(N)$ vector model. For details, see \cite[Chapter~6]{KleinertSchulteFrohlinde:CriticalPhi4}. In brief, this is the theory of $N$ real scalar fields $\phi=(\phi_1,\ldots,\phi_N)$ with Euclidean Lagrangian density
\begin{equation*}
    \mathscr{L} = \frac{1}{2} 
    \sum_{i=1}^N\left(
        \norm{\nabla\phi_i}^2
        +m^2 \phi_i^2 
    \right)
    +\frac{g}{4!} \left( 
        \norm{\phi}^2
    \right)^2.
\end{equation*}
The first term describes free scalars with equal mass $m$. The second term with the coupling constant $g$ encodes their interaction. This density is invariant under orthogonal transformations $\phi\mapsto T\phi$ with $T\in O(N)$.
In terms of the symmetric tensor
\begin{equation}\label{eq:O(n)-tensor}
    \lambda_{ijkl} = \frac{1}{3}\left(
        \delta_{ij}\delta_{kl} + \delta_{ik}\delta_{jl} + \delta_{il}\delta_{jk}
    \right),
\end{equation}
the quartic interaction can be written as
\begin{equation*}
    \left(\norm{\phi}^2\right)^2 
    = (\phi_1^2+\ldots+\phi_N^2)^2
    = \sum_{i,j=1}^N \phi_i^2\phi_j^2
    = \sum_{i,j,k,l=1}^N \lambda_{ijkl} \phi_i\phi_j\phi_k\phi_l
    .
\end{equation*}

The perturbation theory of the $O(N)$ model thus differs from $\phi^4$ theory ($N=1$) only through a prefactor $S_G$ for each Feynman graph $G$. This factor counts all possible assignments of the field labels $\set{1,\ldots,N}$ to each edge compatibly with the expression for the interaction. Let $V$ and $E$ denote the sets of vertices and edges of $G$, and write $e_v^1,\ldots,e_v^4$ for the four edges at a vertex. Then
\begin{equation}\label{eq:phi4-symmetryfac}
    S_G(N)=\sum_{c\colon E\rightarrow \set{1,\ldots,N}} 
    \prod_{v\in V}
    \lambda_{c(e_v^1),c(e_v^2),c(e_v^3),c(e_v^4)}
\end{equation}
for every 4-regular (``vacuum'') graph $G$. Expanding each tensor in this product into the three terms in \eqref{eq:O(n)-tensor} corresponds to a sum over transition systems. The number of parts of the corresponding circuit partition gives the number of freely choosable field labels.

For a 4-regular graph with $n$ vertices, we can therefore identify
\begin{equation}\label{eq:symmetry-from-CPP}
    S_G(N) = 3^{-n}\cdot \CPP(G,N).
\end{equation}
Corrections to the coupling constant are determined by graphs of the form $G\setminus v$, which have 4 free (``external'') half-edges. Their symmetry factor is \cite[eq.~(6.77)]{KleinertSchulteFrohlinde:CriticalPhi4}
\begin{equation*}
    S_{G\setminus v}(N) 
    = \frac{3 S_G(N)}{N(N+2)}
    = \frac{\MartinPol(G,N+2)}{3^{n-1}(N+2)}.
\end{equation*}
Since $\MartinPol(G,x)$ is a multiple of $x$ by \eqref{eq:Martin-rose}, it follows from \eqref{eq:circuit-pol} that $\CPP(G,N)$ is a multiple of $N(N+2)$. Therefore, $S_{G\setminus v}(N)$ is a polynomial in $N$. One can therefore formally consider arbitrary values for $N$, even though the $O(N)$ vector model was defined above only for positive integers $N$. Upon setting $N=-2$, we obtain the Martin invariant:
\begin{equation}\label{eq:symmetry-from-Martin}
    S_{G\setminus v}(-2) 
    = 3^{1-n} \cdot \MartinPol'(G,0) 
    = 2\cdot 3^{2-n} \cdot \Martin(G)
    .
\end{equation}

This was how we discovered the connection between the Martin invariant and Feynman symmetries: analyzing the symmetry factors for graphs with up to 13 vertices, we found that $N=-2$ is the only value for which identities like twist and duality hold.

Strikingly, the formal value $N=-2$ in the $O(N)$ vector model also describes a meaningful physical system. It was recently discovered \cite{WieseFedorenko:FieldErasedRW,ShapiraWiese:ExactLERW} that the formal ``$O(-2)$'' model describes loop erased random walks (LERW). The value $N=-2$ is therefore distinguished not only by the additional mathematical properties of the Martin invariant (the full Martin polynomial does not fulfil identities like twist or duality), but it also corresponds to an exceptional field theory.

\section{Connectivity}
We say that a graph $G$ is $d$-regular if every vertex of $G$ has degree $d$. Except for \cref{thm:decomp-unique}, we always assume that the degree is even, and hence denote it as $2k$.
\begin{figure}
    \centering
    $G=\quad \Graph[0.8]{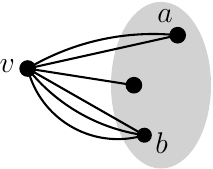}$
    \qquad$\mapsto$\qquad
    $G'=\quad\Graph[0.8]{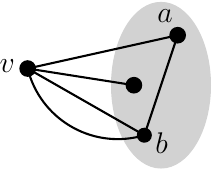}$
    \caption{Splitting off a pair $(va,vb)$ of edges.}%
    \label{fig:splitting-off}%
\end{figure}
\begin{lemma}\label{lem:not-2kcon-zero}
    If $G$ is $2k$-regular but not $2k$-edge connected, then $\Martin(G)=0$.
\end{lemma}
\begin{proof}
    Consider a minimal edge cut $C$ of $G$, with size $\abs{C}=2s<2k$. Let $G_1$ and $G_2$ denote the two sides of the cut, so that $G\setminus C=G_1\sqcup G_2$. If $G_1$ consists of only one vertex, then this vertex has $k-s>0$ self-loops, and therefore $\Martin(G)=0$. Otherwise, if $G_1$ has several vertices, pick one and call it $v$. For every transition $\tau$ at $v$, the edges of $C$ in $G_{\tau}$ define a (not necessarily minimal) edge-cut of $G_{\tau}$, hence $G_{\tau}$ is at most $2s$-edge connected. By induction over the total number of vertices, we may assume that $\Martin(G_{\tau})=0$, and therefore conclude the induction step that $\Martin(G)=0$ as well, by \eqref{eq:Martin-recursion}.
\end{proof}
\begin{theorem}\label{thm:splitting-off-transition}
    If $G$ is $2k$-regular and $2k$-edge connected, and $v$ denotes any vertex, then there exists a transition $\tau\in\trans(v)$ such that $G_{\tau}$ is $2k$-edge connected.
\end{theorem}
\begin{proof}
    Let $\lambda_G(x,y)$ denote the maximal number of edge-disjoint walks in $G$, with endpoints at vertices $x\neq y$. By Menger's theorem, the $2k$-edge connectedness of $G$ implies that $\lambda_G(x,y)\geq 2k$ for all pairs $x\neq y$. In fact, since $G$ is $2k$-regular, equality holds. Lov\'asz' \emph{vertex splitting-off lemma} \cite[Theorem~1]{Lovasz:ConnEulerian} shows that at any vertex $v$ in any Eulerian\footnote{While not needed here, we note that vertex splitting-off generalizes to non-Eulerian graphs \cite{Mader:ReductionEC,Frank:Mader}.} graph $G$, one can find two different edges $va$ and $vb$, such that the graph $G'=G-va-vb+ab$ (see \cref{fig:splitting-off}) has the following property: for every pair $x\neq y$ of vertices different from $v$, one has
    \begin{equation*}
        \lambda_G(x,y)=\lambda_{G'}(x,y).
    \end{equation*}
    Iterating this lemma $k$ times produces a transition $\tau$ with $\lambda_{G_{\tau}}(x,y)=\lambda_G(x,y)=2k$, therefore $G_{\tau}$ is $2k$-edge connected.
\end{proof}
\begin{corollary}\label{cor:2kcon-nonzero}
    If $G$ is $2k$-regular and $2k$-edge connected, then $\Martin(G)>0$.
\end{corollary}
\begin{proof}
By \eqref{eq:Martin-recursion}, we have $\Martin(G)\geq \Martin(G_{\tau})$ for any transition of $G$. Choosing $\tau$ as in \cref{thm:splitting-off-transition}, we can ensure that $G_{\tau}$ is $2k$-edge connected. Induction over the number of vertices thus reduces the proof to the case where $G$ is a rose; then $\Martin(G)>0$ by \eqref{eq:Martin-rose-dipole}.
\end{proof}
Together with \cref{lem:not-2kcon-zero}, \cref{cor:2kcon-nonzero} proves that for a graph $G$ that is $2k$-regular, $\Martin(G)>0$ if and only if $G$ is $2k$-edge connected. This settles parts 1 and 2 of \cref{prop:edge-cuts}. We now show part 3:
\begin{lemma}\label{lem:2k-cut-product}
    Let $G$ be $2k$-regular and $C$ a $2k$-edge cut of $G$. Let $G_1'\sqcup G_2'=G\setminus C$ denote the two sides of the cut, and write $G_i$ for the graph obtained from $G_i'$ by adding a vertex connected to the ends of $C$ in $G_i'$, see \eqref{eq:Martin-edge-product}. Then $\Martin(G)=k!\cdot \Martin(G_1)\cdot\Martin(G_2)$.
\end{lemma}
\begin{proof}
    We proceed by induction over the number of vertices of $G_2'$. If $G_2'$ has only one vertex, then $G_2$ is a $2k$-fold edge and $G_1=G$. The claim is thus trivial, $\Martin(G)=\Martin(G)$, because $\Martin(G_2)=\Martin(K_2^{[2k]})=1/(k!)$.
    If $G_2'$ has more than one vertex, pick any vertex $v$ in $G_2'$. Using the induction hypothesis, expanding at $v$ yields
    \begin{equation*}
        \Martin(G)=\sum_{\tau\in\trans(v)} \Martin(G_{\tau})
        = k! \sum_{\tau\in\trans(v)} \Martin((G_{\tau})_1) \cdot \Martin((G_{\tau})_2)
        ,
    \end{equation*}
    because the edges of $C$ define a $2k$-edge cut in every $G_{\tau}$. If $\tau$ pairs any edges in $C$ with each other, then $C$ induces in $G_{\tau}$ an edge-cut of size less than $2k$, hence such transitions can be dropped from the sum by \cref{lem:not-2kcon-zero}. The remaining transitions commute with the cut: the side $(G_{\tau})_1=G_1$ is the same for all $\tau$, and $(G_{\tau})_2=(G_2)_{\tau}$. Hence the right-hand side sums to the claim. Note that the dropped $\tau$ are precisely those that produce a self-loop in $(G_2)_{\tau}$ at the extra vertex, hence their absence causes no error to $\Martin(G_2)$.
\end{proof}

\subsection{Decompositions and lower bounds}\label{sec decomp and bound}
We saw above that $\Martin(G)>0$ if $G$ is $2k$-regular and $2k$-edge connected. In this subsection, we give a combinatorial description of all such graphs which, at a given number of vertices, minimize the Martin invariant. This is closely related to factorizations with respect to the $2k$-edge cut product from \cref{lem:2k-cut-product}, for which we prove uniqueness (\cref{thm:decomp-unique}).

By repeated replacements $G\mapsto G_1,G_2$ along $2k$-edge cuts, we can decompose every $2k$-edge connected $2k$-regular graph with $\geq 3$ vertices into a sequence of graphs $G_1,\ldots,G_p$ until each graph in the sequence is cyclically $(2k+2)$-connected and has at least 3 vertices. Following \cite{BouchetGhier:BetaIso4,FouquetThuillier:decom3cubic}, we call such a sequence a \emph{decomposition} of $G$. Iterating \eqref{eq:Martin-edge-product}, note that
\begin{equation}\label{eq:Martin-2k-decomposition}
    \Martin(G) = (k!)^{p-1} \cdot \Martin(G_1)\cdots\Martin(G_p).
\end{equation}
We call $G$ \emph{totally decomposable} \cite{BouchetGhier:BetaIso4} if it has a decomposition such that all $G_i$ have only 3 vertices, see \cref{fig:total-decomposition}. By \eqref{eq:Martin-2k-decomposition}, such a graph, with $n$ vertices, has $\Martin(G)=(k!)^{n-3}$. By induction one sees easily that any totally decomposable graph has at least two edges with multiplicity $k$. Decomposing by cutting off such an edge at each step we end with the $k$-fold triangle. Therefore, reversing this we get that totally decomposable graphs can be constructed from the $k$-fold triangle by repeatedly replacing a vertex with a $k$-fold edge.
\begin{figure}
    \centering
    $\Graph[0.5]{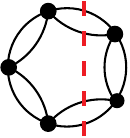} \quad
    \mapsto
    \quad\Graph[0.5]{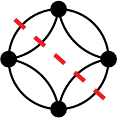}, \Graph[0.4]{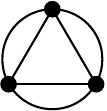} \quad
    \mapsto 
    \quad\Graph[0.4]{C32}, \Graph[0.4]{C32}, \Graph[0.4]{C32}\quad
    \reflectbox{$\mapsto$} 
    \quad\Graph[0.4]{C32}, \Graph[0.5]{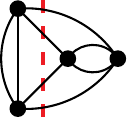}\quad
    \reflectbox{$\mapsto$} 
    \quad \Graph[0.5]{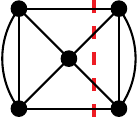}
    $
    \caption{Decomposition of two totally decomposable graphs into 3 copies of $C_3^{[2]}$.}%
    \label{fig:total-decomposition}%
\end{figure}

For example, cutting off $i\geq 2$ consecutive vertices from a $k$-fold cycle,
\begin{equation*}
    C_n^{[k]} \mapsto C_{i+1}^{[k]}, C_{n-i+1}^{[k]}
\end{equation*}
terminates in a decomposition of $C_n^{[k]}$ into $n-2$ copies of the $k$-fold triangle $C_3^{[k]}=K_3^{[k]}$.

The totally decomposable graphs minimize the Martin invariant:
\begin{theorem}\label{prop:tot-decomp}
    A $2k$-regular, $2k$-edge connected graph $G$ with $n\geq 2$ vertices has $M(G) \geq (k!)^{n-3}$, and equality holds if and only if $G$ is totally decomposable.
\end{theorem}
The 4-regular case of this result was known previously \cite[Corollary~5.6]{BouchetGhier:BetaIso4}. To prove \cref{prop:tot-decomp}, we first establish the lower bound, which strengthens \cref{cor:2kcon-nonzero}:
\begin{proposition}\label{thm:martin-min}
    Suppose that the graph $G$ is $2k$-regular, $2k$-edge connected, and that it has $n\geq 2$ vertices. Then $\Martin(G)\geq (k!)^{n-3}$.
\end{proposition}
\begin{proof}
    As in the proof of \cref{thm:splitting-off-transition}, let $\lambda_G(x,y)$ denote the maximum number of edge-disjoint paths between vertices $x$ and $y\neq x$. Given $G$ and fixing a vertex $v$, we call a pair $(va,vb)$ of edges at $v$ \emph{admissible} if $\lambda_{G'}(x,y)\geq 2k$ for all $x\neq y$ with $x,y\neq v$, in the graph $G'=G-va-vb+ab$.

    It was shown in \cite[Theorem~2.12]{BangJensenGabowJordanSzigeti:EdgePartition} that if $v$ has even degree $d$ and $\lambda_G(x,y)\geq 2k$ for all $x\neq y$ different from $v$, then for any edge $va$, there are at least $d/2$ other edges $vb$ such that the pair $(va,vb)$ is admissible. Applying this argument to $G'$, where $v$ has degree $2k-2$, and iterating, we conclude that there are at least
    \begin{equation*}
        \frac{2k}{2}\cdot \frac{2k-2}{2} \cdots \frac{2}{2} = k!
    \end{equation*}
    transitions $\tau\in\trans(G)$ such that $G_{\tau}$ is $2k$-edge connected. So if $c_n$ denotes the minimum value of $\Martin(G)$ over all $2k$-regular, $2k$-edge connected graphs $G$ with $n$ vertices, then taking only those transitions into account, the sum \eqref{eq:Martin-recursion} shows that
    \begin{equation*}
        c_n \geq k! \cdot c_{n-1}.
    \end{equation*}
    We conclude that $c_n\geq (k!)^{n-3}$ by induction over $n$ and $c_3=1$.
\end{proof}

To finish the proof of \cref{prop:tot-decomp}, recall that $\Martin(G)=(k!)^{n-3}$ if $G$ is totally decomposable. If $G$ is not totally decomposable, then any decomposition $G_1,\ldots,G_p$ has at least one factor with at least $4$ vertices. Let $G_1$ be such a factor, and let $n_i$ denote the number of vertices in $G_i$. Then by \eqref{eq:Martin-2k-decomposition} and \cref{thm:martin-min},
\begin{equation*}
    \Martin(G) \geq (k!)^{p-1} \cdot \Martin(G_1) \cdot (k!)^{n_2-3}\cdots (k!)^{n_p-3}
    = (k!)^{n-3} \cdot \frac{\Martin(G_1)}{(k!)^{n_1-3}}
\end{equation*}
because $n_1+\ldots+n_p=n+2(p-1)$. Therefore, $\Martin(G)>(k!)^{n-3}$ follows from our next lemma.
\begin{lemma}\label{lem:cyclic-bigger}
    If $G$ is $2k$-regular and cyclically $2k+2$ connected, with $n\geq 4$ vertices, then $\Martin(G)>(k!)^{n-3}$.
\end{lemma}
\begin{proof}
    Pick any vertex $v$ in $G$. Consider two edges $va$, $vb$ at $v$ with different endpoints $a\neq b$. We first show that such a pair $(va,vb)$ is admissible:

    Suppose that $(va,vb)$ is not admissible. Then $G'=G-va-vb+ab$ has an edge cut $C'$ of size $\abs{C'}<2k$, with at least one vertex $x\neq v$ and at least one vertex $y\neq v$ on each side of $C'$. Let $V(G)=X\sqcup Y$ denote the vertex bipartition corresponding to $C'$, with $v,x\in X$ and $y\in Y$. We show that the cut $C$ of $G$ with this bipartition contradicts the cyclic $2k+2$ connectivity of $G$:
    \begin{itemize}
        \item 
    If $a,b \in X$, then $C=C'$ has size $\abs{C}=\abs{C'}<2k$.
    \item 
    If $a\in X$ and $b\in Y$, then $ab\in C'$ and $C=C'-ab+vb$ has size $\abs{C}=\abs{C'}<2k$.
    \item 
    If $a,b\in Y$ (see \cref{fig:XYab}), then $C=C'+va+vb$ has size $\abs{C}=\abs{C'}+2\leq 2k+1$. Since all vertices have even degree, $\abs{C}$ must be even, thus $\abs{C}\leq 2k$. Note that neither side of $C$ is a single vertex: $\abs{X}\geq\abs{\set{v,x}}=2$ and $\abs{Y}\geq\abs{\set{a,b}}=2$.
    \end{itemize}
    The third point directly contradicts cyclic $2k+2$ connectivity while the first two either contradict $2k$ regularity if one side of the cut is a single vertex and cyclic $2k+2$ connectivity otherwise.
\begin{figure}
    \centering
    $G=\Graph[0.9]{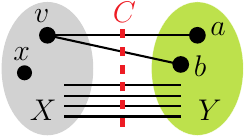}$
    \qquad
    $G'=\Graph[0.9]{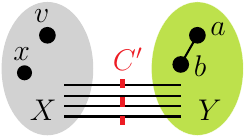}$
    \caption{Illustration for the proof of \cref{lem:cyclic-bigger}: A non-admissible pair $(va,vb)$ implies the existence of a forbidden cut $C$ in $G$.}%
    \label{fig:XYab}%
\end{figure}

    Now pick any neighbour $a$ of $v$. The number $s$ of edges between $a$ and $v$ is at most $k-1$, as otherwise, the subgraph induced by $a$ and $v$ would constitute a cut of size $2\cdot 2k-2s\leq 2k$ with at least 2 vertices on each side since $G$ has $n\geq 4$ vertices. Therefore, at least $2k-(k-1)=k+1$ edges at $v$ have endpoints $b\neq a$. Since such pairs $(va,vb)$ are admissible as we showed above, we find at least $k+1$ admissible pairs containing $va$. After splitting off such a pair, we invoke \cite[Theorem~2.12]{BangJensenGabowJordanSzigeti:EdgePartition} repeatedly, as in the proof of \cref{thm:martin-min}, to conclude that there are at least
    \begin{equation*}
        (k+1)\cdot (k-1)!
    \end{equation*}
    transitions $\tau\in\trans(v)$ such that $G_{\tau}$ is $2k$-edge connected. Considering only those summands and using \eqref{eq:Martin-recursion}, we learn that $\Martin(G)\geq (k+1)\cdot (k-1)! \cdot (k!)^{n-4} > (k!)^{n-3}$.
\end{proof}
According to our definition earlier, a \emph{decomposition} of a $d=2k$-regular graph is obtained after a maximal sequence of $d$-edge cuts. Such a sequence of cuts is rarely unique. For example, to decompose a k-fold cycle on 6 vertices, we can choose to cut off either 2 or 3 vertices in the first step, which leads to two different sequences of cuts:
\begin{itemize}
    \item[a)] $C_6^{[k]} \mapsto C_5^{[k]} C_3^{[k]} \mapsto C_4^{[k]}C_3^{[k]}C_3^{[k]} \mapsto C_3^{[k]}C_3^{[k]}C_3^{[k]}C_3^{[k]}$
    \item[b)] $C_6^{[k]} \mapsto C_4^{[k]} C_4^{[k]} \mapsto C_4^{[k]}C_3^{[k]}C_3^{[k]} \mapsto C_3^{[k]}C_3^{[k]}C_3^{[k]}C_3^{[k]}$
\end{itemize}
However, note that the resulting decomposition is the same: we obtain four times the $k$-fold triangle, even though we constructed those triangles by different sequences of cuts.

We close this section by showing that decompositions $G\mapsto G_1\cdots G_p$ are always unique. Here, \emph{unique} is meant in the sense of multisets (lists up to reordering) of isomorphism classes of graphs.\footnote{
    In the cubic case, Wormald's \cite[Theorem~3.1]{Wormald:Cyclically4cubic} shows uniqueness in a stronger sense, accounting also for how the $G_i$ are arranged in $G$.%
}
Our proof works for arbitrary degrees. For the cubic case $d=3$, a proof of uniqueness was given in \cite[Theorem~3.5]{FouquetThuillier:decom3cubic}.
\begin{proposition}\label{thm:decomp-unique}
    Decompositions of $d$-edge connected graphs are unique.
\end{proposition}
\begin{proof}
    Let two decompositions $A=A_1\cdots A_p$ and $B=B_1\cdots B_{q}$ of $G$ be given. We call an edge cut \emph{trivial} if one side of it consists of a single vertex. If $G$ does not have any non-trivial $d$-edge cut, then we must have $p=q=1$ and $A_1=B_1=G$ is indeed unique.
    
    If $G$ does have a non-trivial $d$-edge cut, then $G$ is not cyclically $d+1$ connected, hence we must have $p,q\geq 2$.
    Let $C$ and $D$ denote the $d$-edge cuts used in the first step of a construction of the decompositions $A$ and $B$, respectively, such that
    \begin{equation*}
        A=A'A'' \qquad\text{and}\qquad B=B'B''
    \end{equation*}
    where $A'$ and $A''$ are decompositions of the two sides of $C$, and similarly $B',B''$ for $D$. By induction over the size of $G$, we may use that $A',A'',B',B''$ are unique.
    
    The two cuts determine a partition of $G\setminus(C\cup D)=G_1\sqcup G_2\sqcup G_3\sqcup G_4$ into four subgraphs, such that $C=[G_1\cup G_4, G_2\cup G_3]$ and $D=[G_1\cup G_2,G_3 \cup G_4]$, where $[X,Y]$ denotes all edges of $G$ with one end in $X$ and the other end in $Y$.
    Let $m_{ij}=\abs{[G_i,G_j]}$ denote the number of edges in $G$ between $G_i$ and $G_j$. Note that
    \begin{equation*}\label{eq:uniqueness-cut-sizes}\tag{$\sharp$}
        m_{12}+m_{13}+m_{24}+m_{34}=\abs{C}=d
        \quad\text{and}\quad
        m_{13}+m_{14}+m_{23}+m_{24}=\abs{D}=d.
    \end{equation*}

    If two of the subgraphs $G_i$ are empty, we have $C=D$ and therefore $A'=B'$ and $A''=B''$ by induction, hence $A=B$.
    \begin{figure}
        \centering
        \begin{tabular}{cccccc}
        $G$ & $G_{12}'$ & $G_{3}'$ & $G_{23}'$ & $G_1'$ & $G_2'$ \\
        \midrule
        $\Graph[0.6]{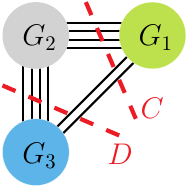}$ &
        $\Graph[0.6]{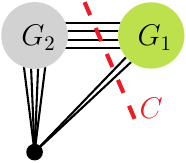}$ &
        $\Graph[0.6]{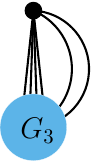}$ &
        $\Graph[0.6]{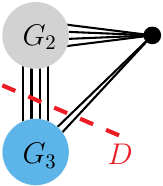}$ &
        $\Graph[0.6]{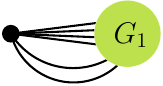}$ &
        $\Graph[0.6]{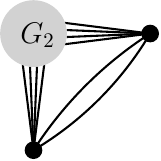}$ \\
        \end{tabular}%
        \caption{Graphs that arise in the decomposition for two cuts $C$ and $D$ with $G_4=\emptyset$.}%
        \label{fig:uniqueness-sequential}%
    \end{figure}

    If only one of the subgraphs is empty, say $G_4=\emptyset$, then $A'$ and $A''$ are by definition decompositions of the graphs $G_{1}'$ and $G_{23}'$ obtained from $G_1$ and $G_2\cup G_3 \cup [G_2,G_3]$, respectively, by reconnecting the edges $C$ to an extra vertex. Similarly, $B'$ and $B''$ are decompositions of the graphs $G_3',G_{12}'$ obtained from $G_3$ and $G_1\cup G_2\cup[G_1,G_2]$, see \cref{fig:uniqueness-sequential}. Note that $C$ defines a cut $G_{12}'\mapsto G_1'G_2'$ and $D$ defines a cut $G_{23}'\mapsto G_3'G_2'$, such that
    \begin{equation*}
        A''=OB'
        \qquad\text{and}\qquad
        B''=OA'
    \end{equation*}
    in terms of a decomposition $O$ of the graph $G_2'$ obtained from $G$ by replacing $G_1$ and $G_3$ by single vertices. Here we exploited the uniqueness of $A''$ and $B''$ (induction hypothesis), which allows us to compute them using any cut we like. We conclude that $A'A''=A'OB'=B'B''$, up to order.

    Finally, consider the case when all four subgraphs $G_i$ are non-empty. Since $G$ is $d$-edge connected, we have the lower bounds
    \begin{align*}
        m_{12}+m_{13}+m_{14} &\geq d, &
        m_{12}+m_{23}+m_{24} &\geq d, \\
        m_{13}+m_{23}+m_{34} &\geq d, &
        m_{14}+m_{24}+m_{34} &\geq d.
    \end{align*}
    Combined with \eqref{eq:uniqueness-cut-sizes}, this system has a unique solution, given by $m_{13}=m_{24}=0$ and $m_{12}=m_{14}=m_{23}=m_{34}=d/2$. Therefore, $d$ must be even, and the edges $[G_1,G_{2}\cup G_{4}]$ constitute a cut of size $m_{12}+m_{14}=d$ that separates $G_1$ from the rest of $G$. Analogously, the other $G_i$ each can be cut out along $d$ edges. Let $G\mapsto G_{14}'G_{23}'$ denote the graphs created by the cut $C$, so that $A'$ and $A''$ are decompositions of $G_{14}'$ and $G_{23}'$, respectively. Here, $G_{ij}$ is obtained from $G_i\cup G_j\cup[G_i,G_j]$ by joining the cut edges to a new vertex (see \cref{fig:uniqueness-crossing}). Using the cuts $[G_2,G_1\cup G_3]$ and $[G_3,G_1\cup G_4]$ of $G_{23}'$, we can decompose
\begin{figure}
    \centering
    \begin{tabular}{ccccccc}
    $G$ & $G_{23}'$ & $G_{14}'$ & $G_1'$ & $G_2'$ & $G_3'$ & $G_4'$
    \\
    \midrule
    $\Graph[0.65]{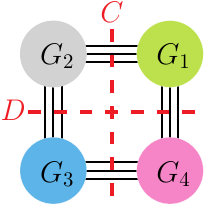}$ &
    $\Graph[0.65]{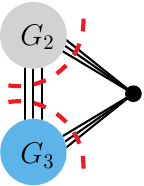}$ &
    $\Graph[0.65]{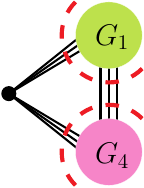}$ &
    $\Graph[0.65]{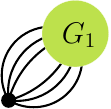}$ &
    $\Graph[0.65]{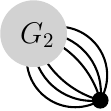}$ &
    $\Graph[0.65]{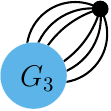}$ &
    $\Graph[0.65]{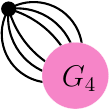}$
    \\
    \end{tabular}%
    \caption{Graphs that arise in the decomposition of two cuts $C,D$ with all $G_i\neq\emptyset$.}%
    \label{fig:uniqueness-crossing}%
\end{figure}
    \begin{equation*}
        G_{23}'\mapsto G_2' G_3' C_3^{[d/2]}
        \quad\text{and thus find}\quad
        A'=O_2 O_3 C_3^{[d/2]}
    \end{equation*}
    where $O_i$ denotes a decomposition of the graph $G_i'$ obtained from $G_i$ (see \cref{fig:uniqueness-crossing}). Again, we use here inductively that the decomposition of $G_{23}'$ is unique. In the same manner, we find $A''=O_1 O_4 C_3^{[d/2]}$ and therefore
    \begin{equation*}
        A'A''=O_1 O_2 O_3 O_4 C_3^{[d/2]} C_3^{[d/2]}.
    \end{equation*}
    Clearly, applying the same procedure to $D$, i.e.\ decomposing $G\mapsto G_{12}' G_{34}'$ and then $G_{12}'\mapsto G_1' G_2' C_3^{[d/2]}$ and $G_{34}'\mapsto G_3' G_4' C_3^{[d/2]}$ we obtain the same result for $B'B''$.
\end{proof}

\subsection{Product and twist}\label{sec:product+twist}

For the \emph{product} identity \eqref{eq:martin vertex product}, consider a $3$-vertex cut $S=\set{v_1,v_2,v_3}$ of a $2k$-regular graph $G$. More precisely, fix a bipartition $E=E_1\sqcup E_2$ of the edges of $G$ such that the only vertices shared by $E_1$ and $E_2$ are in $S$.\footnote{A part $E_i$ is not required to touch any vertices other than $S$.} Let $d_i$ denote the number of edges at $v_i$ that belong to $E_1$, then the other $2k-d_i$ edges at $v_i$ belong to $E_2$. Set
\begin{equation}\label{eq:3cut-edgecounts}
    m_{12} = \frac{d_1+d_2-d_3}{2},\quad
    m_{13} = \frac{d_1+d_3-d_2}{2},\quad
    m_{23} = \frac{d_2+d_3-d_1}{2}.
\end{equation}
These are integers, because $d_1+d_2+d_3$ is the size of an edge cut and hence even. Note that in any graph where every vertex has even degree, all edge cuts are even.
\begin{lemma}\label{lem:positive ns}
    If $G$ is $2k$-edge connected, then $m_{12},m_{13},m_{23} \geq 0$.
\end{lemma}
\begin{proof}
    Suppose that $m_{12}<0$. Then $d_1+d_2<d_3$, hence $d_1+d_2+(2k-d_3)<2k$. So the edges at $v_1,v_2$ in $E_1$ together with the edges at $v_3$ in $E_2$ form a cut of $G$ with size less than $2k$, contradicting the assumption on $G$.
\end{proof}

Now we state and prove the product identity for the Martin invariant introduced in \eqref{eq:martin vertex product}.
\begin{proposition}\label{prop:3vertex-cut}
    Let $G$ be $2k$-regular and $2k$-edge connected, with a 3-vertex cut. Then the two sides of the cut can be turned uniquely into $2k$-regular graphs $G_1$ and $G_2$, by adding edges (but no self-loops) between the cut vertices. Furthermore, we have 
    \begin{equation*}
    \Martin(G)=\Martin(G_1)\cdot\Martin(G_2).
    \end{equation*}
\end{proposition}
\begin{proof}
With notation as above, \cref{lem:positive ns} implies that we can construct a $2k$-regular graph $G_2$ from $E_2$ by adding $m_{ij}$ edges connecting $v_i$ and $v_j$, for each of the three pairs $ij=12,13,23$. By a symmetric construction we get $G_1$ from $E_1$. This proves the first claim of \cref{prop:3vertex-cut}. In fact, \eqref{eq:3cut-edgecounts} is the unique solution to the constraints $m_{ij}+m_{ik}+d_i=2k$, so this is the only way to add non-self loop edges to $E_2$ that leaves a $2k$-regular graph.

To complete the proof of \cref{prop:3vertex-cut}, we exploit that the identity \eqref{eq:martin vertex product}
is linear in each factor. More precisely, for any vertex $v\notin S$ that belongs to $G_1$, we can apply transitions $\tau\in\trans(v)$ at $v$ on both sides: in $G$ and in $G_1$. Since $\tau$ only glues edges from $E_1$, it leaves $G_2$ unchanged and the cut $S$ remains a cut of $G_{\tau}$, such that $(G_1)_{\tau}=(G_{\tau})_1$ and $(G_{\tau})_2=G_2$.
The expansion \eqref{eq:Martin-recursion} at $v$ therefore reduces the claim to the statement that $\Martin(G_{\tau})=\Martin((G_{\tau})_1)\Martin((G_{\tau})_2)$. An induction over the number of vertices therefore reduces the claim to trivial base cases:
\begin{itemize}
    \item If $G$ has a self-loop, then this self-loop is also a self-loop in $G_1$ or $G_2$, so the identity holds trivially because both sides are zero.
    \item When only 3 vertices are left, namely the vertices $S$, and $G$ has no self-loops, then all 3 graphs $G$, $G_1$, and $G_2$, are necessarily identical to the $k$-fold triangle. The identity then holds due to our normalization $\Martin(C_3^{[k]})=1$.
\end{itemize}
\end{proof}

With the same induction principle, we now prove the twist identity \eqref{eq:twist}. 
Suppose we are given a graph $G$ with a $4$-vertex cut $S=\set{v_1,v_2,v_3,v_4}$ and a bipartition $E=E_1\sqcup E_2$ of the edges, such that all vertices shared between $E_1$ and $E_2$ are contained in $S$. Let $d_i$ denote the number of edges at $v_i$ that belong to $E_1$. Schnetz introduced the \emph{twist} operation \cite[\S 2.6]{Schnetz:Census}, which is illustrated in \eqref{eq:twist}.
\begin{definition}
    If $d_1=d_2$ and $d_3=d_4$, then the \emph{twist} of $G$ is the graph $G'$ obtained from $G$ by replacing every edge $e\in E_1$ with an endpoint in $S$, $e=uv_i$, by $uv_{\sigma(i)}$, where $\sigma=(1\ 2)(3\ 4)$ is the double transposition that swaps $v_1\leftrightarrow v_2$ and $v_3\leftrightarrow v_4$.
\end{definition}
At a 4-vertex cut with $d_1=d_2=d_3=d_4$, one can construct 3 twists, by relabeling $v_i$ or, equivalently, using the other double transpositions $\sigma=(1\ 3)(2\ 4)$ and $\sigma=(1\ 4)(2\ 3)$. If we fix $S$ and $\sigma$, then the twist is an involution: $G$ is the twist of $G'$.
Now we state and prove the twist identity for the Martin invariant introduced in \eqref{eq:twist}.
\begin{proposition}\label{prop:twist}
    If two regular graphs of even degree are obtained from each other by a double transposition on one side of a 4-vertex cut, then their Martin invariants agree:
    \begin{equation*}
        \Martin\Bigg(\ \Graph[0.3]{4cut}\ \Bigg)=\Martin\Bigg(\ \Graph[0.3]{twist}\ \Bigg).
    \end{equation*}
\end{proposition}
\begin{proof}
Let $G, S, E, d_i$ be as above.
As in the proof of \cref{prop:3vertex-cut} we can use the expansion \eqref{eq:Martin-recursion} into transitions at any vertex $v\notin S$, to reduce the twist identity
\begin{equation*}
    \Martin(G)=\Martin(G')
\end{equation*}
to the case where $G$ has no vertices other than $S$. If $G$ contains a self-loop, so does $G'$, and the identity becomes trivial since both sides are zero. Hence we may assume that there are no self-loops in $G$. In this situation, $E_1$ is completely determined by the numbers $m_{ij}$ of edges in $E_1$ that connect $v_i$ with $v_j$. Due to the constraints
\begin{align*}
    m_{12}+m_{13}+m_{14}&=d_1=d_2=m_{12}+m_{23}+m_{24}
    \quad\text{and}\\
    m_{13}+m_{23}+m_{34}&=d_3=d_4=m_{14}+m_{24}+m_{34},
\end{align*}
we must have $m_{13}=m_{24}$ and $m_{14}=m_{23}$. This shows that $G=G'$, because the twist changes $E_1$ into $E_1'$ with $m_{ij}'=m_{\sigma(i)\sigma(j)}$. The base case of our induction for the proof of the twist identity is therefore trivial.
\end{proof}

\section{Permanent}
As our first case of a previously studied graph invariant with the same symmetries as the period and which can now be explained in terms of the Martin invariant, we will study a permanent-based invariant due to Iain Crump \cite{CrumpDeVosYeats:Permanent, Crump:ExtendedPermanent}.
Through \cref{thm:Martin-Perm2} below, it inherits all properties of the Martin invariant. We therefore obtain new proofs of the permanent's invariance under completion, twist and duality \cite[Theorem~17--Proposition~20]{CrumpDeVosYeats:Permanent}; its multiplicativity for $3$-vertex cuts and $2k=4$-edge cuts \cite[Corollary~23 and Theorem~24]{CrumpDeVosYeats:Permanent}; and the vanishing in presence of a triple edge \cite[Proposition~69]{Crump:PhD} (a special case of $\Martin(G)=0$ for a $2$-edge cut).
The same applies to the permanents of the graphs $G^{[r]}$ when $kr+1$ is prime, whose invariance was first proved in \cite{Crump:ExtendedPermanent} but now also follows from the properties of the Martin invariant. The Fourier split was not yet discovered at the time of \cite{CrumpDeVosYeats:Permanent} and \cite{Crump:ExtendedPermanent} but also holds for the permanent and is discussed in \cref{sec fourier split}.

Do define the graph permanent, let $G$ be a $2k$-regular graph with $n+2$ vertices. Then $G$ has $kn+2k$ edges. Pick two different vertices, called $\infty$ and $0$, and orient each edge. We assume that there are no self-loops at $\infty$. Then the directed graph $G\setminus \infty$ has $n+1$ vertices and $kn$ edges. Its \emph{reduced incidence matrix} is the $n\times kn$ matrix
\begin{equation}\label{eq:incidence-matrix}
    A_{ve} = \begin{cases}
        +1  & \text{if $e$ points to $v$},  \\
        -1 & \text{if $e$ comes from $v$},\\
        \phantom{-}0  & \text{if $e$ is not incident to $v$,}
    \end{cases}
\end{equation}
whose rows are indexed by the vertices $v\neq 0,\infty$ and whose columns are indexed by the edges $e$ of $G$ that are not attached to $\infty$. Stacking $k$ copies of $A$ produces a $kn\times kn$ square matrix, denoted $A^{[k]}$. Its permanent defines an integer
\begin{equation}\label{eq:perm-def}
    \Perm(G) \defas \perm A^{[k]}
    = \perm \begin{pmatrix} A \\ \vdots \\ A \\ \end{pmatrix}
    .
\end{equation}
This number $\Perm(G)$ depends on the choice of $0$, $\infty$, and the orientation of the edges used to write down $A$. However, it was shown in \cite{CrumpDeVosYeats:Permanent} that the residue
\begin{equation*}
    \Perm(G) \mod (k+1)
\end{equation*}
depends on all these choices only up to a sign. Equivalently, the residue of $\Perm(G)^2$ is a well-defined invariant of the graph $G$. This invariant respects the known identities of Feynman periods \cite{CrumpDeVosYeats:Permanent}.
Moreover, the sequence $\Perm(G^{[r]})^2 \mod (kr+1)$ of permanents of all duplicated graphs is a very strong invariant of $G$ in that it distinguishes almost all Feynman periods \cite{Crump:PhD,Crump:ExtendedPermanent}.

In this section, we show that the permanent is determined by the Martin invariant.
Note that the repetition of rows makes $\Perm(G)$ a multiple of $(k!)^n$. It is therefore congruent to zero modulo $k+1$ unless $k+1$ is prime \cite[Proposition~13]{CrumpDeVosYeats:Permanent}.
\begin{theorem}\label{thm:Martin-Perm2}
    For every $2k$-regular graph $G$ with $n\geq 3$ vertices such that $p=k+1\geq 3$ is prime, we have
    \begin{equation}\label{eq:Martin-Perm2}
        \Martin(G)\equiv (-1)^{n-1}\Perm(G)^2 \mod p.
    \end{equation}
\end{theorem}

Note that each term of the sequence $\Perm(G^{[r]})^2 \mod (kr+1)$ for which $kr+1$ is prime immediately also falls under this theorem and so is determined by the Martin invariant on the $G^{[r]}$. 

\begin{figure}
    \centering
    $G=\Graph[0.5]{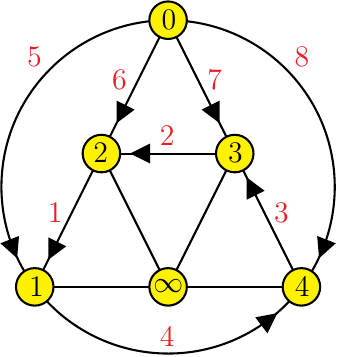}$ 
    \qquad
    $A=
    \begin{pmatrix}
            \phantom{-}1 &  \phantom{-}0 &  \phantom{-}0 & -1 &\  1 &\  0 &\  0 &\  0 \\
           -1 &  \phantom{-}1 &  \phantom{-}0 &  \phantom{-}0 &\  0 &\  1 &\  0 &\  0 \\
            \phantom{-}0 & -1 &  \phantom{-}1 &  \phantom{-}0 &\  0 &\  0 &\  1 &\  0 \\
            \phantom{-}0 &  \phantom{-}0 & -1 &  \phantom{-}1 &\  0 &\  0 &\  0 &\  1 \\
    \end{pmatrix}
    $
    \caption{A labelled and oriented octahedron graph. Labels and orientations of edges at vertex $\infty$ are not shown, since they do not affect the incidence matrix $A$.}%
    \label{fig:perm-octahedron}%
\end{figure}
\begin{example}
    The octahedron graph has $6$ vertices and it is 4-regular, so $k=2$. With labels and orientations as in \cref{fig:perm-octahedron}, the incidence matrix
    gives $\perm\big(\begin{smallmatrix}A\\A\end{smallmatrix}\big)=32$.
    Therefore, $(-1)^{n-1}\Perm(G)^2=-32^2\equiv -1\mod 3$. This agrees indeed with $\Martin(G)=14\equiv -1 \mod 3$, because the Martin invariant is
    \begin{equation*}
        \Martin\left(\Graph[0.5]{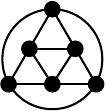} \right)
        =\Martin\left(K_5\right) + 2 \Martin\left(\Graph[0.55]{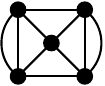}\right)
        =6 + 4 \Martin\left(\Graph[0.4]{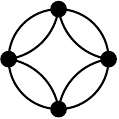}\right) + 2 \Martin\left(\Graph[0.55]{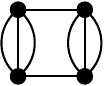}\right)
        =6+8+0.
    \end{equation*}
\end{example}

We will give two proofs of \cref{thm:Martin-Perm2}.  The first is elementary; the main work goes into considering how a transition, viewed as a rule for replacing a vertex with new edges joining its neighbours according to a matching, affects the permanent of $A^{[k]}$.  A similar but more intricate consideration of the effect of moving between a vertex and a matching of its neighbourhood will also be key to the results of \cref{sec diags}.

For the elementary proof of \cref{thm:Martin-Perm2}, we show that the right-hand side of \eqref{eq:Martin-Perm2} satisfies \cref{def:Martin-intro} of $\Martin(G)$, modulo $p$.
Firstly, any self-loop in $G$ corresponds to a zero column in $A$, hence $\perm A^{[k]}=0$ and \eqref{eq:Martin-Perm2} is trivial. Secondly, if $G$ has 3 vertices and no self-loop, then $G$ is the power of a triangle. Thus $G\setminus\infty$ is a $k$-fold edge and $A=(1,\ldots,1)$ is a single row of ones, if we orient the edges away from vertex $0$. Hence
\begin{equation*}
    \Perm\left(C_3^{[k]}\right)^2 = (k!)^2 \equiv (-1)^2 
    = 1
    = \Martin\left(C_3^{[k]}\right)
    \mod p
\end{equation*}
as required, by Wilson's theorem.\footnote{Wilson's theorem states that $(p-1)!\equiv-1\mod p$ if and only if $p$ is prime.} By induction, we can therefore conclude \cref{thm:Martin-Perm2} by proving a third identity: If $G$ has no self-loops and $v$ is a vertex $v\notin\set{0,\infty}$, then
\begin{equation}\label{eq:Perm2-recurrence}
    -\Perm(G)^2 \equiv \sum_{\tau\in\trans'(v)} \Perm(G_\tau)^2
    \mod p,
\end{equation}
where $\trans'(v)$ denotes the subset of transitions that do not produce any self-loops at $\infty$.
To establish this identity in \cref{sec:Perm2-recurrence}, and thereby finish the proof of \cref{thm:Martin-Perm2}, we first work out expansions of the permanents of $G$ and $G_{\tau}$.

For the second proof of \cref{thm:Martin-Perm2}, in \cref{sec:perm2-from-det} we show that the permanent can be read off from the diagonal coefficients of the Kirchhoff and Symanzik polynomials. Together with \cref{thm martin as coeff}, this provides an independent proof of \cref{thm:Martin-Perm2}.

\subsection{Expansion}
Like determinants, permanents can be expanded in rows and columns. For the graph permanent, this strategy leads to explicit formulas \cite[Section~4]{Crump:ExtendedPermanent}. Here, we apply this strategy to prove the recurrence relation \eqref{eq:Perm2-recurrence}.

Label the vertices of $G$ as $0,1,\ldots,n,\infty$ and let $N=\set{e_1,\ldots,e_{2k}}$ denote the $2k$ edges at vertex $n$. Let $v_i\neq n$ denote the other end of $e_i$. Note that several of these edges may share the same other end $v_i=v_j$.

The non-zero entries in row $n$ of $A$ are $-1$, located in the columns $N$. The expansion of $\perm(A^{[k]})$ in the $k$ copies of row $n$ therefore eliminates $k$ of the columns $N$. Let $a\subset N$ denote the leftover edges of $N$. Edges $e_i\in N$ incident to $v_i=\infty$ are not present as columns in $A$ and can thus not be among the columns $N\setminus a$ eliminated during the row expansion.   The expansion then reads
\begin{equation}\label{eq:permG-expansion}
    \Perm(G) = k! \cdot (-1)^k \sum_{a\subset N,\abs{a}=k} P_a
\end{equation}
where we define $P_a=0$ if there is any $e_i$ incident to $\infty$ which is not in $a$, and otherwise define $P_a$ to be the permanent of the $k(n-1)\times k(n-1)$ matrix that remains from $A^{[k]}$ after deleting the $k$ copies of row $n$ and the $k$ columns $N\setminus a$. The factor $k!$ accounts for the different ways of assigning the columns $N\setminus a$ to the $k$ copies of the row $n$.

Now expand $P_a$ in the leftover columns, $e_i\in a$ where $v_i\neq\infty$: the non-zero entries are $+1$, located in the rows corresponding to the $k$ copies of $v_i$. Let $a_i$ denote the number of edges in $a$ whose other vertex is $i$, and write $P'_a$ for the permanent of any matrix obtained from $A^{[k]}$ by deleting all columns $e_i\in N$ with $v_i\neq\infty$, the $k$ copies of row $n$, and $a_i$ copies of row $i$ for each $1\leq i<n$.  This matrix is square with $kn-2k+a_{\infty}$ rows and columns. The different ways to delete $a_i$ copies of row $i$ only affect the matrix up to permutation of rows which is invisible to the permanent, and so $P_a'$, and hence also $P_a$, depends only on the numbers $a_0,\ldots,a_{\infty}$ (the multiset of vertices $a$).  Then the expansion of $P_a$ in columns $a$ reads
\begin{equation*}
    P_a = \left[\prod_{i=1}^{n-1} \frac{k!}{(k-a_i)!}\right] P_a'.
\end{equation*}
The prefactor takes into account the different ways to assign $a_i$ parallel edges (identical columns) to $a_i$ out of the $k$ copies of vertex $i$. We note any column $e_i\in a$ with $v_i=0$ becomes a zero column after deleting the copies of row $n$, since vertex $0$ is absent in $A$. Therefore, $P_a=0$ if $a_0>0$.
Along with our convention that $P_a=0$ if $a_\infty < N_\infty$ (where $N_\infty$ is the number of $e_i$ in $N$ with $v_i=\infty$), the sum in \eqref{eq:permG-expansion} therefore is effectively only a sum over subsets $a$ which contain none of the edges from $n$ to $0$, and all of the edges from $n$ to $\infty$. Squaring, we find
\begin{equation}\label{eq:Perm2-expansion}
    \Perm(G)^2 = (k!)^2 \sum_{\substack{a,b\subset N,\\\abs{a}=\abs{b}=k}} P_a P_b.
\end{equation}

\subsection{Expansion of a transition}
Pick any transition $\tau$ at vertex $n$ that does not produce self-loops at $\infty$ in $G_{\tau}$. Then the incidence matrix $A_\tau$ of the graph $G_\tau$ has dimensions $(n-1)\times k(n-1)$. It is obtained from $A$ by removing row $n$ and the $2k-N_\infty$ columns $N$ that do not connect to $\infty$, and adding $k-N_\infty$ columns for the edges $e_1^\tau, \ldots, e_k^\tau$ given by the perfect matching $\tau$ of $N$, but excluding those edges $e_j^\tau$ that connect to $\infty$.

Let $\Or(\tau)$ denote the set of orientations of the matching, that is, the set of size $2^k$ whose elements $\alpha$ choose for each edge $e_i^\tau$, one of its incident vertices $\alpha(i) = v_{\alpha(i)}$, which we view as the vertex that edge $e^{\tau}_i$ points to under this orientation. Let $\alpha_j$ denote the number of edges $e^{\tau}_i$ which point to $v_{\alpha(i)}=v_j$. 

To define the signs $\pm 1$ in the new columns of $A_{\tau}$, fix any orientation $o\in\Or(\tau)$ of the new edges in $G_{\tau}$. The expansion of $\perm(A_{\tau}^{[k]})$ in one of the columns $e^\tau_i$ either picks one of the copies of row $v_{\alpha(i)}$ or one of the copies of the row for the other end of $e_i^\tau$. The expansion in all these columns can therefore be written as
\begin{equation*}
    \Perm(G_\tau) = \sum_{\alpha\in \Or(\tau)} (-1)^{o \Delta \alpha} P_\alpha,
\end{equation*}
where $o\Delta\alpha$ denotes the number of indices $i$ where $o(i)\neq \alpha(i)$, that is, how often an entry $-1$ is picked from the matrix. Whenever an orientation picks vertex $v_{\alpha(i)}=0$ as the vertex pointed to, the corresponding contribution is zero, because this vertex is not a row in $A_{\tau}$ and hence does not contribute to the expansion in column $e^{\tau}_i$. This is consistent with our convention $P_{\alpha}=0$ whenever $\alpha_0>0$, as discussed in the previous section.

Furthermore, our convention that $P_{\alpha}=0$ whenever $\alpha_{\infty}<N_{\infty}$, forces that all new edges $e^{\tau}_i$ that have one end at $\infty$ must be directed towards $\infty$. This reduces the sum effectively to a sum over the orientations of only those new edges $e^{\tau}_i$ that are not connected to $\infty$. This is indeed correct, because only those edges actually appear as columns in $A_{\tau}$ and thus participate in the column expansion.

In conclusion, we can write the expansion of the permanent squared of $G_{\tau}$ as
\begin{equation}\label{eq:Perm2tau-expansion}
    \Perm(G_\tau)^2 = \sum_{\alpha,\beta\in \Or(\tau)} (-1)^{\alpha \Delta \beta} P_{\alpha} P_{\beta}.
\end{equation}

\subsection{Recurrence}\label{sec:Perm2-recurrence}
The vertices chosen by an orientation $\alpha\in\Or(\tau)$ define a subset $a=\set{\alpha(1),\ldots,\alpha(k)}\subset N$ of size $k$. Conversely, by definition, an orientation $\alpha\in\Or(\tau)$ is determined by the subset $a$. However, $a$ defines an orientation of $\tau$ only if it picks precisely one element in each pair of $\tau$.

Consider any pair $a,b\subset N$ with $\abs{a}=\abs{b}=k$, and compare the coefficient of the product $P_a P_b$ on both sides of \eqref{eq:Perm2-recurrence}. On the left, it appears with coefficient
\begin{equation*}
    - (k!)^2 \equiv -1 \mod p,
\end{equation*}
according to \eqref{eq:Perm2-expansion} and Wilson's theorem
. On the right of \eqref{eq:Perm2-recurrence}, $P_a P_b$ appears as $P_{\alpha}P_{\beta}$ in \eqref{eq:Perm2tau-expansion} for each transition $\tau$ such that $\alpha=a$ and $\beta=b$ are orientations of $\tau$. The collected coefficient of $P_a P_b$ on the right-hand side of \eqref{eq:Perm2-recurrence} is therefore
\begin{equation}\label{eq:perm-rec-congruence}
    (-1)^{a\Delta b} \cdot \abs{\set{\tau\in\trans'(n)\colon a,b\in\Or(\tau)}}.
\end{equation}
We get a $\tau\in \trans'(n)$ with $a,b\in \Or(\tau)$ by pairing the elements in $a\cap b$ with elements of $N\setminus (a\cup b)$ giving those pairs with the same orientation in $a$ and $b$, and pairing the elements in $a\setminus b$ with the elements in $b\setminus a$ giving the pairs with opposite orientation in $a$ and $b$. Every such choice of pairings gives a distinct $\tau\in \trans'(n)$ with $a,b\in \Or(\tau)$.

The number $s=a\Delta b$ of pairs $\tau(i)$ that get opposite orientations in $a$ and $b$ is $s=\abs{a\setminus b}=\abs{b\setminus a}$. The other $k-s$ pairs have equal orientations and go from $N\setminus(a\cup b)$ to $a\cap b$. We conclude that precisely $s!\cdot(k-s)!$ transitions contribute to \eqref{eq:perm-rec-congruence}, namely those that match $a\cap b$ with $N\setminus(a\cup b)$ and $a\setminus b$ with $b\setminus a$. Therefore, \eqref{eq:perm-rec-congruence} equals
\begin{align*}
    & s!\cdot (k-s)!\cdot(-1)^{s} \\
    &\equiv s!\cdot [k+1-(k-s)]\cdot[k+1-(k-s-1)]\cdots[k+1-1]\cdot (-1)^k \mod (k+1)\\
    &= s!\cdot (s+1)(s+2)\cdots k \cdot (-1)^k \\
    & = (-1)^k\cdot k!.
\end{align*}
This is equal to the left-hand side $-1\mod p$ of \eqref{eq:Perm2-recurrence}, by Wilson's theorem and because $k$ is even (so that $p=k+1$ is a prime $\geq 3$). Finally, note that we only need to consider $a$ and $b$ with $a_{\infty}=b_{\infty}=N_{\infty}$. Any transition $\tau$ that creates a self-loop at $\infty$ in $G_{\tau}$, has at most $N_{\infty}-1$ pairs attached to $\infty$, and hence $a$ and $b$ cannot be orientations of such $\tau$. Therefore, in our counting above, indeed only the transitions $\tau\in\trans'(n)$ appear.

This completes the proof of \cref{thm:Martin-Perm2}.

\subsection{From the Kirchhoff polynomial}\label{sec:perm2-from-det}
In this section we express the permanent, which is defined via the incidence matrix $A$ from \eqref{eq:incidence-matrix}, in terms of the graph Laplacian. Together with \cref{thm martin as coeff}, this yields an alternative proof of \cref{thm:Martin-Perm2}. Recall that $G$ is a $2k$-regular graph with $n+2$ vertices, labelled $0,1,\ldots,n,\infty$. In this section we assume that $G$ has no self-loops.
\begin{definition}
   Let $X$ denote the $kn\times kn$ matrix with variables $x_i=X_{ii}$ on the diagonal and zeroes elsewhere. The \emph{reduced graph Laplacian} of the graph $G\setminus\infty$ is the $n\times n$ matrix
    \begin{equation}\label{eq:Laplacian}
        L=A X A^{\Transpose}.
    \end{equation}
\end{definition}
The variables $x_i$ are associated to the edges of $G\setminus \infty$. The rows and columns of $L$ are indexed by vertices $v \neq 0,\infty$. This matrix does not depend on orientations: diagonal entries $L_{v,v}=\sum_e x_e$ are positive sums over all edges $e$ at $v$, whereas off-diagonal entries $L_{v,w}=-\sum_e x_e$ are negative sums over all edges between $v$ and $w\neq v$.
\begin{lemma}
    Denote the $kn\times kn$ matrix obtained by repeating $k\times k$ blocks of $L$ as
    \begin{equation}\label{eq LLLL}
        L^{[k\times k]}
        =
        \begin{pmatrix}
            L & \cdots & L \\
            \vdots & \ddots & \vdots \\
            L & \cdots & L \\
        \end{pmatrix}
        .
    \end{equation}
    Recall the notation $\takecoeff{x^m}$ for coefficient extraction of a monomial $x^m=\prod_i x_i^{m_i}$, see \cref{def coeff extr}. Then
    \begin{equation*}\label{eq:perm2-diag-laplace}\tag{$\ast$}
        \left(\perm A^{[k]}\right)^2
        =\takecoeff{x_1\cdots x_{kn}} 
        \perm \left( L^{[k\times k]} \right)
        .
    \end{equation*}
\end{lemma}
\begin{proof}
    The block matrix on the right-hand side of \eqref{eq LLLL} is equal to $L^{[k\times k]}=A^{[k]} X {(A^{[k]})}^{\Transpose}$. By the MacMahon master theorem, the right-hand side of \eqref{eq:perm2-diag-laplace} is therefore equal to
    \begin{equation*}
        \takecoeff{x_1\cdots x_{kn}} \takecoeff{y_1\cdots y_{kn}} \frac{1}{\det(I_{kn}-Y A^{[k]} X {(A^{[k]})}^{\Transpose})},
    \end{equation*}
    where $Y$ denotes a $kn\times kn$ matrix with variables $y_i=Y_{ii}$ on the diagonal and zeroes everywhere else, and $I_{kn}$ denotes the $kn\times kn$ identity matrix.
    The generalization of the MacMahon master theorem from \cite[Theorem~1]{ChabaudDeshpandeMehraban:QuantPer} shows that the above expression is equal to the product of the permanents $\perm (A^{[k]})$ and $\perm( {(A^{[k]})}^{\Transpose})=\perm(A^{[k]})$.
\end{proof}
\begin{corollary}\label{lem:perm2-detL}
    For $p=k+1$ prime, $(\perm A^{[k]})^2\equiv (-1)^n \takecoeff{x_1\cdots x_{kn}} (\det L)^k \mod p$.
\end{corollary}
\begin{proof}
    This follows by Wilson's theorem and the congruence $\perm(L^{[k\times k]})\equiv (k!)^n (\det L)^k \mod (k+1)$ obtained in \cite[Theorem~7.5]{Glynn:MatsPointsPrime}.\footnote{With the notation in that paper, $X_{k+1}(B)=(\perm B^{[k]})/(k!)^n$.}
    The latter holds because performing row operations (adding a multiple of one row to another) simultaneously in all $k$ copies of any row of $L^{[k\times k]}$ does not change its permanent modulo $k+1$ \cite[Corollary~6]{CrumpDeVosYeats:Permanent}. For $R$ a sequence of such row operations such that $RL=D$ is a diagonal matrix (thus $\det L=\det D$), applying $R$ simultaneously to the $k$ copies of each row shows that
    \begin{equation*}
        \perm \left( L^{[k\times k]} \right)
        = \perm \left( D^{[k\times k]} \right)
        \equiv
        (\det L)^k
        \perm \left( I_{n}^{[k\times k]} \right)
        =
        (\det L)^k (k!)^n
        \mod p
    \end{equation*}
    factorizes into determinants of $L$ and the permanents of the $n$ blocks of $I_n^{[k\times k]}$. Each of the latter blocks is a $k\times k$ matrix with each entry equal to $1$, hence permanent $k!$.
\end{proof}
By the matrix-tree theorem, the determinant of $L$ is the \emph{Kirchhoff polynomial}
\begin{equation}\label{eq:Kirchhoff}
    \det L = \KirchPol_{G\setminus\infty}
    = \sum_T \prod_{e\in T} x_e
\end{equation}
where the sum is over all spanning trees $T$ of $G\setminus \infty$. It has degree $n$, and it is related to the Symanzik polynomial \eqref{eq:psipol} with degree $(k-1)n$ by an inversion of variables,
\begin{equation}\label{eq:Kirchhoff-Symanzik}
    \KirchPol_{G\setminus \infty} (x_1,\ldots,x_{kn})=
    \bigg(\prod_{e=1}^{kn} x_e\bigg) \cdot
    \PsiPol_{G\setminus\infty}\big(x_1^{-1},\ldots,x_{kn}^{-1}\big).
\end{equation}

In conclusion, \cref{lem:perm2-detL} expresses the square of the permanent as a diagonal coefficient of a power of the Kirchhoff or Symanzik polynomial:
\begin{equation}\label{eq:perm2-kirchhoff}
    (-1)^n\Perm(G\setminus\infty)^2
    \equiv \takecoeff{x_1\cdots x_{kn}} \KirchPol_{G\setminus\infty}^k
    = \takecoefff{\big}{x_1^{k-1}\cdots x_{kn}^{k-1}} \PsiPol_{G\setminus\infty}^k
    \mod p
    .
\end{equation}
\begin{remark}\label{rem:perm-non-regular}
    The results in this subsection do not require that $G$ is $2k$-regular. We only used that the number of edges of $G\setminus\infty$ is equal to $kn$, with integer $k$. Therefore, \cref{lem:perm2-detL} proves that the graph permanent \cite{CrumpDeVosYeats:Permanent} and its extension \cite{Crump:ExtendedPermanent} are determined, for any graph, by the cycle matroid of that graph. This proves \cite[Conjecture~3]{Crump:PhD}.
\end{remark}

\section{Diagonals of graph polynomials}\label{sec diags}

A unifying perspective on the Martin invariant, the graph permanent, and the $c_2$ invariant, is that they are related to diagonal coefficients of certain polynomials defined from the graph. This point of view was crucial for recent progress on the $c_2$ invariant \cite{HuSchnetzShawYeats:Further,HuYeats:c2p2,EsipovaYeats:c2powers} and will be used again in \cref{sec c2}. For the permanent, a diagonal interpretation was given in \cite[\S~5.1]{Crump:ExtendedPermanent} in terms of a new polynomial, whereas in \eqref{eq:perm2-kirchhoff} we identified the square of the permanent with a residue of the diagonal of the Kirchhoff polynomial. In the present section, we show that not just the residue, but the diagonal coefficient itself satisfies the Martin recurrence \eqref{eq:Martin-recursion}. Consequently, the Martin invariant of a regular graph can be identified, up to a normalization, with this coefficient (\cref{thm martin as coeff}).

The diagonal coefficients we are interested in have combinatorial interpretations as partitions of the edges into spanning trees (this section), or partitions into spanning trees and certain spanning forests (\cref{sec c2}).   We exploit this combinatorial interpretation to prove the Martin recurrence \eqref{eq:Martin-recursion} for these diagonal coefficients. The key tool is a generalization of the Pr\"ufer encoding enabling us to move from spanning trees on a graph to spanning trees on the graph with a vertex resolved by a transition viewed as a matching; this is the content of \cref{sec:pruefer}.

After these introductory remarks, consider now concretely the Kirchhoff polynomial $\KirchPol_G$ (appeared in \eqref{eq:Kirchhoff} for a decompletion) and the Symanzik polynomial $\PsiPol_G$ from \eqref{eq:psipol}:
\begin{equation*}
    \KirchPol_{G}
    = \sum_T \prod_{e\in T} x_e
    \qquad\text{and}\qquad
    \PsiPol_G = \sum_{T}
    \prod_{e\notin T} x_e
    .
\end{equation*}
For any graph $G$, both $\KirchPol_G$ and $\PsiPol_G$ are homogeneous, are linear in each variable, and have all monomials appearing with coefficient $1$.  For $\KirchPol_G$ each monomial consists of the edges of a spanning tree while for $\PsiPol_G$ each monomial consists of the edges that when deleted leave a spanning tree of $G$. For a disconnected graph, $\PsiPol_G=\KirchPol_G=0$. If $G$ is connected, then the polynomials are non-zero (spanning trees exist), and their total degrees are
\begin{equation*}
    \deg \KirchPol_G = n-1
    \qquad\text{and}\qquad
    \deg \PsiPol_G = m-n+1
\end{equation*}
if $G$ has $n$ vertices and $m$ edges. The degree $m-n+1$ equals the dimension of the cycle space of $G$ (the first Betti number of $G$ viewed as a topological space); this number is known as \emph{nullity} or \emph{corank} and in the quantum field theory community as the \emph{loop number}. All these observations are immediate from the definitions \eqref{eq:psipol} and \eqref{eq:Kirchhoff}.

\begin{example}\label{eg dunce kirch and psi}
    Let $G$ be the triangle with one doubled edge. Let indices $1$ and $2$ be the doubled edge and $3$ and $4$ be the other edges, as in \cref{fig:duncecap}. Then $n=3$, $m=4$, $\KirchPol_G = x_1x_3+x_1x_4+x_2x_3+x_2x_4+x_3x_4$ and $\PsiPol_G = x_2x_4+x_2x_3+x_1x_4+x_1x_3+x_1x_2$.
\end{example}

Let $G$ be a graph with $m$ edges.  We are particularly interested in the \emph{diagonal coefficient} of $\KirchPol_G^k$, namely $\takecoeff{x_1\cdots x_{m}}\KirchPol_{G}^{k}$ where $\takecoeff{\cdot}$ denotes the coefficient extraction from \cref{def coeff extr}.  Since $\KirchPol_G^k$ has degree $k(n-1)$, where $n$ is the number of vertices of $G$, note that the diagonal coefficient is necessarily zero whenever $m\neq k(n-1)$. We only consider cases with $m=k(n-1)$. Then $\PsiPol_G$ has degree $(k-1)(n-1)$, and so we also have a diagonal coefficient $\takecoefff{}{x_1^{k-1}\cdots x_{m}^{k-1}}\PsiPol_{G}^{k}$. Due to \eqref{eq:Kirchhoff-Symanzik}, they agree:
\begin{equation}\label{eq coeffs agree}
\takecoeff{x_1\cdots x_{m}}\KirchPol_{G}^{k} = \takecoefff{\big}{x_1^{k-1}\cdots x_{m}^{k-1}}\PsiPol_{G}^{k}.
\end{equation}

Consider now a $2k$-regular graph $G$ with $n$ vertices and thus $kn$ edges, and let $v$ denote a vertex without any self-loops. Then the graph $G\setminus v$ has $n-1$ vertices and $m=k(n-2)$ edges, so that the above diagonal coefficients have the potential to be non-zero.
In \cref{thm martin as coeff} we will show that
\begin{equation*}
    \takecoeff{x_1\cdots x_{m}}\KirchPol_{G\setminus v}^{k}=k!\cdot \Martin(G).
\end{equation*}
The main technical step is to prove that the diagonal coefficient on the left-hand side satisfies the Martin recurrence.  We will do this using a combinatorial interpretation of the diagonal coefficient. In \cref{sec 3 inv recurrence} we will apply similar considerations to the diagonal of a different but related polynomial.

\begin{definition}
A \emph{partition} of a set $S$ is a set $P=\set{p_1,\ldots,p_k}$ of non-empty subsets $p_i\subseteq S$, called \emph{the parts of $P$}, such that each element of $S$ is in exactly one part of $P$.

An \emph{ordered partition} of $S$ is a tuple $P = (p_1, \ldots, p_k)$ of distinct non-empty subsets of $S$ that form a partition of $S$.
\end{definition}

\begin{definition}
    For a graph $G$ and integer $k$, define $N_k(G)$ to be the number of ordered partitions of the set of edges of $G$ into $k$ parts, so that each part is a spanning tree of $G$.
\end{definition}

We will use the following combinatorial interpretation of the diagonal coefficient.
\begin{lemma}\label{lem comb interp of diag}
    For a graph $G$ with $m$ edges, and any positive integer $k$,
    \begin{equation*}
    N_k(G) = \takecoeff{x_1\cdots x_{m}}\KirchPol_{G}^{k}.
    \end{equation*}
\end{lemma}
\begin{proof}
Every monomial which appears in $\KirchPol_{G}$ has coefficient $1$, so $N_k(G)$ counts the ordered partitions of the variables (edges of $G$) into a $k$-tuple of parts (one part from each factor of $\KirchPol_G^k$), where each part is a monomial (spanning tree) appearing in $\KirchPol_{G}$.
\end{proof}
Here the order of the trees $T_i$ matters. Without the order, we get a partition $P=\set{T_1,\ldots,T_k}$ of the edge set of $G$ and require an extra combinatorial factor.
\begin{corollary}\label{diag=STP}
    The diagonal coefficient $N_k(G)$ of $\KirchPol_G^k$ is equal to $k!$ times the number of partitions of the edge set of $G$ such that each part is a spanning tree.
\end{corollary}
\begin{proof}
    Since the trees are edge-disjoint, the map $(T_1,\ldots,T_k)\mapsto \set{T_1,\ldots,T_k}$ that forgets the order is $k!$-to-1.
\end{proof}

\begin{example}
    Continuing with the graph $G$ in \cref{eg dunce kirch and psi}, we see that there are two partitions of the edges of $G$ into spanning trees $\{\{1,3\}, \{2,4\}\}$ and $\{\{1,4\}, \{2,3\}\}$ and hence four ordered partitions as there are two orderings of each partition, giving $N_2(G)=4$.  This is consistent with $[x_1x_2x_3x_4]\KirchPol_G^2 = [x_1x_2x_3x_4](x_1x_3+x_1x_4+x_2x_3+x_2x_4+x_3x_4)^2 = 4$.
\end{example}

\begin{remark}\label{rem Symanzik version of combi argument}
It is also the case that $\PsiPol_G$ has every monomial with coefficient $1$, so the same argument as in the proof of \cref{lem comb interp of diag} interprets the diagonal coefficient of $\PsiPol_G^k$ as a $k$-tuple of complements of spanning trees, of which there are also $N_k(G)$. This again explains \eqref{eq coeffs agree}, where here the act of taking the complement of a set of variables is interpreted combinatorially---instead of the equivalent algebraic formulation in \eqref{eq:Kirchhoff-Symanzik}.
\end{remark}

\begin{lemma}\label{duplication-diagonal}
    For a graph $G$ with $m$ edges and for all positive integers $r$ and $k$, we have
    \begin{equation}\label{eq:duplication-diagonal}
        N_{kr}\big(G^{[r]}\big)
        =\takecoeff{x_1\cdots x_{rm}} \KirchPol_{G^{[r]}}^{kr}
        =(r!)^{m}\takecoeff{x_1^{r}\cdots x_{m}^r} \KirchPol_{G}^{kr}.
    \end{equation}
\end{lemma}
\begin{proof}
This is \cite[Proposition~40]{Crump:ExtendedPermanent}. Set $e_i=(e-1)r+i$ so that $e_1,\ldots,e_r$ label the $r$ copies in $G^{[r]}$ of any edge $e\in\set{1,\ldots,m}$ of $G$. For a spanning tree $T'$ of $G^{[r]}$, replacing each $e_i\in T'$ by $e$ defines a spanning tree $T$ of $G$. In particular, $T'$ can contain at most one copy of any edge $e$ of $G$. Thus $T'$ is determined by $T$ together with choices $\iota(e)\in\set{1,\ldots,r}$ of one copy for each $e\in T$. Therefore,
\begin{equation*}
    \KirchPol_{G^{[r]}}(x_1,\ldots,x_{rm})
    =\sum_T \sum_{\iota} \prod_{e\in T} x_{e_{\iota(e)}}
    =\sum_T \prod_{e\in T} \sum_{i=1}^r x_{e_i}
    =\KirchPol_G(y_1,\ldots,y_m)
\end{equation*}
where $y_e=x_{r(e-1)+1}+\ldots+x_{er}=x_{e_1}+\ldots+x_{e_r}$ is the sum of the variables associated to the copies of edge $e$. This identity also follows from the matrix tree theorem \eqref{eq:Kirchhoff}, since the reduced graph Laplacian $L$ from \cref{sec:perm2-from-det} for $G^{[r]}$ is the same as that of $G$, except for the replacement $x_e\mapsto y_e$. Now note that
\begin{equation*}
    \takecoeff{x_{e_1}\cdots x_{e_r}} y_e^q
    = \takecoeff{x_{e_1}\cdots x_{e_r}} (x_{e_1}+\ldots+x_{e_r})^q
    = \begin{cases}
        r! & \text{if $q=r$ and} \\
        0  & \text{otherwise,} \\
    \end{cases}
\end{equation*}
so taking the linear coefficient in the copies of edge $e$ is the same---up to a factor $r!$---as taking the coefficient of $y_e^r$. The claim follows.
\end{proof}
Similarly to \cref{lem comb interp of diag} and \cref{rem Symanzik version of combi argument}, we can interpret the coefficient $\takecoeff{x_1^{r}\cdots x_{m}^r} \KirchPol_{G}^{kr}\in\Z$ as the number of lists $(T_1,\ldots,T_{kr})$ of spanning trees of $G$ such that each edge of $G$ appears in precisely $r$ of these trees. In contrast to \cref{diag=STP}, this does not imply that this coefficient is a multiple of $(kr)!$, because the $T_i$ are not necessarily distinct. Instead, we can show
\begin{lemma}\label{diag-multinomial-multiple}
    Let $G$ be a graph with $m$ edges and $n\geq 2$ vertices. Let $k$ and $r$ be positive integers. Then $\takecoeff{x_1^{r}\cdots x_{m}^r} \KirchPol_{G}^{kr}$ is divisible by the integer $(kr)!/(r!)^k$.
\end{lemma}
\begin{proof}
    Every spanning tree has $n-1$ edges, so we may assume that $m=k(n-1)$: otherwise, no edge decompositions into $kr$ spanning trees can exist and the claim becomes trivial since the coefficient is zero.

    Any tuple $(T_1,\ldots,T_{kr})$ of spanning trees determines a multiset $t_1^{\mu_1}\ldots t_p^{\mu_p}$ that records the set $\set{t_1,\ldots,t_p}=\set{T_1,\ldots,T_{kr}}$ of \emph{distinct} trees and their multiplicities, where $\mu_i\geq 1$ denotes the number of indices $s$ such that $t_i=T_s$. Note that $\mu_1+\ldots+\mu_p=kr$, so $\mu$ is necessarily a composition of $kr$, and therefore
    \begin{equation*}
        \frac{(kr)!}{\mu_1!\cdots\mu_p!} = \binom{kr}{\mu_1,\ldots,\mu_p}
    \end{equation*}
    is a multinomial coefficient. It counts the number of distinct ordered tuples that correspond to the multiset. We can therefore write the diagonal coefficient as a sum
    \begin{equation*}\label{eq:diag-multiset}\tag{$\dagger$}
        \takecoeff{x_1^{r}\cdots x_{m}^r} \KirchPol_{G}^{kr}
        = \sum_{p,t,\mu} \frac{(kr)!}{\mu_1!\cdots\mu_p!},
    \end{equation*}
    over all multisets of spanning trees with the constraint that each edge $e$ of $G$ appears in precisely $r$ of the trees, counted with multiplicity. This condition can be phrased as
    \begin{equation*}
        r = \mu_1\chi_1(e)+\ldots+\mu_p \chi_p(e)
    \end{equation*}
    where we set $\chi_i(e)=1$ if $e\in t_i$ and $\chi(e)=0$ otherwise. It follows that
    \begin{equation*}
        \frac{r!}{(\mu_i!)^{\chi_1(e)}\cdots(\mu_p!)^{\chi_p(e)}} = \binom{r}{\mu_1 \chi_1(e),\ldots,\mu_p \chi_p(e)}
        \in \Z
    \end{equation*}
    is an integer. Multiplying these identities for all $k(n-1)$ edges of $G$, we find that
    \begin{equation*}
        \left(\frac{(r!)^k}{\mu_1!\cdots\mu_p!}\right)^{n-1}
        \in \Z
    \end{equation*}
    is an integer. Here we used that $\sum_e \chi_i(e)=\abs{t_i}=n-1$ as $t_i$ is a spanning tree. We conclude that the rational number $(r!)^k/(\mu_1!\cdots\mu_p!)$ is an integer.\footnote{An elementary argument shows that $q^n\in\Z$, for $n\in\Z_{>0}$ and rational $q$, implies that $q\in\Z$.} Hence, for every multiset that contributes to \eqref{eq:diag-multiset}, the summand is divisible by $(kr)!/(r!)^k$.
\end{proof}

\begin{example}\label{ex:diag-dipole-power}
For the $k$-fold edge $G=K_2^{[k]}$ with $m=k$ edges and $n=2$ vertices, we have
$\KirchPol_{G} = x_1+\cdots +x_k$ and $\PsiPol_{G} = \sum_{i=1}^k \prod_{j\neq i} x_j$. The diagonal coefficients are
\begin{equation*}
    \takecoefff{\big}{x_1^{r(k-1)}\cdots x_k^{r(k-1)}}\PsiPol_{G}^{kr} = \takecoeff{x_1^{r}\cdots x_k^{r}}\KirchPol_{G}^{kr} = \binom{kr}{r,\ldots,r} = \frac{(kr)!}{(r!)^k}.
\end{equation*}
\end{example}

\subsection{Coloured matchings and distinct representatives}\label{sec:pruefer}

The key technical step towards the proof of \Cref{thm martin as coeff,thm:c2-martin} is a bijection based on Pr\"ufer codes between two sets called $A$ and $B$, which we will now define. Throughout, we fix an integer $k>0$ and a non-empty set $S$ with at most $2k$ elements. We set $d=2k-\abs{S}\geq0$.

Given a partition $P=\set{p_1,\ldots,p_t}$ of $S$, we write $\abs{P}=t$ for the number of parts.
A subset $X\subseteq S$ is called a \emph{system of distinct representatives} (SDR) of $P$ if it contains exactly $\abs{X \cap p}=1$ element in each part $p\in P$.

Given any sequence $P_1,\ldots,P_k$ of partitions of $S$, we want to count families $(X_1,\ldots,X_k)$ of pairwise disjoint systems of distinct representatives $X_i$ of $P_i$ that cover $S=X_1\cup\ldots\cup X_k$. We can encode such as a map $f\colon S\rightarrow\set{1,\ldots,k}$ with $X_i=f^{-1}(i)$, and hence define
\begin{equation}\label{eq A}
    A(P_1,\ldots,P_k) = \set{f\colon S\rightarrow \set{1,\ldots,k}\colon f^{-1}(i)\ \text{is a SDR of}\ P_i\ \text{for each}\ i}.
\end{equation}
In our application, the set $S$ will consist of the $h$-neighbourhood of a vertex $u$ after deleting a vertex $v$.  Since $u$ began with degree $2k$, $d$ is the number of edges from $u$ to the deleted vertex $v$.  In our application, the set $A(P_1, \ldots, P_k)$ encodes all ways to extend forests of the graph without $u$ inducing the partitions $P_1, \ldots, P_k$ into trees using each edge around $u$ exactly once.

\begin{example}\label{eg split up 5.1 eg A part}
    We will run an example through this section. Let $S=\{1,2,\ldots, 12\}$ and $k=8$. In the context where we will ultimately use the bijection that we'll define in \cref{prop:pruefer-bij}, $S$ will be the $h$-neighbourhood of a vertex $u$. For instance, the particular values of $S$ and $k$ could occur as in \cref{fig:prueffersetup}.

    \begin{figure}
        \centering
        \includegraphics{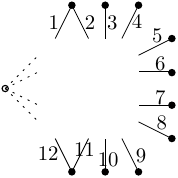}
        \caption{An illustration of $S$ as an $h$-neighbourhood of a central vertex $u$ which has been removed.  The dotted half edges and empty vertex on the left represent the connections to the deleted $v$, and are not included in $S$.}%
        \label{fig:prueffersetup}%
    \end{figure}
    
    Take the partitions $P_1=P_2=P_3=P_4=P_5=\set{S}$ to be trivial (one part) and let
    \begin{align*}
        P_6 & = \set{\{1,2,3,4,5,8\}, \{6,7\}, \{9,10,11,12\}}, \\
        P_7 & = \set{\{1,2,3,4,7,11,12\}, \{5,6,8,9,10\}}, \\
        P_8 & = \set{\{1,2,3,4,5,6\}, \{7,8,9,10,11,12\}}.
    \end{align*}
    For our context involving the $h$-neighbourhood of $u$, each partition of $S$ is induced from a partition of the vertices adjacent to $u$, and so with the vertices as illustrated in \cref{fig:prueffersetup}, the partitions must all have $1$ and $2$ in the same part and likewise for $11$ and $12$.

    Then one example of an element of $A(P_1, \ldots, P_8)$ in this case is given by $f(1)=f(12)=8, f(2)=1, f(3)=2, f(4)=3, f(5)=4, f(6)=5, f(7)=f(8)=f(9)=6, f(10)=f(11)=7$. \Cref{fig:pruefferwithu} illustrates the partitions as coloured blobs and $f$ with bold edges.  We shall return to this example after the next definitions.

    \begin{figure}
        \centering
        \includegraphics{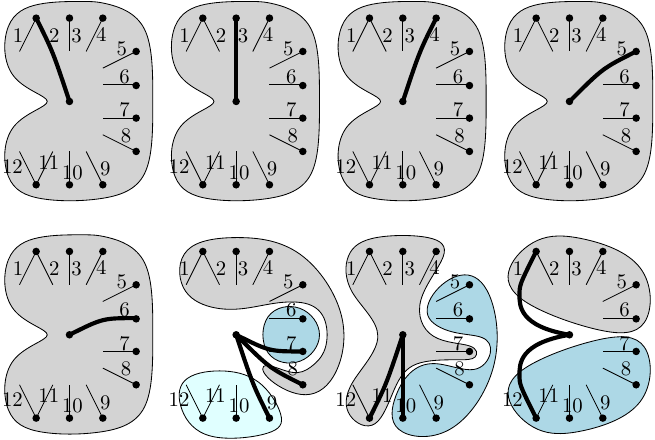}
        \caption{The systems of distinct representatives given by $f$ illustrated in the context of $S$ as the $h$-neighbourhood of the central vertex $u$.}%
        \label{fig:pruefferwithu}%
    \end{figure}
\end{example}


A \emph{matching} of $S$ is a set  of disjoint pairs of elements of $S$. Its \emph{support} is the set $\supp m \subseteq S$ of all elements of $S$ that appear in some pair $\set{a,b}\in m$.

Given a partition $P=\set{p_1,\ldots,p_t}$ of $S$, we can view a set $m$ of pairs of elements of $S$ as a graph whose vertices are the parts of $P$, by interpreting each element $\set{a,b}\in m$ as an edge between the vertices $p_i$ and $p_j$ corresponding to the parts that contain $a\in p_i$ and $b\in p_j$. If this graph happens to be a tree, then we say that \emph{$m$ defines a tree on $P$}.
\begin{definition}\label{def B}
    For any sequence $P_1,\ldots,P_k$ of partitions of $S$, let $B(P_1,\ldots,P_k)$ denote the set of tuples $(m,z)$ that consist of $k$ matchings $m = (m_1,\ldots,m_k)$ of $S$ and $d$ elements $z=(z_1,\ldots,z_d)\in S^d$ such that:
    \begin{enumerate}
        \item the matchings have $\abs{m_1}+\ldots+\abs{m_k}=k-d$ pairs in total,
        \item $m_i$ defines a tree on $P_i$ for each $1\leq i\leq k$, and
        \item $S=\supp m_1\cup\ldots\cup \supp m_k \cup \set{z_1,\ldots,z_d}$.
    \end{enumerate}
\end{definition}
Since $\abs{\supp m_i}=2\abs{m_i}$ and $\abs{S}=2k-d$, the first condition implies that the union in 3.\ is disjoint. Therefore, all $z_i$ are necessarily distinct, and $m_1\sqcup\ldots\sqcup m_k$ is a $k$-coloured perfect matching of $S\setminus\set{z_1,\ldots,z_d}$.

In our application the set $B(P_1, \ldots, P_k)$ will encode the ways to extend forests of the graph without vertex $u$ (and without $v$), inducing the partitions $P_1, \ldots, P_k$, into trees in such a way that the new edges are formed from a matching of the half edges at $u$, i.e.\ a transition. Each $z_i$ labels a half-edge that is paired with the deleted vertex $v$.

\begin{example}\label{eg split up 5.1 eg B part}
    Continuing \cref{eg split up 5.1 eg A part}, if we set $z = (3,4,5,6)$ and define the matchings $m_1=m_2=m_3=m_4=m_5 = \emptyset$, 
    \begin{equation*}
    m_6 = \set{\{2,7\}, \{8,9\}},\qquad
    m_7 = \set{\{10,11\}}, \qquad 
    m_8 = \set{\{1,12\}} \qquad
    \end{equation*}
    then this gives an element of $B(P_1, \ldots, P_8)$ and is is 
    illustrated in \cref{fig:prueffermatchings}.
    \begin{figure}
        \centering
        \includegraphics{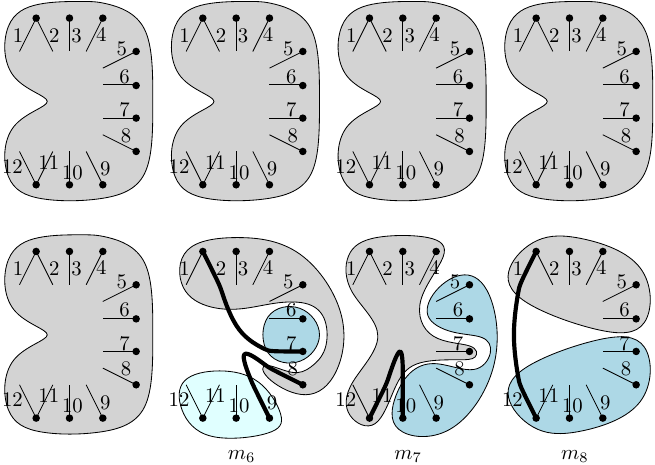}
        \caption{The matchings giving a $B$ set. The unmatched half-edges $3,4,5,6$ then form the vector $z$ and in this graph context would connect to the deleted vertex $v$.}%
        \label{fig:prueffermatchings}%
    \end{figure}
\end{example}

\begin{lemma}\label{lem when 0}
$A(P_1,\ldots, P_k)$ and $B(P_1,\ldots, P_k)$ are empty whenever $\sum_{i=1}^k\abs{P_i} \neq \abs{S}$.
\end{lemma}
\begin{proof}
For $A$, we must choose exactly one element from each part and use each element of $S$ exactly once, so the sum of the number of parts in each $P_i$ must be equal to $\abs{S}$. For $B$, since a tree has one fewer edges than vertices, we have $\sum_{i=1}^k (\abs{P_i}-1) = k-d$. Since $\abs{S}=2k-d$ by definition, this is the same as the condition for $A$.
\end{proof}

Let us recall the classic Pr\"ufer code \cite{Pruefer:Permutationen}. Fix a set $P$ of size $n\geq 2$ and equip it with a total order. Given a tree $t$ with vertex set $P$, start with $i=1$ and repeat:
\begin{enumerate}
    \item pick the smallest leaf $a_i\in P$ of $t$ and let $b_i\in P$ denote its neighbour,
    \item remove $a_i$ from $t$ and increment $i$.
\end{enumerate}
After $n-2$ steps, this produces the Pr\"ufer sequence $b=(b_1,\ldots,b_{n-2})$. The $b_i$ need not be distinct, in fact every tuple $b\in P^{n-2}$ arises exactly once this way. This is one way to prove Cayley's formula that there are $n^{n-2}$ labelled trees on $n$ vertices.

\begin{proposition}\label{prop:pruefer-bij}
    For any choice of $k>0$, any non-empty set $S$ of size $\abs{S}\leq 2k$, and any sequence $P_1,\ldots,P_k$ of partitions of $S$, we have 
    \begin{equation*}
    \abs{A(P_1,\ldots,P_k)} = \abs{ B(P_1,\ldots,P_k) }.
    \end{equation*}
\end{proposition}

\begin{proof}
Let $r$ denote the number of partitions $P_i$ with only one part. We may place these trivial partitions first, so that $P_1=\cdots=P_r=\set{S}$ and then $\abs{P_i}\geq 2$ for $i>r$. By \cref{lem when 0}, we may also assume that $\sum_{i=1}^k\abs{P_i} = \abs{S}=2k-d$, and therefore
\begin{equation*}
    r = d+\sum_{i=r+1}^k \left(\abs{P_i} - 2\right).
\end{equation*}

For each partition $P_i$, fix some order of its parts. We will now define a bijection $\alpha\colon B\rightarrow A$ between the sets $B=B(P_1,\ldots, P_k)$ and $A=A(P_1, \ldots, P_k)$.
Let $(m,z) \in B$.

For each $i>r$, consider the tree $t_i$ defined by $m_i$ on $P_i$. We run the Pr\"ufer algorithm for $t_i$ to build a sequence $y_i\in S^{\abs{P_i}-2}$ and a set $X_i\subseteq S$ of size $\abs{P_i}$ as follows: 
\begin{itemize}
    \item Loop for $\abs{P_i}-2$ steps:
    \begin{itemize}
        \item at the $j$th step take the remaining leaf which is first in the order on $P_i$ and let $\{a,b\}$ be the element of $m_i$ corresponding to the unique edge adjacent to this leaf where $a$ is in the leaf part and $b$ is in the other part, then
        \item put $b$ in the $j$th slot of $y_i$,
        \item add $a$ to $X_i$, and
        \item remove this leaf from the tree.
    \end{itemize}
    \item Let $\{a_1,a_2\}$ be the element of $m_i$ corresponding to the last remaining edge of the tree, then add $a_1$ and $a_2$ to $X_i$.
\end{itemize}
Since elements are added to $X_i$ when the part they are contained in is a leaf of the tree and then the leaf is removed or the algorithm terminates, there will be exactly one element of each part in $X_i$. Therefore, $X_i$ is a SDR for $P_i$. We also note that the tree $t_i$ can be reconstructed from $y_i$, because the classical Pr\"ufer code of $t_i$ is obtained by replacing each entry of $y_i$ with the part of $P_i$ that contains it.

Let $y\in S^r$ denote the sequence of length $r$ obtained by concatenation of the sequences $y_{r+1}, \ldots, y_k,z$. We define a function $f\colon S\rightarrow\set{1,\ldots,k}$ as follows:
\begin{itemize}
    \item for each $i>r$ set $f(a)=i$ for each $a\in X_i$,
    \item for each $i\leq r$ set $f(b)=i$ for $b$ the $i$th entry of $y$.
\end{itemize}
By the observations above, this construction defines an element $f\in A$. We have thus defined a map $\alpha(m,z)=f$ from $A$ to $B$.
We will now define a map $\beta\colon A\rightarrow B$ and then show that this is a two-sided inverse for $\alpha$.

Given $f\in A$, set $X_i=f^{-1}(i)$. Since $X_i$ is a SDR for $P_i$, we have $\abs{X_i}=\abs{P_i}$. Thus the first $r$ sets $X_i=\set{x_i}$ are singletons and define a sequence $y=(x_1,\ldots,x_r)\in S^r$. Cut $y$ into pieces of lengths $\abs{P_{r+1}}-2,\ldots,\abs{P_{k}}-2,d$ to define sequences $y_{r+1},\ldots,y_k,z$. Set $m_1=\ldots=m_r=\emptyset$ and define matchings $m_i$ of $S$ for $i>r$ as follows:
\begin{itemize}
    \item Let $t_i$ denote the tree on $P_i$ corresponding to the Pr\"ufer code obtained by replacing the elements of $y_i$ by the parts of $P_i$ that contain them.
    \item Run the Pr\"ufer algorithm for $t_i$. In the $j$th step, for $p\in P_i$ the smallest leftover leaf, add $\set{a,b}$ to $m_i$, where $\set{a}=X_i\cap p$ and $b$ is the $j$th entry of $y_i$.
    \item For the last edge $\set{p_1,p_2}$ of $t_i$, add $\set{a_1,a_2}$ to $m_i$, where $\set{a_j}=p_j\cap X_i$.
\end{itemize}
We define $\beta(f)=(m_1, \ldots, m_k,z)$. From this construction it is clear that applying first $\alpha$ and then $\beta$ is the identity on $A$.
Applying first $\beta$ then $\alpha$ observe that in specifying the last edge for each $m_i$, the two remaining elements will be one in each part not corresponding to a part already removed as a leaf, so this does give a well defined element of $B$ and all other steps are direct inverses.
\end{proof}

\begin{example}
    Let us work through an example of $\alpha$ and $\beta$ described above. We will see in hindsight that
    \cref{eg split up 5.1 eg A part} and \cref{eg split up 5.1 eg B part} were chosen with this in mind.
    
    As noted above, with $S$ interpreted as the specific $h$-neighbourhood of a vertex $u$ as illustrated in \cref{fig:prueffersetup}, the partitions must all have $1$ and $2$ in the same part and likewise for $11$ and $12$. The fact that half-edges $1$ and $2$ have the same incident vertex in the figure, and likewise for $11$ and $12$ has no effect on the maps $\alpha$ and $\beta$ but can occur in the context where we will use \cref{prop:pruefer-bij} and so is useful to illustrate in the example. 
    
    The defect $d=2k-\abs{S}=4$ indicates that in the 16-regular completion of the graph, $u$ has 4 edges connecting it to the deleted vertex $v$ as illustrated in \cref{fig:prueffersetup}.
    
    Take the partitions to be as in \cref{eg split up 5.1 eg A part}, and as illustrated by the coloured blobs in \cref{fig:pruefferwithu} and \cref{fig:prueffermatchings}. Set $z = (3,4,5,6)$ along with the matchings $m_1=m_2=m_3=m_4=m_5 = \emptyset$, 
    \begin{equation*}
    m_6 = \set{\{2,7\}, \{8,9\}},\qquad
    m_7 = \set{\{10,11\}}, \qquad 
    m_8 = \set{\{1,12\}} \qquad
    \end{equation*}
    as in \cref{eg split up 5.1 eg B part}.

    Applying $\alpha$ to this $(m,z)$ we obtain $f$ with $f(1)=f(12)=8, f(2)=1, f(3)=2, f(4)=3, f(5)=4, f(6)=5, f(7)=f(8)=f(9)=6, f(10)=f(11)=7$ which is the $f$ we began with in \cref{eg split up 5.1 eg A part}.
\end{example}

\subsection{Diagonals and the Martin invariant}\label{sec diag and martin}

We apply the results of \cref{sec:pruefer} to show that the diagonal coefficient satisfies the Martin recurrence and hence with suitable normalization agrees with the Martin invariant.

\begin{theorem}\label{thm martin as coeff}
    Let $G$ be a $2k$-regular graph with $n$ vertices. Then for any vertex $v$ without self-loops,
    \begin{equation}\label{eq:diagonal-Martin}
        \takecoeff{x_1\cdots x_{k(n-2)}} \KirchPol_{G\setminus v}^{k}
        = k!\cdot \Martin\left(G\right).
    \end{equation}
\end{theorem}
\begin{proof}
    By \cref{lem comb interp of diag} the left-hand side, $N_k(G\setminus v)$, is the number of ordered partitions $E=T_1\sqcup\ldots\sqcup T_{k}$ of the edge set $E$ of $G\setminus v$ into $k$ edge-disjoint spanning trees $T_i$. 

    If $G$ has a self-loop, then $\Martin(G)=0$ and by hypothesis on $v$, $G\setminus v$ also has a self-loop and so $N_k(G\setminus v)=0$ since no spanning tree can include any self-loop. Therefore, \eqref{eq:diagonal-Martin} holds whenever $G$ has one or more self-loops.

    From now on, assume that $G$ has no self-loops. Then $G$ must have at least two vertices. If $G$ has exactly 2 vertices, then $G \cong K_2^{[k]}$ is a $k$-fold edge with $k!\cdot\Martin(G)=1$, see \eqref{eq:Martin-rose-dipole}, and $G\setminus v$ is a single vertex without edges and thus $N_{k}(G\setminus v)=1$ from the partition $T_1=\ldots=T_k=\emptyset$. So \eqref{eq:diagonal-Martin} holds for $n=2$.\footnote{A direct proof for $n=3$ is also straightforward (though not necessary), since then $G\cong K_3^{[k]}$ with $\Martin(G)=1$ and the claim follows from the $r=1$ case of \cref{ex:diag-dipole-power}.}

    To complete the proof, we will show that the diagonal coefficient $N_{k}(G\setminus v)$ satisfies the Martin recurrence \eqref{eq:Martin-recursion}: Let $w\neq v$ be another vertex of $G$, then we will prove that
    \begin{equation}\label{eq:trees-induct}
        N_{k}(G\setminus v)= \sum_{\tau\in\trans'(w)} N_{k}(G_{\tau}\setminus v)
    \end{equation}
    where $\trans'(w)$ is the subset of transitions at $w$ that does not create a self-loop at $v$. This induction step reduces \eqref{eq:diagonal-Martin} from $n$ to $n-1$ vertices: if $N_k(G_\tau\setminus v)=k!\cdot\Martin(G_\tau)$ is already known, then the right-hand side of \eqref{eq:trees-induct} is $k!\cdot\Martin(G)$. The transitions $\tau\in\trans(w)\setminus\trans'(w)$ missing from \eqref{eq:Martin-recursion} only produce graphs with $\Martin(G_\tau)=0$ due to self-loops at $v$.
    
    We now prove \eqref{eq:trees-induct}. Let $S$ be the $h$-neighbourhood (see \cref{sec martin poly}) of $w$ in $G\setminus v$.

    Let $T_1, \ldots T_{k}$ be a partition of $E$ into spanning trees of $G\setminus v$.
    For each $T_i$, let $X_i=T_i\cap S$ and $X_i'$ be the set of edges incident to $w$ which contain a half-edge of $X_i$.  The remaining edges $F_i = T_i \setminus X_i'$ determine a spanning forest of $G\setminus \{v,w\}$. Such a forest $F_i$ induces a partition $P_i$ of $S$ as follows: two half-edges from the $h$-neighbourhood of $w$ are in the same part of $P_i$ if the corresponding neighbours of $w$ are in the same component of $F_i$. Since $T_i = X_i' \sqcup F_i$ and $X_i$ is in bijection with $X_i'$, $X_i$ includes precisely one half-edge in each part of $P_i$, so it is an SDR of $P_i$. Since the $T_i$ partition $E$, the SDRs $X_1,\ldots,X_k$ partition $S$. Conversely, given any list of forests $F_i$ that induce $P_i$, and any partition of $S$ into SDRs $X_i$ of $P_i$, then each $T_i$ formed by taking $F_i=X_i'\sqcup F_i$ (where $X_i'$ is again the set of edges which contain a half-edge of $X_i$) is a spanning tree of $G\setminus v$ and these trees partition $E$. Therefore, we can write
    \begin{equation*}\label{eq trees and A}\tag{$\sharp$}
    N_{k}(G\setminus v) = \sum_{P} \abs{A(P_1, \ldots, P_k)} \cdot \abs{C(P_1,\ldots,P_k)}
    \end{equation*}
    as a sum over all ordered sequences $P=(P_1, \ldots, P_k)$ of partitions of $S$. Here $A$ is defined in \eqref{eq A} and counts the choices of $X_i$. To count the choices of $F_i$, we define $C(P)$ as the set of lists $(F_1,\ldots,F_k)$ of spanning forests $F_i\subseteq E'$ of $G\setminus\set{v,w}$ such that each $F_i$ induces $P_i$, and such that all $F_i$ together form a partition of the edge set $E'$ of $G\setminus\set{v,w}$.

    Now consider a transition $\tau$ of $G$ at $w$.  This is a matching of the $h$-neighbourhood of $w$ in $G$.  This neighbourhood differs from $S$ by the $d = 2k-|S|$ half-edges $e_1,\ldots,e_d$ incident to $v$. Suppose additionally that $\tau$ does not create any self-loops at $v$. Then let $z=(z_1,\ldots,z_d)$ be the elements of $S$ which are paired in $\tau$ with $e_1,\ldots,e_d$.  Let $m$ be the perfect matching of $S\setminus \set{z_1,\ldots,z_d}$ given by the rest of the transition. We can thus identify the set $\trans'(w)$ with pairs $(m,z)$ of a $d$-tuple $z$ and a matching $m$ of the above form. Note that the elements of $m$ are canonically identified with edges in $G_\tau\setminus v$, such that the edge set of $G_\tau\setminus v$ is the disjoint union of $E'$ and $m$.
    
    Let $T'_1, \ldots T'_{k}$ be a partition of the edges of $G_{\tau}\setminus v$ into spanning trees. Similarly to before, we can decompose each $T'_i=m_i\sqcup F_i$ into the edges $m_i=T'_i\cap m$ that belong to $m$ and a spanning forest $F_i=T_i'\setminus m_i$ of $G\setminus \set{v,w}$. Let $P_i$ be the partition of $S$ induced by $F_i$. Then $m_i$ defines a tree on $P_i$ in the sense explained before \cref{def B}. Conversely, given any forest $F_i$ that induces $P_i$, together with a matching $m_i$ defining a tree on $P_i$, we obtain a spanning tree $T_i'=F_i\sqcup m_i$ of $G_\tau\setminus v$.
    Therefore, we can write
    \begin{equation*}\label{eq trees and B}\tag{$\flat$}
    \sum_{\tau \in \trans'(w)}N_{k}(G_\tau \setminus v) = \sum_{P}\abs{B(P_1, \ldots, P_k)}\cdot \abs{C(P_1,\ldots,P_k)},
    \end{equation*}
    as a sum over sequences $P=(P_1,\ldots,P_k)$ of partitions of $S$. Here, $B(P)$ from \cref{def B} counts the choices of $\tau$ (that is, $m$ and $z$) and also the choice of partition $m=m_1\sqcup\ldots\sqcup m_k$; whereas $C(P)$ counts the spanning forest decompositions $E'=F_1\sqcup\ldots\sqcup F_k$ as before.

    By \cref{prop:pruefer-bij}, \eqref{eq trees and A} and \eqref{eq trees and B} agree and hence we proved \eqref{eq:trees-induct}.
\end{proof}
\begin{corollary}
    For a $2k$-regular graph $G$ and a vertex $v$ without self-loops, $\Martin(G)$ is equal to the number of partitions of the edge set of $G\setminus v$ into $k$ spanning trees. Hence, the number of these partitions is independent of $v$.
\end{corollary}
This combinatorial interpretation, stated as \cref{thm:Martin=ST-partitions} in the introduction, follows from \cref{diag=STP}. It is illustrated for the example $\Martin(K_4)=6$ in \eqref{eq:STP-K4}. The claim remains true even if there are self-loops at $v$, for then $\Martin(G)=0$ but also $N_k(G\setminus v)=0$, because $G\setminus v$ must then have more than $kn$ edges (where $G$ has $n+2$ vertices) and hence does not admit partitions into $k$ spanning trees (each of which contains $n$ edges).
\begin{remark}
From this point of view, the condition $\Martin(G)>0$ is equivalent to the existence of an edge partition of $G\setminus v$ into spanning trees. Necessary and sufficient conditions for the existence of such a partition were given by Nash-Williams \cite{NashWilliams:EdgeDisjointTrees} and Tutte \cite{Tutte:IntoFactors}. In fact, for our case of $2k$-regular graphs $G$, it follows easily from \cite[Lemma~5]{NashWilliams:EdgeDisjointTrees} that $G\setminus v$ has $k$ edge-disjoint spanning trees if and only if $G$ is $2k$-edge connected. This gives an alternative proof of \cref{cor:2kcon-nonzero}, which says that $\Martin(G)>0$ in this case. Also for more general graphs, the maximal number of edge-disjoint spanning trees is related to the edge-connectivity via upper and lower bounds \cite{Gusfield:connEdgeST,Kundu:BoundsNumST}.
\end{remark}
\begin{corollary}\label{martin-power-from-diagonal}
    Let $G$ be a $2k$-regular graph with $n$ vertices, and let $v$ denote a vertex without self-loops. Then for every positive integer $r$,
    \begin{equation}\label{eq:martin-power-diagonal}
        \Martin(G^{[r]}) = \frac{(r!)^{k(n-2)}}{(kr)!} \takecoefff{\big}{x_1^r\cdots x_{k(n-2)}^r} \KirchPol_{G\setminus v}^{kr}.
    \end{equation}
\end{corollary}
\begin{proof}
    Combine \cref{thm martin as coeff} with \cref{duplication-diagonal}.
\end{proof}
\begin{corollary}\label{martin-power-divisibility}
    For a $2k$-regular graph $G$ with $n\geq 3$ vertices and any positive integer $r$, the Martin invariant $\Martin(G^{[r]})$ is an integer that is divisible by $(r!)^{k(n-3)}$.
\end{corollary}
\begin{proof}
    Combine \cref{diag-multinomial-multiple} with \cref{martin-power-from-diagonal}.
\end{proof}
The identity \eqref{eq:martin-power-diagonal} shows that the sequence $\Martin(G^{\bullet})$ of Martin invariants of all duplications $G^{[r]}$ of a fixed graph encodes the diagonal coefficients of all powers of the Kirchhoff polynomial. These coefficients are related to each other by a difference equation:
\begin{corollary}\label{cor p recursive}
    For any $2k$-regular graph $G$, there exists an integer $d>0$ and univariate polynomials $P_0,\ldots,P_d\in\Z[t]$ with $P_d\neq 0$ such that, for all $r\geq 1$,
    \begin{equation}\label{eq:Martin-P-recursive}
        P_0(r)\cdot\Martin\left(G^{[r]}\right)+P_1(r)\cdot\Martin\left(G^{[r+1]}\right)+\ldots+P_d(r)\cdot\Martin\left(G^{[r+d]}\right) = 0.
    \end{equation}
\end{corollary}
\begin{proof}
If a rational function $f(x_1,\ldots,x_m)$ of several variables is non-singular at $x=0$, then it has a Taylor expansion at $x=0$. The diagonal coefficients of this expansion define a univariate formal power series $\Diag(f)=\sum_{r\geq 0} z^r \takecoeff{x_1^r\cdots x_m^r} f$ called \emph{the diagonal} of $f$. Such diagonals are $D$-finite functions \cite{Lipshitz:diagDfin}, hence the sequence of diagonal coefficients is $P$-recursive \cite{Stanley:DiffFinPow} (namely it fulfils a polynomial recurrence).

Therefore, the diagonal coefficients $\Martin(G^{[r]})\cdot(kr)!/(r!)^{k(n-2)}$ of the geometric series $f=1/(1-\KirchPol_{G\setminus v})=\sum_{i\geq 0}\KirchPol^i_{G\setminus v}$ fulfil a recurrence of the form \eqref{eq:Martin-P-recursive}. Multiplying this recurrence with $((r+d)!)^{k(n-2)}/(kr)!$ produces the claimed recurrence for $\Martin(G^{[r]})$ alone, by the transformation $P_i(r)\mapsto P_i(r) \cdot [(r+i+1)\cdots(r+d)]^{k(n-2)} \cdot (kr+1)\cdots(kr+ki)$.
\end{proof}

\subsection{Duality}\label{sec duality}
As an application of the connection between the Martin invariant and diagonal coefficients, we can relate the Martin sequences of graphs that are planar duals of each other.

Suppose that $G$ is a graph with a plane embedding and $G^\dual$ denotes the corresponding dual graph. Then the edges of $G$ are in canonical bijection with the edges of $G^\dual$. Under this correspondence, a set of edges defines a spanning tree in $G$ precisely when its complement defines a spanning tree in $G^\dual$. Therefore, we have
\begin{equation}\label{eq:kirchpol-dual}
    \PsiPol_{G^\dual}
    =\KirchPol_{G}.
\end{equation}
Recall that the Martin invariant of a regular graph $G$ can be computed by the diagonal coefficient of the Kirchhoff polynomial of a decompletion $G\setminus v$. For two regular graphs $G_1$ and $G_2$ with vertex degrees $d_1$ and $d_2$, respectively, a duality $(G_1\setminus v)^\dual \cong G_2\setminus w$ between decompletions implies that $(d_1,d_2)\in\set{(3,6),(4,4),(6,3)}$. This follows from the constraints on the edge- and vertex numbers imposed by regularity and duality. In such cases, \eqref{eq:kirchpol-dual} translates into a relation between the Martin sequences of $G_1$ and $G_2$. An example is shown in \cref{fig:duality}.
\begin{figure}
    \centering
    \begin{tabular}{ccccc}
    $\Graph[0.55]{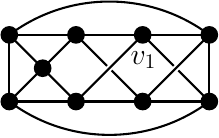}$ & $\Graph[0.45]{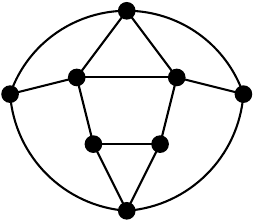}$ &  $\Graph[0.45]{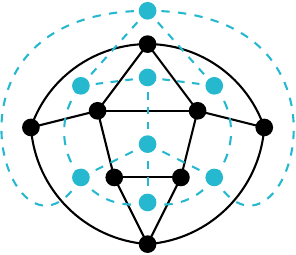}$ & $\Graph[0.45]{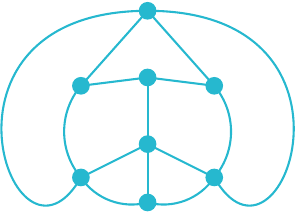}$ & $\Graph[0.55]{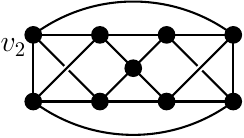}$ \\
    $G_1$ & $G_1\setminus v_1$ & $(G_1\setminus v_1)^\dual\cong (G_2\setminus v_2)$ & $G_2\setminus v_2$ & $G_2$ \\
    \end{tabular}
    \caption{Planar embeddings of decompletions $G_i\setminus v_i$ that are planar duals of each other. The completed graphs are labelled $G_1=P_{7,4}$ and $G_2=P_{7,7}$ in \cite{Schnetz:Census}.}%
    \label{fig:duality}%
\end{figure}
\begin{corollary}\label{prop:phi4-duality-sequence}
    If $G_1$ and $G_2$ are 4-regular graphs such that there exist vertices $v_1$ and $v_2$ without self-loops such that $G_1\setminus v_1$ is a planar dual of $G_2\setminus v_2$, then the Martin sequences $\Martin(G_1^{\bullet})=\Martin(G_2^{\bullet})$ agree.
\end{corollary}
\begin{proof}
    Let $n$ denote the number of vertices of $G_1$. Then $G_1\setminus v_1$ and $G_2\setminus v_2$ each have $n-1$ vertices, $m=2n-4$ edges, and $n-2$ loops. From \eqref{eq:Kirchhoff-Symanzik} and \eqref{eq:kirchpol-dual} we see that
\begin{equation*}
    \takecoeff{x_1^r\cdots x_m^r}\KirchPol_{G_1\setminus v_1}^{2r}
    = \takecoeff{x_1^r\cdots x_m^r}\PsiPol_{G_2\setminus v_2}^{2r}
    = \takecoeff{x_1^r\cdots x_m^r}\KirchPol_{G_2\setminus v_2}^{2r}
\end{equation*}
    and thus the claim follows from \cref{martin-power-from-diagonal}.
\end{proof}
\begin{corollary}
    Let $G_3$ be a 3-regular graph and $G_6$ a 6-regular graph, such that there exist vertices $v$ and $w$, without self-loops, such that $G_3\setminus v$ is a planar dual of $G_6\setminus w$. Let $n$ denote the number of vertices of $G_6$. Then for every positive integer $r$, we have
    \begin{equation}
        \Martin\left(G_3^{[2r]}\right) = \left(\frac{(2r)!}{r!}\right)^{3(n-2)}\Martin\left(G_6^{[r]}\right).
    \end{equation}
\end{corollary}
\begin{proof}
    The graphs $G_3\setminus v$ and $G_6\setminus w$ each have $m=3(n-2)$ edges. Setting $(r,k)$ to $(2r,3/2)$ in \eqref{eq:duplication-diagonal}, we find that
    $\Martin(G_3^{[2r]})\cdot(3r)!=[(2r)!]^m \takecoeff{x_1^{2r}\cdots x^{2r}_m} \KirchPol^{3r}_{G_3\setminus v}$
    from \cref{thm martin as coeff}.\footnote{We may set $k$ to rational values in \cref{duplication-diagonal}, as long as the exponent $rk$ stays positive integer.} From \eqref{eq:Kirchhoff-Symanzik} and \eqref{eq:kirchpol-dual} we see that
    \begin{equation*}
    \takecoeff{x_1^{2r}\cdots x_m^{2r}}\KirchPol_{G_3\setminus v}^{3r}
    = \takecoeff{x_1^{2r}\cdots x_m^{2r}}\PsiPol_{G_6\setminus w}^{3r}
    = \takecoeff{x_1^r\cdots x_m^r}\KirchPol_{G_6\setminus w}^{3r},
    \end{equation*}
    and thus the claim follows from \cref{martin-power-from-diagonal}.
\end{proof}

Through \cref{thm:Martin-Perm2}, the above dualities imply relations between the graph permanents. Those have been noted before, see \cite[Proposition~28]{Crump:ExtendedPermanent}.

\begin{remark}\label{rem:duality-matroids}
    The identity \eqref{eq:kirchpol-dual} extends beyond the realm of planar graphs $G$ to all matroids $M$ and their duals $M^\dual$, if we replace ``spanning tree of $G$'' in the definition of $\KirchPol_G$ and $\PsiPol_G$ by ``basis of $M$''. The \emph{matroid polynomial} $\KirchPol_M$ thus defined is still linear in each variable, and its degree is the rank of $M$. The degree of $\PsiPol_M$ is the corank (formerly loop number) of $M$. The combinatorial interpretation from \cref{lem comb interp of diag} identifies the diagonal coefficient of $\KirchPol_M^k$ with decompositions into bases. The existence of such decompositions has been characterized in \cite[Theorem~1]{Edmonds:PartitionMatroid}.
    For regular matroids, $\KirchPol_M$ is the determinant of a matrix, similar to the graph Laplacian \eqref{eq:Kirchhoff}.

    Consequently, any function of the diagonal coefficient of $\KirchPol_M$ is invariant under matroid duality, for all matroids in the domain of the function---not just cycle matroids of planar graphs. For example, the permanent of regular matroids is invariant under duality. This follows from \cref{sec:perm2-from-det} or the matrix manipulations in \cite{CrumpDeVosYeats:Permanent}, since those apply without change to regular matroids.\footnote{An example of a permanent of a neither graphic nor co-graphic matroid $R_{10}$ is computed in \cite[\S~7.5]{Crump:PhD}.}
    For the $c_2$ invariant, the situation is less clear: its connection to the diagonal coefficient is known only in the presence of a 3-valent vertex, see \cref{rem 3 valent needed}. Moreover, invariance of $c_2$ under duality is not even known for all graphic matroids \cite{Doryn:4face}.
\end{remark}

\begin{remark}\label{sec fourier split}
An identity of Feynman periods called \emph{Fourier split} was introduced in \cite[\S 2]{HuSchnetzShawYeats:Further}. It operates by taking a planar dual on one side of a 3-vertex cut in a decompletion $G\setminus v$. This identity also holds for the Martin invariant. The same proof strategy as in \cref{sec:product+twist} applies: the Martin recursion \eqref{eq:Martin-recursion} at a vertex on the other side of the 3-vertex cut reduces the statement by induction to smaller graphs. The remaining base case is the ordinary duality relation (Fourier identity).

We hence conclude via \cref{thm:Martin-Perm2,thm:c2-martin} that also the extended graph permanent and the $c_2$ invariant at primes respect the Fourier split identity.
\end{remark}

\section{\texorpdfstring{$c_2$}{c2} invariant}\label{sec c2}

For any power $q=p^s$ of a prime $p$, let $\F_q$ denote the finite field with $q$ elements.
Given any polynomial $P\in\Z[x_1,\ldots,x_m]$ in $m$ variables $x_i$, we denote by
\begin{equation*}
    \PointCount{P}{q}=\big|\big\{x\in\F_q^m\colon P(x)=0\big\}\big|
    \in\Z_{\geq 0}
\end{equation*}
the number of points on the hypersurface $\set{P=0}$ over $\F_q$. In the special case $P=\PsiPol_G$ of Symanzik polynomials \eqref{eq:psipol} of graphs, the study of point-counts goes back to \cite{Stanley:SpanningTrees,Stembridge:CountingPoints}. For all graphs with less than 14 edges \cite{Schnetz:Fq}, the point-counting function is a polynomial,
\begin{equation*}\label{eq:pointcount-polynomial}\tag{$\ast$}
    q\mapsto\PointCount{\PsiPol_G}{q}
    = \sum_{i=0}^m c_i(G) q^i.
\end{equation*}
Starting at 14 edges, there exist graphs whose point-count is not polynomial \cite{Doryn:CounterKontsevich,Schnetz:Fq}. In fact, point-counting functions of graph hypersurfaces $\set{\PsiPol_G=0}$ can be as complicated as the point-counting functions of any variety, as made precise and proved in \cite{BelkaleBrosnan:MatroidsMotives}.
\begin{lemma}\label{lem:q2-divides}
    Let $G$ be a graph with at least 3 vertices and at least 2 edges. Then for every prime power $q$, the point count $\PointCount{\PsiPol_G}{q}$ is divisible by $q^2$.
\end{lemma}
\begin{proof}
This is proved in \cite[Corollary~2.8]{Schnetz:Fq} and \cite[Proposition-Definition~18]{BrownSchnetz:K3phi4} when $G$ is connected. If $G$ is disconnected, then $\PsiPol_G=0$ and hence $\PointCount{\PsiPol_G}{q}=q^m$ where $m$ is the number of edges of $G$.
\end{proof}
We will only consider point-counts of graphs with $n\geq 3$ vertices and $m=2(n-1)$ edges, hence \cref{lem:q2-divides} applies and we can write
\begin{equation}\label{eq:c2-congruence}
    \PointCount{\PsiPol_G}{q}\equiv q^2\cdot c_2^{(q)}(G) \mod q^3
\end{equation}
for well-defined residues $c_2^{(q)}(G)\in\Z/q\Z$. The \emph{$c_2$-invariant} \cite{Schnetz:Fq} is the sequence of these residues, one for each prime power, given by\footnote{
When the point-count happens to be a polynomial, then $c_0(G)=c_1(G)=0$ and $c_2^{(q)}(G)\equiv c_2(G)\mod q$ is determined by the coefficient of $q^2$ in \eqref{eq:pointcount-polynomial}.
}
\begin{equation}
    c_2^{(q)}(G) \equiv \frac{\PointCount{\PsiPol_G}{q}}{q^2} \mod q.
\end{equation}
Our main result in this section extracts $c_2$ at primes from the Martin invariant:
\begin{theorem}\label{thm:c2-martin}
    Let $G$ be a $4$-regular graph with at least 6 vertices. Then for every vertex $v$ of $G$, and every prime $p$, we have
    \begin{equation}\label{eq:c2-martin}
        c_2^{(p)}(G\setminus v)\equiv \frac{\Martin(G^{[p-1]})}{3p} \mod p.
    \end{equation}
\end{theorem}
In \cref{sec p not 3} we give a proof for primes $p\neq3$, by relating $c_2$ to the diagonal coefficients of powers of $\PsiPol_G$ as in \cref{sec diags}. In \cref{sec structure,sec 3 inv recurrence,sec base case} we present another proof, which is more technical, but which works for all primes---including $p=3$.
\begin{remark}
The relation \cref{eq:c2-martin} does not extend to prime powers. Only a few relations for $c_2$ are known at prime powers, see \cite{EsipovaYeats:c2powers,SchnetzYeats:hourglass,Doryn:InvariantInvariant}.
\end{remark}
\begin{corollary}
    Let $G$ be a $4$-regular graph $G$ with at least 6 vertices, and let $v,w$ denote two vertices. Then $c_2^{(p)}(G\setminus v)\equiv c_2^{(p)}(G\setminus w)\mod p$ for all primes $p$.
\end{corollary}
This proves the \emph{completion conjecture}, \cite[Conjecture~4]{BrownSchnetz:K3phi4}, for all prime fields. This conjecture is a special case of the stronger \cite[Conjecture~5]{BrownSchnetz:K3phi4} from \cite[Remark~2.11]{Schnetz:Fq}:
\begin{conjecture}\label{conj period c2}
    Whenever $\Period(G_1)=\Period(G_2)$ for two primitive $\phi^4$ graphs, then $c_2^{(q)}(G_1)\equiv c_2^{(q)}(G_2)\mod q$ for all prime powers $q$.
\end{conjecture}
Here a graph $G$ is called \emph{primitive $\phi^4$} if it can be obtained from a cyclically 6-connected 4-regular graph $G'$ by deleting a vertex: $G\cong G'\setminus v$. We then call $G'$ \emph{the completion} of $G$, and we also say that $G$ is \emph{a decompletion} of $G'$, as originally defined in \cref{sec residues intro}.  Every $\phi^4$ graph has a unique completion, but a 4-regular graph can have several non-isomorphic decompletions.

The \emph{completion identity} $\Period(G\setminus v)=\Period(G\setminus w)$ for cyclically 6-connected 4-regular graphs $G$ is only one of several known period identities proved in \cite{Schnetz:Census}. The twist and duality identities lead to further special cases of \cref{conj period c2}. Combining our results on the Martin invariant (\cref{prop:twist} and \cref{prop:phi4-duality-sequence}) with \cref{thm:c2-martin}, we also obtain proofs of the corresponding identities for $c_2$. Specifically, for all primes $p$:
\begin{itemize}
    \item If $G_1$ is a twist of $G_2$, then $c_2^{(p)}(G_1\setminus v)\equiv c_2^{(p)}(G_2\setminus w) \mod p$.
    \item If $G$ is planar, and $G^\dual$ a planar dual, then $c_2^{(p)}(G)\equiv c_2^{(p)}(G^\dual)\mod p$.
\end{itemize}
The duality relation was known previously, even for prime powers \cite{Doryn:InvariantInvariant}. Our twist identity for $c_2$ is a new result, however our proof of the completion identity is arguably more significant. The completion conjecture has been repeatedly stated explicitly \cite{BrownSchnetz:K3phi4, BrownSchnetz:ModularForms} and has had many unsuccessful or only partially successful attempts at a proof.

To apply the main theorem, we use some divisibility properties, which we record here and which are proved in \cref{sec base case}. In \cref{lem:4vert-divp,lem:Martin(G2)div9} we show:
\begin{lemma}\label{lem:Martin-pow-divisible}
    Let $G$ denote a 4-regular graph with at least 4 vertices. Then for every prime $p$, $\Martin(G^{[p-1]})$ is divisible by $p$. For $p=3$ and at least 5 vertices, we also have that $\Martin(G^{[2]})$ is divisible by $3^2$. Hence, the right-hand side of \cref{eq:c2-martin} is well-defined.
\end{lemma}
In practice, for the purpose of computing $c_2^{(p)}$, one can simplify the Martin recursion and ignore all graphs that have a $\geq p$ fold edge:
\begin{lemma}\label{lem no huge edges}
    If $G$ is $4r$-regular with $p=r+1$ prime, and $G$ has at least 6 vertices and at least one edge with multiplicity $p$ or more, then $\Martin(G)$ is divisible by $p^2$ (or $p^3$ in the case $p=3$) and hence such a graph does not contribute to $c_2$.
\end{lemma}
\begin{proof}
    \Cref{lem:pedge-div} shows that applying the Martin recurrence around either end of an edge with multiplicity $\geq p$ gives a factor divisible by $p$ multiplying the Martin invariant of each $G_\tau$. Applying the previous lemma to $\Martin(G_\tau)$ gives the result.
\end{proof}
\begin{remark}
    \begin{figure}
        \centering
        \includegraphics{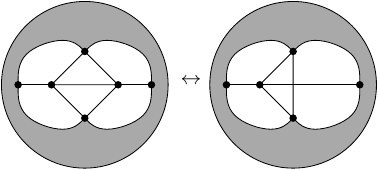}
        \caption{The double triangle operation.}%
        \label{fig:doubletriangle}%
    \end{figure}
    The properties of the Martin invariant and \cref{thm:c2-martin} also explain other $c_2$ identities which are known from \cite{BrownSchnetzYeats:PropertiesC2} and \cite[Corollary~34]{BrownSchnetz:K3phi4}:
\begin{itemize}
    \item If $G$ has a 2-edge cut, a 4-edge cut, or a 3-vertex cut, then $c_2^{(p)}(G\setminus v)\equiv 0 \mod p$.
    \item If $G$ and $G'$ are related as in \cref{fig:doubletriangle}, then $c_2^{(p)}(G)\equiv c_2^{(p)}(G')\mod p$.
\end{itemize}
    The first follows from \cref{prop:edge-cuts} and the factors of $p$ in $\Martin(G^{[p-1]})$ (see \cref{lem:Martin-pow-divisible}) producing $p^2$ in the products \eqref{eq:Martin-edge-product} and \cref{prop:3vertex-cut}.
    The double-triangle relation follows from \cref{lem no huge edges}, since for either vertex incident to the edge shared between two triangles in $G$, after $(p-1)$-duplicating the only transitions that do not lead to self-loops or edges of multiplicity $p$ or larger, are precisely those that produce $G'$.
\end{remark}

We can exploit \cref{lem no huge edges} to simplify the calculation of $c_2^{(p)}$ in recursive families of graphs. A recursive family \cite[\S 5]{ChorneyYeats:c2rec} is in the simplest case a sequence $G_1,G_2,\ldots$ of graphs where $G_{n+1}$ is obtained from $G_{n}$ through a fixed local modification that affects only a finite number of vertices (call these the \emph{active} vertices).
\begin{lemma}\label{martin-families}
For recursive families of $\phi^4$ graphs $G_1,G_2,\ldots$ the Martin invariants $\Martin(G_1^{[r]}),\Martin(G_2^{[r]}),\ldots$ for a fixed integer $r$ fulfil a linear recurrence with constant coefficients.
\end{lemma}
\begin{proof}
    Applying the Martin recurrence to the finite number of vertices in $G_n$ that are involved in moving from $G_{n-1}$ to $G_n$ produces a linear combination of finitely many graphs which have the same recursive structure as $G_{n-1}$ and the same active vertices but with potentially different connections between the active vertices.  The Martin recurrence can likewise be applied to these graphs and so on.  There are only finitely many possibilities for connecting the active vertices so there are only finitely many such graphs to consider. Collecting the Martin invariants of these graphs into a vector $a_n$, we obtain a system of recurrences $a_{n}=Ta_{n-1}$ for some matrix $T$.
\end{proof}
For example, the prisms $K_2\times C_n$ can be constructed using 4 active vertices, by adding one rung at a time. We derive the corresponding recurrence for $\Martin(G^{[2]}_n)$ at the end of \cref{sec:phi3}. For the family of circulants $C_{1,2}^n$ we compute $\Martin(G_n)$ in \eqref{eq:Martin(C12)}.

Through \cref{thm:c2-martin}, the sequence $c_2^{(p)}(G_n\setminus v)$ at a fixed prime fulfils the same linear recurrence as $\Martin(G^{[p-1]})$. \Cref{lem no huge edges} shows that we can drop all graphs with $\geq p$-fold edges from the vector $a_n$, to obtain a shorter recurrence for $c_2$. It could be interesting to compare this approach with the method of \cite{Yeats:c2circulants}, since the latter produces recurrences that are considerably larger than necessary \cite{Yeats:prefixes}. Experiments for circulants $C_{1,3}^n$ and $C_{2,3}^n$ at primes $p=2,3$ suggest that indeed the approach above produces shorter recurrences.

In order to prove \cref{thm:c2-martin}, we relate $c_2$ to diagonal coefficients of some polynomials. One such relation was known previously, and its combinatorial interpretation (\cref{lem c2 as forests and trees}) was key to the results of \cite{Yeats:SpecialCaseCompletion,HuYeats:c2p2}. We will use this interpretation to prove \cref{thm:c2-martin} for all primes, including $p=3$.
For $p\neq 3$, we will also give a simpler proof using the diagonal coefficient of $\PsiPol$.
We will need the following polynomials from \cite{BrownYeats:SpanningForestPolynomials}:
\begin{definition}\label{def sp for}
    Let $G$ be a graph and let $P$ be a partition of a subset of the vertices of $G$. The associated \emph{spanning forest polynomial} is defined to be
    \begin{equation*}
    \Phi_G^P = \sum_{F}\prod_{e\neq F}x_e
    \end{equation*}
    where the sum is over spanning forests $F$ of $G$ which are \emph{compatible} with $P$ in the sense that there is a bijection between the trees of $F$ and the parts of $P$ such that each tree of $F$ contains all vertices that belong to the corresponding part of $P$.
\end{definition}
\begin{lemma}\label{lem c2 as 3 inv coeff}
    Let $G$ be a 4-regular graph with at least 5 vertices and let $v$ and $w$ be two adjacent vertices, both without self-loops. Let $a,b,c$ be the three neighbours of $w$ which are not $v$ and let $x_4, \ldots, x_m$ be the variables associated to the edges of $G\setminus \{v,w\}$. Then
    \begin{equation*}
    c_2^{(p)}(G\setminus v) \equiv - \takecoeff{x_4^{p-1}\cdots x_m^{p-1}}{\left(\Phi^{\{a,b\}, \{c\}}_{G\setminus\{v,w\}}\PsiPol_{G\setminus \{v,w\}}\right)^{p-1}} \mod p
    \end{equation*}
    for all primes $p$.\footnote{
    Such an identity also holds at prime powers $q=p^s$, but only modulo $p$ not $q$ (as detailed in \cite{EsipovaYeats:c2powers}).} Accordingly, the right-hand residue is invariant when permuting $a,b,c$.
\end{lemma}
\begin{proof}
    This is a consequence of a chain of results from \cite{BrownSchnetz:K3phi4} and \cite{BrownSchnetzYeats:PropertiesC2}.  To see them all laid out in detail to obtain the statement of the lemma, see \cite[\S 2]{EsipovaYeats:c2powers}.\footnote{
    The lower bound of 5 vertices is required to apply \cite[Lemma~24]{BrownSchnetz:K3phi4} to $G\setminus v$. With only 4 vertices, the lemma fails. The graph from \cref{rem:dunce-failure} provides a counter-example.}
\end{proof}
\begin{remark}\label{rem:dodgson-denominator}
    With $G$ as in the lemma, and with $1,2,3$ the edges from $w$ to $a,b,c$ respectively, then the two spanning forest polynomials in the lemma can also be written in terms of Dodgson polynomials, which were introduced in \cite{Brown:PeriodsFeynmanIntegrals}. Specifically
    \begin{align*}
        \Phi^{\{a,b\}, \{c\}}_{G\setminus\{v,w\}} & = \PsiPol_{G\setminus v, 3}^{1,2}\\
        \PsiPol_{G\setminus \{v,w\}} & = \PsiPol_{G\setminus v}^{13,23} .
    \end{align*}
    Proving these two facts is one of the steps in the chain of results proving the lemma, coming from results of \cite{BrownSchnetzYeats:PropertiesC2} (see also \cite[Example~2]{EsipovaYeats:c2powers} or \cite[\S2.2]{HuYeats:c2p2} for the calculation exhibited pedagogically).  We will not need to use Dodgson polynomials in the present paper, so we will not give their definition, but may use the two displayed above simply as an alternate notation for the corresponding spanning forest polynomials.
\end{remark}
\begin{definition}\label{def:N-tree-forest}
    For a graph $H$ with 3 vertices marked $a,b,c$ and an integer $r$ define $N_{r,r}(H)$ to be the number of ordered partitions of the set of edges of $H$ into $2r$ parts, the first $r$ of which are spanning trees of $H$, and the last $r$ of which are spanning forests compatible with the partition $\set{\{a,b\}, \{c\}}$.
\end{definition}
The following combinatorial interpretation of $c_2$ was proven and put to productive use in \cite{Yeats:SpecialCaseCompletion,HuYeats:c2p2,EsipovaYeats:c2powers}. It will be used in our second proof of \cref{thm:c2-martin}.
\begin{lemma}\label{lem c2 as forests and trees}
    Let $G$ be as in \cref{lem c2 as 3 inv coeff}.
    Then for every integer $r$,
    \begin{equation*}
        (r!)^{m-3}\cdot \takecoeff{x_4^r\cdots x_m^r}\left(\Phi^{\{a,b\}, \{c\}}_{G\setminus\{v,w\}}\PsiPol_{G\setminus \{v,w\}}\right)^r
        = N_{r,r}(G^{[r]}\setminus\set{v,w})
        .
    \end{equation*}
    Consequently, $c_2^{(p)}(G\setminus v)\equiv N_{r,r}(G^{[r]}\setminus\set{v,w})\mod p$ for every prime $p=r+1$.
\end{lemma}
\begin{proof}
    For the first claim, apply the argument from the proof of \cref{duplication-diagonal} to relate the left-hand side to the diagonal coefficient of the polynomials for the duplicated graph $G^{[r]}\setminus\set{v,w}$; the interpretation as $N_{r,r}$ then follows as in \cref{lem comb interp of diag}, using the observation from \cref{rem Symanzik version of combi argument} that taking the complement of each part does not change the count.
    
    The second claim is then a corollary of \cref{lem c2 as 3 inv coeff} and Wilson's theorem $r!\equiv -1\mod p$, noting that the number $m=2n-4$ of edges in $G\setminus v$ is even (here $n$ is the number of vertices in $G$).
\end{proof}

\begin{example}\label{ex:c2(2)}
To compute $c_2^{(2)}(G)$, the Martin recursion has only 3 terms, which are again 4-regular graphs, hence we get a simple recursion as in \eqref{eq:Martin-recursion-4}:
\begin{equation}\label{eq:c2(2)-rec}
    c_2^{(2)}\Bigg(\Graph[0.4]{circstu0}\Bigg)
    \equiv
    c_2^{(2)}\Bigg(\Graph[0.4]{circstus}\Bigg)
    +c_2^{(2)}\Bigg(\Graph[0.4]{circstut}\Bigg)
    +c_2^{(2)}\Bigg(\Graph[0.4]{circstuu}\Bigg)
    \mod 2
    .
\end{equation}
\end{example}

This recursion of $c_2$ specifically for $p=2$ admits a quick proof using \emph{denominator reduction}, see \cite[\S10]{Brown:PeriodsFeynmanIntegrals}. The \emph{denominator} $D^r_G(e_1,\ldots,e_r)$ of a graph arises as the denominator of the Feynman integral after integrating out the variables corresponding to the edges $e_1,\ldots, e_r$. It is a polynomial in the remaining edge variables and can be defined by purely algebraic means. For example, the second denominator is always a square, $D_G^2(1,2)=P^2$, of a certain (Dodgson) polynomial ($P=\Psi_G^{1,2}$). The third denominator is the product $D_G^3(1,2,3)=(\takecoeff{x_3^1}P)(\takecoeff{x_3^0}P)$, i.e.\ the expression $\Psi_G^{13,23}\Psi_{G,3}^{1,2}$ used in \cref{lem c2 as 3 inv coeff} via \cref{rem:dodgson-denominator}.
Higher denominators can be computed as follows: whenever $D^r_G(e_1,\ldots,e_r)=PQ$ with $r\geq 3$ is a product of two polynomials that are both linear in $x_e$, then $D^{r+1}_G(e_1,\ldots,e_r,e)=[P,Q]_e$ is the resultant $[P,Q]_e=P^eQ_e-P_e Q^e$. We use the shorthand notations $P^e=\takecoeff{x_e^1}P$ and $P_e=\takecoeff{x_e^0}P$ for the coefficients of a polynomial $P$ that is linear in $x_e$.

Denominators starting from the third determine $c_2$ by taking the point count:
\begin{equation*}
    c_2^{(q)}(G)\equiv (-1)^r \PointCount{D^r_G(e_1,\ldots,e_r)}{q} \mod q
\end{equation*}
is shown in \cite[Theorem~29]{BrownSchnetz:K3phi4}. The first case $r=3$ is used in the proof of \cref{lem c2 as 3 inv coeff}.

Let $G$ (left-hand side), $G_s$, $G_t$, $G_u$ be the four graphs in the order illustrated in \eqref{eq:c2(2)-rec}.
As described at the beginning of \cite[\S6.1]{BrownSchnetzYeats:PropertiesC2}, the square roots of the second denominators after reducing the two explicitly drawn edges of $G_s$, $G_t$, and $G_u$ are given by 
$(A-B)$, $(B-C)$,  and $(C-A)$ respectively where $A$, $B$, and $C$ are certain explicit spanning forest polynomials.  Furthermore letting $e$ be any edge in the common part of $G, G_s, G_t$, and $G_u$, by \cite[Lemma~42]{BrownSchnetzYeats:PropertiesC2}, the denominator after reducing the 4 illustrated edges of $G$ along with edge $e$ is equal to
\begin{equation*}
D^5_G=[A,B]_e+[B,C]_e+[C,A]_e.
\end{equation*}
Computing the third denominator for $G_s$ we get $(A-B)^e(A-B)_e = A^eA_e + B^eB_e - A^eB_e - A_eB^e = A^eA_e + B^eB_e - [A,B]_e - 2A_eB^e$.  Adding the third denominators for $G_s$, $G_t$, and $G_u$ we get
\begin{align*}
& 2A^eA_e + 2B^eB_e + 2C^eC_e - [A,B]_e - [B,C]_e - [C,A]_e - 2A_eB^e - 2B_eC^e - 2C_eA^e  \\
& \equiv [A,B]_e + [B,C]_e + [C,A]_e \mod 2.
\end{align*}
So the fifth denominator of $G$ equals modulo $2$ the sum of the third denominators of $G_s$, $G_t$, and $G_u$, but since denominators from the third on up compute $c_2$, by taking the diagonal coefficient (point count $\mod p$) we obtain
\begin{equation*}
c_2^{(2)}(G)
\equiv c_2^{(2)}(G_s) + c_2^{(2)}(G_t) + c_2^{(2)}(G_u) \mod 2
\end{equation*}
as desired.  This provides a very efficient recurrence to compute $c_2^{(2)}$ and gives a short alternative proof of \cite{HuYeats:c2p2}. Discovering this relation between the Martin invariant and $c_2$ at $p=2$ was the starting point that ultimately led to the result below for arbitrary primes.

\subsection{Primes \texorpdfstring{$p\neq 3$}{p<>3} from the diagonal}\label{sec p not 3}

\begin{lemma}\label{lem:c2-from-diag}
    Consider a graph $G$ with $n\geq 4$ vertices and $m=2(n-1)$ edges. Let $p=r+1$ be prime. Then the diagonal coefficient of $\PsiPol_G^{2r}$ fulfils
    \begin{equation}
        \takecoeff{ x_1^r\cdots x_{m}^r } \PsiPol_G^{2r}
        \equiv -3p^2 \cdot c_2^{(p)}(G) \mod p^3.
    \end{equation}
\end{lemma}
\begin{remark}\label{rem:dunce-failure}
\begin{figure}
    \centering
    \includegraphics[height=18mm]{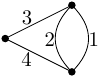}
    \caption{This graph shows the necessity of requiring at least 4 vertices in \cref{lem:c2-from-diag}.}%
    \label{fig:duncecap}%
\end{figure}
    The lemma fails for graphs with less than $4$ vertices:
    The graph $G$ in \cref{fig:duncecap} has $\PointCount{\PsiPol_G}{q}=q^3$ and thus $c_2^{(p)}(G)\equiv 0$, whence the right-hand side is $0\mod p^3$. This point-count of $\PsiPol_G=x_1x_2+(x_1+x_2)(x_3+x_4)$ is obtained using \cite{Stembridge:CountingPoints,Schnetz:Fq}.
    The diagonal coefficient of $\PsiPol_G^{2r}$ on the left-hand side, however, is equal to $\binom{2r}{r}^2\equiv p^2\mod p^3$.
\end{remark}
\begin{proof}
    Pick a 3-valent vertex $w$, let $1,2,3$ be its incident edges, and let $a,b,c$ denote the vertices on the other ends of these edges. Then the spanning forest polynomials denoted
    \begin{equation*}
        f_0 = \PsiPol_{G\setminus w},\quad
        f_1 = \Phi_{G\setminus w}^{\{a\}, \{b,c\}},\quad
        f_2 = \Phi_{G\setminus w}^{\{a,c\},\{b\}},\quad
        f_3 = \Phi_{G\setminus w}^{\{a,b\},\{c\}},\quad
        f_{123} = \Phi_{G\setminus w}^{\{a\},\{b\},\{c\}}
    \end{equation*}
    are subject to $f_0f_{123}=f_1f_2+f_1f_3+f_2f_3$ \cite[Lemma~22]{BrownSchnetz:K3phi4}. Furthermore, we have
    \begin{equation*}
        \PsiPol_G = f_0(x_1x_2+x_1x_3+x_2x_3)+x_1(f_2+f_3)+x_2(f_1+f_3)+x_3(f_1+f_2)+f_{123}.
    \end{equation*}
    This implies $f_0\PsiPol_G= y_1y_2+y_1y_3+y_2y_3$ with $y_i=f_0 x_i+f_i$. Expanding the trinomial,
    \begin{equation*}
        f_0^{2r}\PsiPol_G^{2r}
        = \sum_{i+j+k=2r} \binom{2r}{i,j,k} 
        (f_0 x_1+f_1)^{j+k} (f_0 x_2+f_2)^{i+k} (f_0 x_3+f_3)^{i+j}
    \end{equation*}
    and then extracting the coefficient of $x_1^rx_2^rx_3^r$, we see that if $f_0\neq 0$ then we can cancel $2r$ powers of $f_0$ from each side and get
    \begin{equation*}
        \takecoeff{ x_1^r x_2^r x_3^r }
        \PsiPol_G^{2r}
        = f_0^r \sum_{i+j+k=2r} \binom{2r}{i,j,k} 
        \binom{i+j}{r}\binom{i+k}{r}\binom{j+k}{r}
        f_1^{r-i} f_2^{r-j} f_3^{r-k}.
    \end{equation*}
    This identity remains true when $f_0=0$, because in this case, the left-hand side vanishes because $\Psi_G^{2r}$ has degree $2r$ in the variables $x_1,x_2,x_3$.

    Note that the first binomial coefficient is always divisible by $p$. The second binomial is divisible by $p$ when $k<r$ (so that $r<i+j=2r-k$), and similarly the remaining binomials are divisible by $p$ when $j<r$ and $i<r$, respectively.

    In conclusion, the summand is divisible by $p^3$ unless two summation indices are equal to $r$. Therefore, only the indices $(i,j,k)\in\set{(0,r,r),(r,0,r),(r,r,0)}$ contribute $\mod p^3$:
    \begin{align*}
        \takecoeff{ x_1^r x_2^r x_3^r } \PsiPol_G^{2r}
        &\equiv \binom{2r}{r}^2 \left( (f_0f_1)^r + (f_0f_2)^r + (f_0f_3)^r \right) \mod p^3
        \\
        &\equiv p^2 \left( (f_0f_1)^r + (f_0f_2)^r + (f_0f_3)^r \right) \mod p^3.
    \end{align*}
    Now take the diagonal coefficient $\takecoeff{x_4^r\cdots x_m^r}$ in the remaining variables, on both sides. Since each of the three products $f_0 f_i$ is the polynomial on the right-hand side of \cref{lem c2 as 3 inv coeff} for some permutation of $\set{a,b,c}$, the claim follows.
\end{proof}

This result converts the diagonal coefficient of the more complicated polynomials of \cref{lem c2 as 3 inv coeff} into the diagonal coefficient of $\PsiPol^{2r}$---at least modulo $p^3$. For $p\neq 3$, by using \eqref{eq:Kirchhoff-Symanzik} (or the combinatorial argument of \cref{rem Symanzik version of combi argument}) to move between $\PsiPol$ and $\KirchPol$ and using \cref{duplication-diagonal} to move between $G$ and $G^{[r]}$, we see that \cref{lem:c2-from-diag} and \cref{thm martin as coeff} together give the desired relationship between the Martin invariant and the $c_2$ invariant. This proves \cref{thm:c2-martin} for all primes $p\neq 3$.

Unfortunately, the explicit factor $3$ in \cref{lem:c2-from-diag} does not allow us to draw the same conclusion when $p=3$. In this case we would need the congruence in \cref{lem:c2-from-diag} to hold modulo $p^4$.
Lacking a direct method to establish this congruence, we prove \cref{thm:c2-martin} for $p=3$ using a different approach, detailed in the following subsections.
The essence of this second proof is to avoid the troublesome explicit 3 appearing in \cref{lem:c2-from-diag} by instead recreating the argument used in \cref{thm martin as coeff} directly on the polynomials of \cref{lem c2 as 3 inv coeff} using \cref{lem c2 as forests and trees}.  This second proof applies for all primes, although even there the case $p=3$ requires some special treatment in the base case of the induction.

As a consequence of this second proof, having established \cref{thm:c2-martin} for $p=3$ we can infer indirectly that indeed, for $p=3$ the statement of \cref{lem:c2-from-diag} holds modulo $p^4$, provided that $G$ has at least $8$ edges and $5$ vertices.\footnote{
        For 6 edges, a counter-example is discussed after \cref{lem:Martin(G2)div9}.
}

\begin{remark}\label{rem 3 valent needed}
    The $c_2$ invariant and the diagonal coefficient of $\KirchPol$ are defined for arbitrary graphs (with at least 5 edges and 4 vertices, say). It thus makes sense to ask if the statement of \cref{lem:c2-from-diag} holds more generally. Note that our proof requires the existence of a $3$-valent vertex, which may not exist in more general graphs. Almost nothing is known, however, for $c_2$ invariants for graphs without 3-valent vertices. The only paper making a dent in this direction is \cite{Doryn:4regularc2}.
\end{remark}

\subsection{Structure of the proof for all primes}\label{sec structure}

Our second proof of \Cref{thm:c2-martin} proceeds by showing that the diagonal coefficient $N_{r,r}$ from \cref{lem c2 as forests and trees} is congruent to the Martin invariant. Concretely, we will show that
\begin{equation}\label{eq:Nrr-martin}
    N_{r,r}(G\setminus\set{v,w}) \equiv \frac{\Martin(G)}{3p} \mod p
\end{equation}
for every prime $p=r+1$ and every $4r$-regular graph $G$ with at least 6 vertices, provided $w$ is a vertex with four neighbours $\set{a,b,c,v}$, each connected to $w$ by an $r$-fold multiedge. Applied to the special case where $G$ is the $r$-fold duplication of a 4-regular graph, we obtain \cref{thm:c2-martin} as the consequence of \cref{lem c2 as forests and trees}. This approach works for all primes $p$ and thus fills the gap at $p=3$, which is not covered by the method of \cref{sec p not 3}.

The key step to prove \eqref{eq:Nrr-martin} is to establish the Martin recurrence \eqref{eq:Martin-recursion} for the left-hand side. Like the diagonal coefficient $N_{k}$ of $\PsiPol^{k}$ from \cref{sec diags}---see \eqref{eq:trees-induct}---the diagonal coefficient $N_{r,r}$ also satisfies the Martin recurrence \emph{exactly}: For a vertex $u\notin\set{v,w,a,b,c}$,
\begin{equation}\label{eq:Nrr-recurrence}
    N_{r,r}(G\setminus\set{v,w}) = \sum_{\tau\in\trans'(u)} N_{r,r}(G_{\tau}\setminus\set{v,w})
\end{equation}
is the sum over all transitions at $u$ which do not produce a self-loop at $v$.
The residue modulo $p$ in \eqref{eq:Nrr-martin} is needed only to connect $N_{r,r}$ with $\Martin(G)$ in the base cases (\cref{sec base case}), and with the $c_2$ invariant when $G$ is the $r$-fold duplication of a 4-regular graph.

At its core the idea to prove \eqref{eq:Nrr-recurrence} is the same as for \cref{thm martin as coeff}, namely we interpret a transition as going from sets counted by $A$ \eqref{eq A} to those counted by $B$ (\cref{def B}) and then use the Pr\"ufer based \cref{prop:pruefer-bij}. However, this is more intricate for $N_{r,r}$ than for $N_k$, since we are now counting partitions of the edges into $r$ trees and $r$ forests consisting of two trees each, one containing $\set{a,b}$ and the other containing $\set{c}$. We thus need to keep track of where the marked vertices $\set{a,b,c}$ are. This requires augmenting the constructions of the proof of \cref{thm martin as coeff} by 3 coloured vertices.

As in the proof of \cref{thm martin as coeff}, applying the recurrence \eqref{eq:Martin-recursion} will lead us outside the world of $r$-duplicated graphs and we need to consider arbitrary graphs with the same vertex degrees. To this end, by $H$ we denote a graph with
\begin{itemize}   
    \item two vertices ($a$ and $b$) marked with one colour, say red, and one ($c$) with a second colour, say blue, all three of which have degree at most $3r$, 
    \item all other vertices have degree at most $4r$, and
    \item the number of edges is $r(2n-3)$, where $n\geq 4$ is the total number of vertices of $H$.
\end{itemize}
Then let $G$ be the graph obtained from $H$ by adding a new vertex $w$ with $r$ edges to each of the three coloured vertices of $H$, then adding a new vertex $v$ with $r$ edges to $w$ and edges to each remaining vertex of degree less than $4r$ so as to raise the degree to $4r$, and finally forgetting the colours. The above conditions on $H$ ensure that $G$ is $4r$-regular, $v$ and $w$ have no self-loops, and $H=G\setminus\set{v,w}$. Setting $H_{\tau}=G_{\tau}\setminus\set{v,w}$, we can state \eqref{eq:Nrr-recurrence} as $N_{r,r}(H)=\sum_{\tau} N_{r,r}(H_{\tau})$.

Consequently, the structure of our inductive proof is as follows: we first show the Martin recurrence \eqref{eq:Nrr-recurrence} for the number of partitions of edges of $H$ into $r$ spanning trees and $r$ spanning forests compatible with the colouring in \cref{sec 3 inv recurrence}. Then in \cref{sec base case} we compute these numbers when $H$ has only the three coloured vertices, and compare them with the Martin invariant of the corresponding graph $G$. This completes our second proof of \cref{thm:c2-martin} and in particular proves the theorem in the case $p=3$.

\subsection{Proof of the recurrence for all primes}\label{sec 3 inv recurrence}
Let $H$, $G$, $v$, $w$, and $r$ be as in \cref{sec structure}. In this section, we prove the recurrence \eqref{eq:Nrr-recurrence}.

Recall the definitions of $h$-neighbourhood and $v$-neighbourhood from \cref{sec martin poly}. Throughout this section, let $u$ be an uncoloured vertex of $H$, and  $S$ its $h$-neighbourhood in $H$.

\begin{example}\label{eg 6.3 start}
    We will use as a running example in this section the graph $H=G\setminus\set{v,w}$ illustrated on the left-hand side of \cref{fig 6.3 start}. In this case $G$ is obtained from a doubled 4-regular graph; this is not necessary for the proof, but the application we care about.
\begin{figure}
    \centering
    $H=\Graph{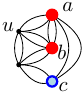} \qquad G=\Graph{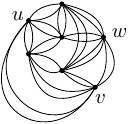}$
    \caption{The graphs for a running example. The large filled vertices, $a$ and $b$, are the two red vertices and the large empty vertex, $c$, is the blue vertex.}%
    \label{fig 6.3 start}%
\end{figure}
\end{example}

Given a partition of the edges of $H$ into $r$ spanning trees and $r$ spanning forests compatible with the vertex colours, consider what happens when we remove $u$.  
As in the proof of \cref{thm martin as coeff}, for each spanning tree, if $u$ has degree $k$ in the tree then upon removing $u$ the tree is split into $k$ trees and this induces a partition of $S$ into $k$ parts.  For each spanning forest, the vertex $u$ is in one of the trees of the forest and if $u$ has degree $k$ in that tree, then upon removing $u$ this tree is split into $k$ trees and the other tree of the forest is unchanged, leaving a spanning forest of $k+1$ trees of $H\setminus u$.  There are then two possibilities for the resulting induced partition of the $h$-neighbourhood $S$ of $u$: if all the vertices of the $v$-neighbourhood of $u$ were originally in the same tree as $u$, then locally around $u$ we only see one tree and so we get a partition of the $h$-neighbourhood $S$ of $u$ into $k$ parts as before.  However, if some of the $v$-neighbourhood of $u$ was originally in the other tree, then there will additionally be a $(k+1)$th part in the induced partition of the $h$-neighbourhood corresponding to the other tree.  In order to keep track of this we will need to introduce coloured partitions in this case.

We will first consider the case where for all the $r$ spanning forests the whole $v$-neighbourhood of $u$ is in the same tree.  Here the local situation is exactly as for \cref{thm martin as coeff} and so the same argument applies.  We now briefly go through this argument again in order to set notation that will be useful when both trees interact around $u$.

Given a partition $P$ of the $h$-neighbourhood of $u$, say a spanning forest is compatible with both $P$ and the vertex colouring if it has $|P|+1$ trees, one of which involves all vertices of one colour and no vertices of $P$ or of the other colour and where the remaining $|P|$ trees are compatible with $P$.
As observed in the proof of \cref{thm martin as coeff}, given a partition $P$ be a partition of the $h$-neighbourhood of $u$, any spanning forest of $H\setminus u$ compatible with $P$ can be extended to a spanning tree $T$ by choosing a system of distinct representative (SDR, as defined in \cref{sec:pruefer}) for $P$ and then adding the edges containing each of these half edges in $H$ to $T$; likewise for any spanning forest of $H\setminus u$ compatible with $P$ and the vertex colouring.  Every spanning tree or spanning forest which induces the partition $P$ on the $h$-neighbourhood of $u$ is obtained in this way.

We call such an SDR a \emph{$u$-extension} of $P$.

For the other side of the Martin recurrence, given a transition $\tau$ of $G$ which does not create self-loops at $v$, 
write $H_\tau$ for the graph obtained by applying the transition restricted to $H$ and removing the unmatched half-edges (those that were matched to $v$ in $G$).
On each $H_\tau$ we are then counting partitions of the edges into $r$ spanning trees and $r$ spanning forests compatible with the coloured vertices.  By breaking the edges matched by $\tau$, each such tree or forest induces a partition on the $v$-neighbourhood of $u$.  For such a spanning forest, if all of the $v$-neighbourhood of $u$ is in the same tree of the forest, then locally we are only working with trees and so as in \cref{thm martin as coeff} for any partition $P$ of the $h$-neighbourhood of $u$ with at most $\deg(u)-2r+1$ parts, any matching that defines a tree on $P$ in the sense of \cref{def B} serves to extend any spanning forest of $H\setminus u$ compatible with $P$ or compatible with $P$ and the vertex colouring.  Furthermore, every such matching comes from at least one transition $\tau$ with no self loops at $v$ since the constraint on the number of parts guarantees enough half edges left to connect to $v$ without loops, and every spanning tree or spanning forest with all of the $v$-neighbourhood of $u$ is obtained in this way.

We call such a choice of matching an \emph{$m$-extension} of $P$.

Let $S$ be the $h$-neighbourhood of $u$ in $H$ and let $P_1, \ldots, P_{2r}$ be partitions of $S$. For $H'=H$ or $H'=H_\tau$ for some transition at $u$, write \[N_{H'}(P_1, \ldots, P_{2r})\] for the number of partitions of the edges of $H'$ into $r$ spanning trees and $r$ spanning forests compatible with the coloured vertices so that the $i$th spanning tree induces the same partition on the $v$-neighbourhood of $u$ that  $P_i$ does and the $j$th spanning forest induces the same partition on the $v$-neighbourhood of $u$ that $P_{r+j}$ does and only one tree of the spanning forest is involved in the induced partition.  
Then what we need from \cref{thm martin as coeff} is the following:

\begin{proposition}\label{prop recurrence 4p-4 all single tree} 
    Let $P_1, \ldots, P_{2r}$ be partitions of the $h$-neighbourhood of $u$ in $H$. Then
    \begin{equation*}
        N_H(P_1, \ldots, P_{2r}) = \sum_\tau N_{H_\tau}(P_1, \ldots, P_{2r})
    \end{equation*}
    where the sum is over transitions $\tau$ of $u$ in $G$ which do not create self loops at $v$. 
\end{proposition}
\begin{proof}
Consider spanning forests $F_i$ compatible with $P_i$ on $H\setminus u$ for each $1\leq i\leq r$ and spanning forests $F_j$ for each $p\leq i \leq 2r$ so $F_j$ has $|P_j|+1$ trees of which $|P_j|$ of them are compatible with $P_j$ and also contain the vertices of one of the two colours and the last tree contains the vertex of the other colour.  Let $C(P_1, \ldots, P_{2r})$ be the set of possible choices for the $F_i$.

If $\sum_{i=1}^{2r}|P_i| = \deg(u)$ and each partition has at most $\deg(u)-2r+1$ parts, then as in \cref{thm martin as coeff}
\begin{equation*}
    N_H(P_1, \ldots, P_{2r}) = \abs{C(P_1, \ldots, P_{2r})}\cdot \abs{A(P_1, \ldots, P_{2r})}
\end{equation*}
and collecting the sets of $m$-extensions by their $z$s,
\begin{equation*}
    \sum_{\tau} N_{H_\tau}(P_1, \ldots, P_{2r}) = \abs{C(P_1, \ldots, P_{2r})}\cdot \abs{ B(P_1, \ldots, P_{2r})}
\end{equation*}
where the sum is over transitions $\tau$ of $u$ in $G$ which do not create self loops at $v$.  Here we have used that for $\tau$ to have no self loops at $v$ we must have $|z|=\deg(u) - \sum_{i=1}^{r}2(|P_i|-1) = 4r-\deg(u)$ which is the same condition as $\sum_{i=1}^{2r}|P_i| = \deg(u)$.
Applying \Cref{prop:pruefer-bij}
\begin{equation*}
N_H(P_1, \ldots, P_{2r}) = \sum_\tau N_{H_\tau}(P_1, \ldots, P_{2r})
\end{equation*}
as desired.

Finally if $\sum_{i=1}^{2r}|P_i| \neq \deg(u)$ or some partition has more than $\deg(u)-2r+1$ parts then both sides are $0$: the left side because the $u$-extensions are SDRs and each $P_i$ has at least one part and the right side because all compatible transitions would have self loops at $v$.
\end{proof}

\begin{example}\label{eg 6.3 no colours}
    Let us continue \cref{eg 6.3 start} with a case where \cref{prop recurrence 4p-4 all single tree} applies.  Write $S=\{1,2,3,4,5,6\}$ for the $h$-neighbourhood of $u$ where half edges $1$ and $2$ are incident to vertex $a$, half edges $3$ and $4$ are incident to vertex $b$, and half edges $5$ and $6$ are incident to the unnamed fifth vertex of $H$.  One possible partition of the edges into two spanning forest compatible with the colouring and two spanning trees is given in \cref{fig 6.3 no colours forest and trees}.  The corresponding partitions of $S$ are $P_1 = P_2 = \{\{1,2,3,4,5,6\}\}$ and $P_3=P_4 = \{\{1,2,5,6\}, \{3,4\}\}$, which we can also visualize locally around $u$ in the same manner as in the examples of \cref{sec:pruefer}, as shown in \cref{fig 6.3 no colours forest and trees locally}.

    \begin{figure}
        \centering
        \includegraphics{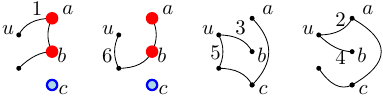}
        \caption{A partition of the edges of $H$ into two spanning forests compatible with the colouring and two spanning trees.}%
        \label{fig 6.3 no colours forest and trees}%
    \end{figure}
    \begin{figure}
        \centering
        \includegraphics[width=0.5\linewidth]{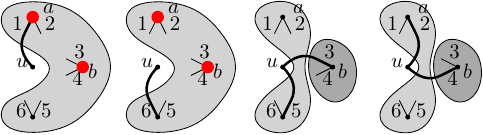}
        \caption{A local view of the partitions around $u$ obtained from \cref{fig 6.3 no colours forest and trees}.}%
        \label{fig 6.3 no colours forest and trees locally}%
    \end{figure}

    These $u$-extensions are given by the function $f\colon 1\mapsto 1, 2\mapsto 4, 3\mapsto 3, 4\mapsto 4, 5\mapsto 3, 6\mapsto 2$.  Then under \cref{prop recurrence 4p-4 all single tree} these $u$-extensions correspond to the matching of \cref{fig 6.3 no colours matching}, where using the notation of the Pr\"ufer bijection we have $y=(1,6)$, $X_3=\{3,5\}$, $X_4=\{2,4\}$.  The transition $\tau$ in this case is the one pairing $\{2,4\}$ and $\{3,5\}$, while $1$ and $6$ each get paired to $v$ in $G$ and so are no longer present in $H_\tau$. The partition of the edges of $H_\tau$ into the corresponding two spanning forests compatible with the colouring and two spanning trees is illustrated in \cref{fig 6.3 no colours after tau}.

    \begin{figure}
        \centering
        \includegraphics{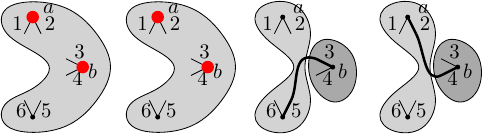}
        \caption{The corresponding matching of the $h$-neighbourhood of $u$.}%
        \label{fig 6.3 no colours matching}%
    \end{figure}
    \begin{figure}
        \centering
        \includegraphics{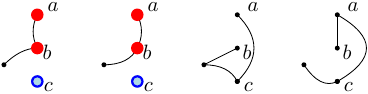}
        \caption{The corresponding partition of the edges of $H_\tau$ into two spanning forests compatible with the colouring and two spanning trees.}%
        \label{fig 6.3 no colours after tau}%
    \end{figure}
\end{example}

Now we extend \cref{prop recurrence 4p-4 all single tree} and the associated set-up to the case where the $v$-neighbourhood of $u$ may involve both trees in one or more of the spanning forests.
For such a spanning forest, the vertices of the $v$-neighbourhood of $u$ which were originally in the same tree as $u$ are now partitioned into $k$ parts, but there may additionally be a part containing vertices in the $v$-neighbourhood of $u$ which are in the other tree of the spanning forest.  Furthermore, if the tree corresponding to the red vertices is the tree which contains $u$, then the two red vertices may be in the same tree after removing $u$ or they may be in two different trees after removing $u$.  Thus, in the case that not all vertices of the $v$-neighbourhood of $u$ are in the same tree of the spanning forest, we need to mark the parts of the partition of the $v$-neighbourhood obtained by removing $u$ according to whether or not the tree associated to that part (after removing $u$) contained a red vertex or a blue vertex.  This will result in one or two red parts, one blue part, and any other parts uncoloured.  The partition of the $h$-neighbourhood is induced from the partition of the $v$-neighbourhood and also inherits the colours.  Call this a \emph{coloured partition}.  

The extensions are now more intricate to describe:
\begin{lemma}\label{lem local forest completion}
Let $P$ be a coloured partition of the $h$-neighbourhood of a vertex $u$ of $H$ (that is $P$ has one or two parts marked as red, one part marked as blue, and the rest uncoloured).  Any spanning forest of $H\setminus u$ compatible with $P$ can be extended to a spanning forest $F$ of $H$ where the red part(s) are in one tree and the blue part in the other tree 
by 
\begin{itemize}
    \item choosing one half edge from each red part of $P$ and each unmarked part of $P$ and then adding the edges containing each of these half edges in $H$ to $F$, or
    \item if there is only one red part in $P$, choosing one half edge from the blue part of $P$ and from each unmarked part of $P$ and then adding the edges containing each of these half edges in $H$ to $F$.
\end{itemize}
Every spanning forest of $H$ which is compatible with the coloured vertices of $H$ and which induces the coloured partition $P$ on the $h$-neighbourhood of $u$ is obtained in this way.
\end{lemma}

We again call such a choice of edges a \emph{$u$-extension} of the coloured partition $P$.  

\begin{proof}
In any case, all uncoloured components must join through $u$ as otherwise we would have 3 or more trees in the final forest.
If there are two red parts then they must join through $u$ in order to be connected and so only the first option is possible.  With one red and one blue part we are free to join either one through $u$.  

For coloured and uncoloured parts, the ways to join through $u$ are as in the tree case.
\end{proof}

Now consider the other side of a Martin-type recurrence.  
Given a transition $\tau$ of $u$ in $G$ which does not create self-loops around $v$ and with $H_\tau$ as before.
A spanning forest of $H_\tau$ again results in a partition of the vertices of $H$ and hence the $h$-neighbourhood of $u$ and if both trees of the forest appear in the $h$-neighbourhood of $u$ then we need to mark the parts of the partition of the $h$-neighbourhood of $u$ red or blue if their tree contained a red or blue vertex after splitting all edges of the matching, resulting in a coloured partition of the $h$-neighbourhood of $u$.  

Again we only need this local information to extend forests. 
\begin{lemma}\label{lem local forest matching}
Let $u$ be an uncoloured vertex of $H$ and let $P$ be a coloured partition of the $h$-neighbourhood of $u$ with at most $\deg(u)-2r+2$ parts.  Any spanning forest of $H\setminus u$ compatible with $P$ can be extended to a spanning forest $F$ of some $H_\tau$ where the red part(s) are in one tree and the blue part in the other tree 
by choosing a forest structure $X$ where the vertex set of $X$ is the set of parts of $P$ and the forest is compatible with the colours coming from the parts of $P$, and then for each end of an edge of $X$ choosing a distinct half edge from the corresponding part of $P$.  

Every spanning forest of every $H_\tau$ which is compatible with the coloured vertices of $H$ and which induces the coloured partition $P$ on the $h$-neighbourhood of $u$ is obtained in this way.
\end{lemma}

We again call such a choice of edges an \emph{$m$-extension} of the coloured partition $P$.

\begin{proof}
As for the tree case but so as to obtain a compatible spanning forest.  The condition on the number of parts likewise makes it always possible for $\tau$ to avoid self-loops at $v$ where the spanning forest structure on $X$ uses one fewer edge than the tree structure did.
\end{proof}

Extend the meaning of $N_H(P_1, \ldots, P_{2r})$ where $P_{p}, \ldots, P_{2r}$ may be coloured to be the number of partitions of the edges of $H$ into $r$ spanning trees and $r$ spanning forests compatible with the coloured vertices so that the $i$th spanning tree induces the same partition on the $v$-neighbourhood of $u$ that $P_i$ does and the $j$th spanning forest induces the same partition on the $v$-neighbourhood of $u$ that $P_{r+j}$ does, with the inherited colouring if multiple trees of the forest interact with the $v$-neighbourhood of $u$.

When some of the $P_i$ may be coloured partitions, we bootstrap off \cref{prop recurrence 4p-4 all single tree}.
\begin{proposition}\label{prop recurrence 4p-4 no other restriction}
    Let $P_1, \ldots, P_{2r}$ be partitions of the $h$-neighbourhood of $u$ where for $p\leq i\leq 2r$ the partition $P_i$ may be a coloured partition.  Then
    \begin{equation*}
        N_H(P_1, \ldots, P_{2r}) = \sum_\tau N_{H_\tau}(P_1, \ldots, P_{2r})
    \end{equation*}
    where the sum is over transitions $\tau$ of $u$ in $G$ which do not create self loops at $v$.
\end{proposition}

\begin{proof}
    \Cref{prop recurrence 4p-4 all single tree} gives the result when all of the partitions are uncoloured and so we now reduce to that case.

    With $P_i$ as in the statement, let $A(P_1, \ldots, P_{2r})$, generalizing \eqref{eq A}, denote the set of $u$-extensions for each $P_i$ (in the coloured sense where appropriate) which partition the edges of $H$ around $u$.  Likewise, let $B(P_1, \ldots, P_{2r})$, generalizing \cref{def B}, denote the set of $m$-extensions for each $P_i$ (in the coloured sense where appropriate) for which the union of the matchings gives the matching of $\tau$ on $S$.

    We will show that 
    \begin{equation}\label{eq A and B in coloured case}
    |A(P_1, \ldots, P_{2r})| = |B(P_1, \ldots, P_{2r})|.
    \end{equation}
    Multiplying by the number of possible choices for the forests in $H\setminus u$ which give $P_i$ will then give the result.
    
    Suppose $P_i$ is a coloured partition with one red part and one blue part. For any $u$-extension of this coloured partition we take either an edge from the red part or an edge from the blue part and that choice is independent of the choices for the other parts.  Therefore the $u$-extensions of the coloured partition $P_i$ are the same as the $u$-extensions of the uncoloured partition resulting from merging the red and blue parts of $P_i$ and then forgetting the colours.  Similarly for any $m$-extension of $P_i$, we obtain a spanning forest on $P_i$ with two trees, one containing the vertex associated to the red part and one containing the vertex associated to the blue part.  Merging the red and blue parts then the same edge and end information gives an $m$-extension of the uncoloured partition obtained from merging the red and blue parts and forgetting the colour and every such tree becomes a forest and a valid $m$-extension upon splitting this part to reobtain $P_i$.  If $P_i$ had no more than $\deg(u)-2r+2$ parts  before merging then it has no more than $\deg(u)-2r+1$ parts after, and vice versa.  
    
    Therefore, for the purposes of \eqref{eq A and B in coloured case}, whenever there is a coloured partition with one red and one blue part, merge those two parts and proceed with an uncoloured partition instead.
    
    It remains to consider partitions with two red parts and one blue part. Suppose $P_i$ is such a partition and let $p$ be the blue part and $q$ and $s$ the two red parts.  Similarly to the above, every $u$-extension of $P_i$ is also a $u$-extension after merging $p$ and $q$ and forgetting the colours, and is also a $u$-extension after merging $p$ and $s$ and forgetting the colours.  However, merging $p$ and $q$ also allows $u$-extensions which use an edge from $p$ and an edge from $s$, and merging $p$ and $s$ 
    also allows $u$-extensions which use an edge from $p$ and an edge from $q$.  These are exactly the $u$-extensions obtained from merging $q$ and $s$ and forgetting the colours.  Therefore, indicating the merged parts as superscripts, for example $P_i^{p\cup q}=(P_i\setminus\set{p,q})\cup\set{p\cup q}$, we see that
    \begin{multline*}
       2 |A(P_1, \ldots, P_i, \ldots, P_{2r})| \\
       = |A(P_1, \ldots, P_{i}^{p\cup q}, \ldots, P_{2r})| + |A(P_1, \ldots, P_{i}^{p\cup s}, \ldots, P_{2r})|
       - |A(P_1, \ldots, P_{i}^{q\cup s}, \ldots, P_{2r})|
    \end{multline*}
    where the counts on the right-hand side involve one coloured partition fewer than on the left-hand side.  
    
    The situation is similar for the $m$-extensions.  Every $m$-extension of $P_i$ is also an $m$-extension after merging $p$ and $q$ and forgetting the colours, and is also an $m$-extension after merging $p$ and $s$ and forgetting the colours.  However, merging $p$ and $q$ also allows $m$-extensions where $s$ is in the tree with $p$ after re-splitting $p$ and $q$, and merging $p$ and $s$ 
    also allows $m$-extensions where $s$ is in the tree with $q$ after re-splitting $p$ and $s$.  These are exactly the $m$-extensions obtained from merging $q$ and $s$ and forgetting the colours.  The bounds on the numbers of parts correspond as before. Therefore we have
    \begin{multline*}
       2 |B(P_1, \ldots, P_i, \ldots, P_{2r})| \\
       =|B(P_1, \ldots, P_{i}^{p\cup q}, \ldots, P_{2r})| + |B(P_1, \ldots, P_{i}^{p\cup s}, \ldots, P_{2r})|- |B(P_1, \ldots, P_{i}^{q\cup s}, \ldots, P_{2r})|
    \end{multline*}
    which lines up term by term with what was obtained for the $u$-extensions.  
    
     Similarly expanding for each remaining coloured partition, we can write $A(P_1, \ldots, P_{2r})$ and $B(P_1, \ldots, P_{2r})$ as identical signed sums involving only uncoloured partitions of $S$.  Therefore the expansions agree term by term by \cref{prop:pruefer-bij}, proving \eqref{eq A and B in coloured case} and hence proving the proposition.  
\end{proof}
Summing both sides of \cref{prop recurrence 4p-4 no other restriction} over all choices of the partitions $P_i$, we obtain the Martin recursion $N_{r,r}(H)=\sum_\tau N_{r,r}(H_{\tau})$ as claimed in \cref{eq:Nrr-recurrence}.

\begin{example}\label{eg 6.3 with colours}
    Let us continue \cref{eg 6.3 no colours} with a case where \cref{prop recurrence 4p-4 no other restriction} applies.  One possible partition of the edges into two spanning forest compatible with the colouring and two spanning trees for which \cref{prop recurrence 4p-4 no other restriction} is needed is given in \cref{fig 6.3 with colours forest and trees}.  The corresponding partitions of $S$ are $P_1 = \{\{1,2,3,4\}_{\text{red}},\{5,6\}_{\text{blue}}\}$, $P_2 = \{\{1,2\}_{\text{red}},\{3,4\}_{\text{red}},\{5,6\}_{\text{blue}}\}$, $P_3=\{\{1,2\},\{3,4,5,6\}\}$, and $P_4 = \{\{1,2,3,4,5,6\}\}$, with the colours indicated by the subscripts. We can visualize this situation locally around $u$ as shown in \cref{fig 6.3 with colours forest and trees locally}.

    \begin{figure}
        \centering
        \includegraphics{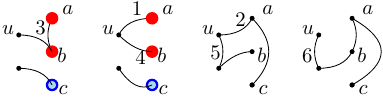}
        \caption{Another partition of the edges of $H$ into two spanning forests compatible with the colouring and two spanning trees.}%
        \label{fig 6.3 with colours forest and trees}%
    \end{figure}
    \begin{figure}
        \centering
        \includegraphics[width=0.5\linewidth]{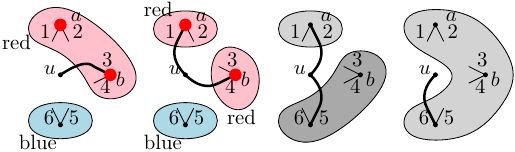}
        \caption{A local view of the partitions around $u$ obtained from \cref{fig 6.3 with colours forest and trees}.}%
        \label{fig 6.3 with colours forest and trees locally}%
    \end{figure}

    Now we need to convert to uncoloured partitions in order to reduce to cases where \cref{prop recurrence 4p-4 all single tree} applies.  Consider first $P_1$. As described in the proof of \cref{prop recurrence 4p-4 no other restriction}, in this case we can simply merge the red and blue parts obtaining in this case the partition with only one part.  $P_2$ is more complicated.  Label the parts as follows $P_2 = \{p,q,s\}$ with $s=\{1,2\}_{\text{red}}$, $q=\{3,4\}_{\text{red}}$, $p=\{5,6\}_{\text{blue}}$.  Considering both the uncoloured partitions with $p\cup q$ and $p\cup s$ we see that the spanning trees and forests we started with are consistent with both, as illustrated in \cref{fig 6.3 p q and p s}. Note that every $u$-extension of the coloured partition of \cref{fig 6.3 with colours forest and trees locally} also appears as a $u$-extension for each of the partitions of \cref{fig 6.3 p q and p s}.  However, some $u$-extensions of each of the partitions of \cref{fig 6.3 p q and p s} do not give $u$-extensions of \cref{fig 6.3 with colours forest and trees locally}.  For example taking edges $1$ and $6$ satisfies the requirements for $P_2$ with $p\cup q$ as in the upper set of partitions of \cref{fig 6.3 p q and p s} but not in \cref{fig 6.3 with colours forest and trees locally}.  However, $1$ and $6$ does work upon taking $q\cup s$.  This is an example of the general fact that $q\cup s$ gives exactly the terms that need subtracting, as described in the proof of \cref{prop recurrence 4p-4 no other restriction}.

    \begin{figure}
        \centering
        \includegraphics[width=0.5\linewidth]{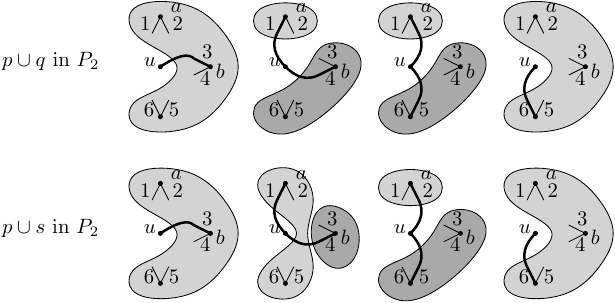}
        \caption{The partitions around $u$ after merging as described in \cref{prop recurrence 4p-4 no other restriction}.}
        \label{fig 6.3 p q and p s}
    \end{figure}
    
    \begin{figure}
        \centering
        \includegraphics[width=0.5\linewidth]{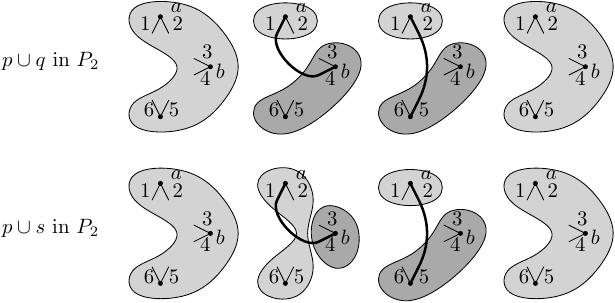}
        \caption{The corresponding matchings.}
        \label{fig 6.3 p q and p s matchings}
    \end{figure}   
    
    Under \cref{prop recurrence 4p-4 all single tree}, the two SDRs of \cref{fig 6.3 p q and p s} correspond to the two matchings of \cref{fig 6.3 p q and p s matchings}.  The transition $\tau$ in both cases is the one pairing $\{1,4\}$ and $\{2,5\}$ while $3$ and $6$ each get paired to $v$ in $G$ and so are no longer present in $H_\tau$. The argument to reduce coloured partitions to partitions by merging parts and subtracting for overcounting is the same as it was in $H$.  In this case we obtain the matching of the coloured partitions illustrated in \cref{fig 6.3 with colours matching} and hence the partition of the edges of $H_\tau$ into two spanning forests compatible with the colouring and two spanning trees as illustrated in \cref{fig 6.3 with colours after tau}.

    \begin{figure}
        \centering
        \includegraphics{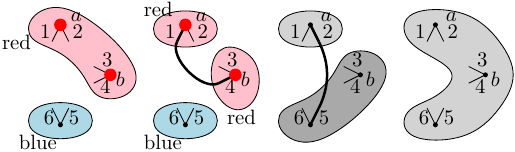}
        \caption{The corresponding matching of the $h$-neighbourhood of $u$.}%
        \label{fig 6.3 with colours matching}%
    \end{figure}
    
    \begin{figure}
        \centering
        \includegraphics{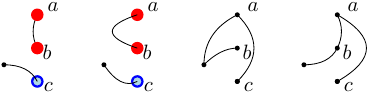}
        \caption{The corresponding partition of the edges of $H_\tau$ into two spanning forests compatible with the colouring and two spanning trees.}%
        \label{fig 6.3 with colours after tau}%
    \end{figure}
\end{example}

\subsection{Base cases}\label{sec base case}

Recall that $p=r+1$ is prime.
Given a graph $H$ with three marked vertices called $\set{a,b,c}$, with $(2|V(H)|-3)r$ edges, and with no vertex of degree more than $4r$, recall from \cref{sec structure} that $H$ determines a $4r$-regular graph $G$ by adding two vertices, labelled $\set{v,w}$, adding $r$ edges between $w$ and each of $\set{a,b,c,v}$, and finally adding $3r$ edges between $v$ and the original vertices of $H$ in the unique way such that the resulting graph $G$ is $4r$-regular. With this construction, $H=G\setminus\set{v,w}$.

Recall that $N_{r,r}(H)$ from \cref{def:N-tree-forest} counts the ordered partitions of the edges of $H$ into $r$ spanning trees and $r$ spanning forests compatible with the partition $\{a,b\}, \{c\}$.

For the base case of our inductive proof, we need to show that
\begin{equation}\label{eq:c2-basecase}
    N_{r,r}(H) \equiv \frac{\Martin(G)}{3p} \mod p
\end{equation}
for all marked graphs $H$ as above on a fixed small number of vertices. The argument is easier for $G$ and $H$ smaller, and so for $p\neq 3$ we will prove \eqref{eq:c2-basecase} for $H$ with 3 vertices (only the marked vertices) and hence $G$ with 5 vertices.  For $p=3$ we will need to move to $H$ with 4 vertices and hence $G$ with $6$ which remains sufficient for \cref{thm:c2-martin}.

Suppose, then, that $H$ is a marked graph with only 3 vertices---namely, the marked vertices---and $3r$ edges, and $G$ as above. 

Consider first the left-hand side.
\begin{lemma}\label{lem:TFP-3}
    For any graph $H=G\setminus\set{v,w}$ with $3$ vertices and $3r$ edges and $p=r+1$ prime and $G$ as explained 
    above, we have
    \begin{equation}
        N_{r,r}(H) \equiv \begin{cases}
            -1 \mod p, & \text{if $G\cong K_5^{[r]}$, and} \\
            \phantom{-}0 \mod p, & \text{otherwise.} \\
        \end{cases}
    \end{equation}
\end{lemma}
\begin{proof}
Any self-loop in $H$ cannot belong to any tree or forest, so the presence of self-loops forces $N_{r,r}(H)=0$. Thus suppose that $H$ has no self-loops. Let $m_{ij}$ denote the number of edges between vertices $i$ and $j$. The 2-forests of $H$ that separate $\set{a,b}$ from $\set{c}$ consist of a single edge between $a$ and $b$. Hence there are precisely $\binom{m_{ab}}{r}$ ways to assign $r$ of these edges to 2-forests. The remaining edges are to be grouped into $\alpha$ spanning trees of the form $\set{ab,ac}$, $\beta$ spanning trees $\set{ab,bc}$, and $\gamma$ spanning trees $\set{ac,bc}$. The constraints $m_{ab}-r=\alpha+\beta$, $m_{ac}=\alpha+\gamma$, and $m_{bc}=\beta+\gamma$ have the unique solution $\alpha=r-m_{bc}$, $\beta=r-m_{ac}$, $\gamma=2r-m_{ab}$ where we note $\alpha+\beta+\gamma=r$. After choosing $\alpha$ out of the $m_{ab}-r$ edges $ab$ and $\alpha$ out of the $m_{ac}$ edges $ac$, there are $\alpha!$ ways to pair them up into spanning trees of the first type, etc. Finally, accounting for the $r!$ ways to order the spanning trees and 2-forests, we find
\begin{equation}
    N_{r,r}(H)
    = r! \binom{m_{ab}}{r} r! \alpha! \beta! \gamma! \binom{m_{ab}-r}{\alpha}\binom{m_{ac}}{\alpha}\binom{m_{bc}}{\beta}
    = \frac{r!\cdot m_{ab}!\cdot m_{ac}!\cdot m_{bc}!}{\alpha! \beta! \gamma!}
    .
\end{equation}
We need $m_{ab}\geq r$ for $N_{r,r}(H)$ to be non-zero. Then $\gamma\leq r$ and we always have $\alpha,\beta\leq r$, hence the denominator is not divisible by $p$. So the numerator, and hence $N_r(H)$, become divisible by $p$ if any $m_{ij}>r$. For this not to happen, we need each $m_{ij}\leq r$. But then we must have $m_{ab}=m_{ac}=m_{bc}=r$, due to the constraint $m_{ab}+m_{ac}+m_{bc}=3r$. Then all edges in $G$ are $r$-fold, and we get $N_{r,r}(H)=(r!)^3 \equiv -1 \mod p$ by Wilson's theorem.
\end{proof}

We now turn to the right-hand side of \eqref{eq:c2-basecase}.

Consider the effect of multiedges on the expansion \eqref{eq:Martin-recursion} of the Martin invariant at a vertex $u$ without self-loops. Let $w_1,\ldots,w_s$ denote the neighbours of $u$, and say that there are $d_i$ edges $uw_i$ between $u$ and $w_i$. Then $G_{\tau}$ depends only on the numbers $D_{ij}$ of pairs $\set{uw_i,uw_j}$ in $\tau$. Hence different transitions with the same numbers $D_{ij}$ will repeat the same contributions $\Martin(G_{\tau})$.
\begin{definition}\label{def:trans-matrix}
    Let $D$ be a symmetric matrix $D$ of non-negative integers with diagonal $D_{ii}=0$ and row sums $\sum_j D_{ij}=d_i$ equal to the multiplicities of the edges between $u$ and its neighbours $w_i$. We write $G_D=G_{\tau}$ for any transition $\tau\in\trans(u)$ with precisely $D_{ij}$ matches $\set{uw_i,uw_j}$ between any pair of neighbours. We also introduce the notation
    \begin{equation*}
        d!=d_1!\cdots d_s!
        \quad\text{and}\quad
        D!=\prod_{1\leq i<j\leq s} (D_{ij}!).
    \end{equation*}
\end{definition}
We only consider zero diagonals, because $D_{ii}>0$ would produce self-loops in $G_D$, which thus do not contribute to $\Martin(G)$. The row sum constraint arises because a transition has to pair up all half-edges at $u$.
\begin{lemma}
    Let $G$ be a regular graph without self-loops and $u$ any vertex. Let $d=(d_1,\ldots,d_s)$ denote the edge multiplicities at $u$. Then
    \begin{equation}\label{eq:Martin-multi-expansion}%
        \Martin(G) = \sum_{D} \frac{d!}{D!} \Martin(G_D)
    \end{equation}
    where the sum is over all matrices $D$ as in \cref{def:trans-matrix}. The coefficients are integers,
    \begin{equation}\label{eq:multi-expansion-coeff}%
        \frac{d!}{D!}
        = D! \prod_{i=1}^s \binom{d_i}{D_{i1},\ldots,D_{is}}.
    \end{equation}
\end{lemma}
\begin{proof}
    With the notation introduced above, this is just a rewriting of the vertex expansion \eqref{eq:Martin-recursion}, including only transitions that do not produce self-loops in $G_{\tau}$. The coefficient of $\Martin(G_D)$ is the number of transitions with the number of pairings prescribed by $D$. For every neighbour $w_i$, a transition partitions the $d_i$ edges $uw_i$ into a part of size $D_{i1}$ (which is paired with an edge $uw_1$), a part of size $D_{i2}$ (to be paired with $uw_2$), and so on. These partitions are counted by the multinomial coefficients in \eqref{eq:multi-expansion-coeff}. The remaining factors $D_{ij}!$ in $D!$ count the ways to pair the thus chosen $D_{ij}$ copies of $uw_i$ with the $D_{ji}=D_{ij}$ chosen copies of $uw_j$ into the new edges $\set{uw_i,uw_j}$ in $G_D$.
\end{proof}
\begin{example}
    Let $K_4^{[\alpha,\beta,\gamma]}$ denote the graph on 4 vertices $\set{1,2,3,4}$ with $m_{12}=m_{34}=\alpha$, $m_{13}=m_{24}=\beta$, and $m_{14}=m_{23}=\gamma$, where $m_{ij}$ denotes the number of edges between vertices $i$ and $j$. This is a regular graph of degree $\alpha+\beta+\gamma$. Suppose that $\alpha+\beta+\gamma=2k$ is even. Then there is at most a single summand in \eqref{eq:Martin-multi-expansion}, because a matrix $D$ for the transitions at vertex $4$ is uniquely determined by the row sums $(d_1,d_2,d_3)=(\gamma,\beta,\alpha)$; it must be $D_{12}=k-\alpha$, $D_{13}=k-\beta$, $D_{23}=k-\gamma$. Therefore
    \begin{equation}\label{eq:Martin(K4abc)}
        \Martin\left(K_4^{[\alpha,\beta,\gamma]}\right)
        =\begin{cases}
            \frac{\alpha!\beta!\gamma!}{(k-\alpha)!(k-\beta)!(k-\gamma)!} & \text{if $\alpha,\beta,\gamma\leq k$ and} \\
            0 & \text{otherwise.} \\
        \end{cases}
    \end{equation}
\end{example}
\begin{lemma}\label{lem:pedge-div}
    If $p$ is a prime and $u$ a vertex with a neighbour $w_i$ with edge multiplicity $d_i\geq p$, then for every summand in \eqref{eq:Martin-multi-expansion}, the coefficient $d!/D!$ is divisible by $p$.
\end{lemma}
\begin{proof}
    If $D$ has an entry $D_{ij}\geq p$, the divisibility follows from the $D!$ factor on the right-hand side of \eqref{eq:multi-expansion-coeff}. So assume that all entries of $D$ are less than $p$. Then $D!$ is not divisible by $p$, but $d_i!$, and thus $d!/D!$, is.
\end{proof}
\begin{corollary}\label{lem:4vert-divp}
    For every $4r$-regular graph $G$ on at least 4 vertices, with $p=r+1$ prime, $\Martin(G)$ is divisible by $p$.
\end{corollary}
\begin{proof}
    If $G$ has self-loops, then the claim is trivial by $\Martin(G)=0$. For graphs with more than 4 vertices, the claim follows inductively by vertex expansion \eqref{eq:Martin-recursion} to smaller graphs. Hence we only need to consider $G$ with $4$ vertices and no self-loops, that is, $G\cong K_4^{[\alpha,\beta,\gamma]}$ for some integers $\alpha,\beta,\gamma$ with $\alpha+\beta+\gamma=4r$. Since at least one of $\set{\alpha,\beta,\gamma}$ must exceed $r=p-1$, \cref{lem:pedge-div} implies the claim.
\end{proof}
\begin{corollary}\label{lem:Martin(5vert)p2}
    Let $p=r+1$ be prime. Then for every $4r$-regular graph $G$ with 5 vertices, which is not isomorphic to $K_5^{[r]}$, we have that $\Martin(G)$ is divisible by $p^2$.
\end{corollary}
\begin{proof}
    If $G$ has any edge of multiplicity greater than $r$, expansion at a vertex with such an edge writes $\Martin(G)$ as a linear combination of $(d!/D!) \Martin(G_{D})$. The factor $d!/D!$ is divisible by $p$ due to \cref{lem:pedge-div}, and $\Martin(G_{D})$ is divisible by $p$ due to \cref{lem:4vert-divp}. So we conclude that $\Martin(G)$ is divisible by $p^2$ whenever $G$ has an edge with multiplicity $\geq p$. So, assume that all edges have multiplicity $\leq r$. To reach degree $4r$ at every vertex, we must then have $G\cong K_5^{[r]}$.
\end{proof}
\begin{lemma}\label{lem:Martin(K5r)}
    The Martin invariant of the $r$-fold power of the complete graph $K_5$ is
\begin{equation}\label{eq:Martin(K5r)}
    \Martin\left(K_5^{[r]}\right)
    =(r!)^4 \sum_{\alpha+\beta+\gamma=r} \frac{(r+\alpha)!(r+\beta)!(r+\gamma)!}{(\alpha!\beta!\gamma!)^2(r-\alpha)!(r-\beta)!(r-\gamma)!}
\end{equation}
and congruent to $\Martin(K_5^{[r]})\equiv -3p\mod p^2$ whenever $p=r+1$ is prime.
\end{lemma}
\begin{proof}
The expansion \eqref{eq:Martin-multi-expansion} at any vertex produces
\begin{equation*}
    \Martin\left(K_5^{[r]}\right)
    = \sum_{\alpha+\beta+\gamma=r} \frac{r!r!r!r!}{(\alpha!\beta!\gamma!)^2} \Martin\left(K_4^{[r+\alpha,r+\beta,r+\gamma]}\right).
\end{equation*}
Inserting \eqref{eq:Martin(K4abc)} gives \eqref{eq:Martin(K5r)}. Now consider $p=r+1$ prime. The denominator of the summand in \eqref{eq:Martin(K5r)} is not divisible by $p$. The numerator is divisible by $p^2$ unless at most one of $\alpha,\beta,\gamma$ is non-zero. We conclude that modulo $p^2$, only the three summands with $(\alpha,\beta,\gamma)\in\set{(0,0,r),(0,r,0),(r,0,0)}$ contribute. Therefore,
\begin{equation*}
    \Martin\left(K_5^{[r]}\right)
    \equiv 3(r!)^2(2r)!
    = 3(r!)^3 p(p+1)\cdots(2p-2)
    \equiv -3p\cdot (r!)^4
    \mod p^2
\end{equation*}
and the claim follows from Wilson's theorem $r!\equiv -1\mod p$.
\end{proof}

Note how the three explicit summands contributing in the last step of the proof are what leads to the explicit factor of $3$, causing extra trouble for $p=3$.

\begin{corollary}
    For every $4r$-regular graph with 5 vertices and $r+1$ prime, we have $\Martin(G)\equiv 3p N_{r,r}(H) \mod p^2$.
\end{corollary}
\begin{proof}
    For every graph other than $K_5^{[r]}$ the claim follows from \cref{lem:Martin(5vert)p2} and \cref{lem:TFP-3}. For $G=K_5^{[r]}$, see \cref{lem:Martin(K5r)} and \cref{lem:TFP-3}.
\end{proof}

This establishes the base case \eqref{eq:c2-basecase} for all primes $p\neq 3$. Now consider the special case $p=3$. Then \eqref{eq:c2-basecase} is equivalent to
\begin{equation}\label{eq:c2-base-p=3}
    \Martin(G)\equiv 9 N_{2,2}(H) \mod 27,
\end{equation}
that is, we have to compute the Martin invariants $\mod p^3$. We first note that the expression $\Martin(G)/9\mod 3$ on the right-hand side of \eqref{eq:c2-basecase} makes sense.
\begin{lemma}\label{lem:Martin(G2)div9}
    The Martin invariant of every 8-regular graph with at least 5 vertices is divisible by $9$.
\end{lemma}
\begin{proof}
    By vertex expansion \eqref{eq:Martin-recursion}, the claim reduces to the case of graphs $G$ with $5$ vertices. If $G=K_5^{[2]}$, then $\Martin(G)\equiv -3p\equiv 0 \mod p^2$ for $p=3$. For all other graphs, see \cref{lem:Martin(5vert)p2}.
\end{proof}
We thus have $\Martin(G)\equiv 9j\mod 27$ for some integer $j$, but for graphs with $5$ vertices, $j$ is not necessarily congruent to $N_{2,2}(H)$. For example,
the totally decomposable graph
\begin{equation*}
    \Martin\left(\Graph[0.55]{K5plus2}^{[2]}\right) = (4!)^2 \equiv 9 \mod 27
\end{equation*}
has $j\equiv 1 \mod 3$; but we saw in \cref{lem:TFP-3} that $N_{2,2}(H)\equiv 0 \mod 3$.
Therefore, for prime $p=3$ and $5$ vertices, the $c_2$ invariant cannot be read off the Martin invariant modulo $p^3$. We need to base the induction one step up, at 6 vertices.

So consider an $8$-regular graph $G$ with 6 vertices. Recall that $G$ has 5 special vertices labelled $\set{v,w,a,b,c}$ with $H=G\setminus\set{v,w}$, such that $w$ is connected only to $\set{a,b,c,v}$, via a double edge to each. It follows that the sixth vertex, call it $u$, is not connected to $w$ and has at most 4 neighbours $\set{a,b,c,v}$. If any of these connections to $u$ has multiplicity 3 or higher, then \cref{lem:pedge-div} shows that $N_{2,2}(H)\equiv 0 \mod 3$, since $N_{2,2}(H)$ fulfils the Martin recursion at $u$. By \cref{lem:pedge-div} and \cref{lem:Martin(G2)div9}, we also find $\Martin(G)\equiv 0 \mod 27$, hence we have established \eqref{eq:c2-base-p=3} in this case.

It remains to consider graphs $G$ where $u$ is connected to each of $\set{a,b,c,v}$ via a double edge. So we are only left to arrange 8 edges between $\set{a,b,c,v}$ to obtain an 8-regular graph $G$. Let $m_{ij}$ denote these multiplicities. The 8-regularity implies that $m_{ab}=m_{cv}$, $m_{ac}=m_{bv}$, and $m_{av}=m_{bc}$. As before, any edge of multiplicity $m_{ij}\geq 3$ will force $\Martin(G)\equiv 0 \mod 27$. Similarly, we will have $N_{2,2}(H)\equiv 0 \mod 3$, because the expansion at $u$ can only add further edges between $\set{a,b,c}$; so if $m_{aj}\geq 3$ already, it will stay strictly above 2 even in $G_D\setminus\set{v,w}$, and therefore $N_{2,2}(H)\equiv 0 \mod 3$ by \cref{lem:TFP-3}.

\begin{figure}
    \centering
    $G=\Graph[0.5]{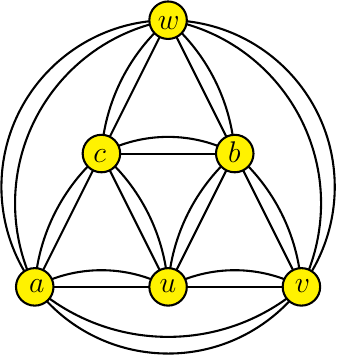}$
    \qquad
    $G'=\Graph[0.5]{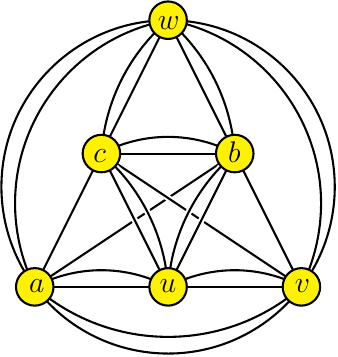}$
    \caption{The doubled octahedron $G=(C^6_{1,2})^{[2]}$ and the graph $G'$.}%
    \label{fig:octas}%
\end{figure}
The only remaining cases then are those where all $m_{ij}\leq 2$. Up to permuting the vertex labels, there are only two isomorphism classes of such graphs: $(m_{ab},m_{ac},m_{bc})=(0,2,2)$ and $(m_{ab},m_{ac},m_{bc})=(1,1,2)$. The corresponding graphs $G$ and $G'$ are shown in \cref{fig:octas}.
\begin{table}
    \centering
    \begin{tabular}{rcccccc}
    \toprule
    $(D_{ab},D_{ac},D_{bc})$ & $(2,0,0)$ & $(0,2,0)$ & $(0,0,2)$ & $(1,1,0)$ & $(1,0,1)$ & $(0,1,1)$ \\
    \midrule
    $d!/D!$ & $4$ & $4$ & $4$ & $16$ & $16$ & $16$ \\
    $G_D$  & $K_5^{[2]}$ & $\Graph[0.35]{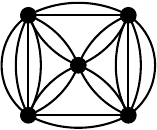}$ & $\Graph[0.35]{octaexp3}$ & $\Graph[0.3]{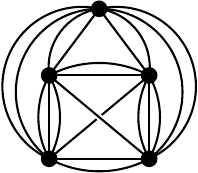}$ & $\Graph[0.3]{octaexp1}$ &  $\Graph[0.35]{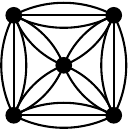}$ \\
    $G'_D$ & $\Graph[0.3]{octaexp1}$ & $\Graph[0.3]{octaexp1}$ & $\Graph[0.3]{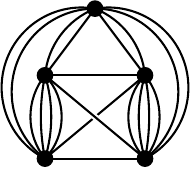}$ & $K_5^{[2]}$ & $\Graph[0.3]{octaexp1}$ & $\Graph[0.3]{octaexp1}$ \\
    \bottomrule
    \end{tabular}%
    \caption{Contributions to the vertex expansion at $u$ of the graphs from \cref{fig:octas}.}%
    \label{tab:octa-expansions}%
\end{table}
\begin{lemma}\label{octa-martins}
    The Martin invariants of the two graphs in \cref{fig:octas} are
    \begin{equation*}
    \Martin(G) = 84096 \equiv -9 \mod 27
    \quad\text{and}\quad
    \Martin(G') = 97920 \equiv -9 \mod 27.
\end{equation*}
\end{lemma}
\begin{proof}
Apply the expansion \eqref{eq:Martin-multi-expansion} to the vertex $u$. This produces the terms in \cref{tab:octa-expansions}. Two of the arising 5-vertex graphs are totally decomposable due to 4-fold edges, so that
\begin{equation*}
    \Martin\left(\Graph[0.35]{octaexp3}\right)
    =\Martin\left(\Graph[0.3]{octaexp4}\right)
    = (4!)^2.
\end{equation*}
We evaluate the remaining two 5-vertex graphs with the expansion \eqref{eq:Martin-multi-expansion} for the bottom left vertex. Together with \eqref{eq:Martin(K4abc)}, this results in
\begin{align*}
    \Martin\left(\Graph[0.35]{octaexp2}\right)
    &= \frac{2!3!3!}{1!1!2!}\Martin\left(K_4^{[2,3,3]}\right)
    = 1296 
    \quad\text{and}\\
    \Martin\left(\Graph[0.3]{octaexp1}\right)
    &= \frac{2!2!3!}{1}\Martin\left(K_4^{[2,3,3]}\right)
    + \frac{2!2!3!}{2}\Martin\left(K_4^{[2,2,4]}\right)
    + \frac{2!2!3!}{2}\Martin\left(K_4^{[2,3,3]}\right)
    =1584.
\end{align*}
Inserting these values and $\Martin(K_5^{[2]})=2016$ from \eqref{eq:Martin(K5r)} into \cref{tab:octa-expansions}, we obtain $\Martin(G)=84096$ and $\Martin(G')=97920$.
\end{proof}
It remains to show that $N_{2,2}(H)\equiv -1 \mod 3$ for $H=G\setminus\set{v,w}$ and $H'=G'\setminus\set{v,w}$. By \cref{lem:TFP-3}, we can compute $N_{2,2}(H)\mod 3$ by counting the number of transitions at $u$ that produce $G_D\cong K_5^{[2]}$ and then multiplying by $-1$. According to \cref{tab:octa-expansions}, in each case there is only one matrix $D$ with this property, thus $N_{2,2}\equiv -d!/D! \mod 3$ gives
\begin{equation*}
    N_{2,2}(H) \equiv -4 \equiv -1 \mod 3
    \qquad\text{and}\qquad
    N_{2,2}(H') \equiv -16 \equiv -1 \mod 3.
\end{equation*}
This finishes the proof that \eqref{eq:c2-base-p=3} holds for all $H=G\setminus\set{v,w}$ with 4 vertices, that is, for all graphs $G$ with 6 vertices.

Consequently, this completes the proof of \cref{thm:c2-martin} for all primes, especially $p=3$.

\section{Computer calculations}\label{sec:data}
We implemented the recursion \eqref{eq:Martin-multi-expansion} in {\MapleTM}, using {\nauty}~\cite{McKayPiperno:II} to identify isomorphic graphs.{\MapleNote} We use a cache table indexed by the canonical label of a graph, to avoid the recomputation of Martin invariants of isomorphic graphs.
To test our implementation, we confirmed that the results for $\Martin(G)$ and $\Martin(G^{[2]})$ for all 4-regular graphs with $\leq 13$ vertices, and $\Martin(G^{[2]})$ for all 3-regular graphs with $\leq 18$ vertices, agree with the values obtained by an independent, simple but much slower {\FORM} \cite{Vermaseren:NewFORM} program. We also checked in many cases that our results for $\Martin(G^{[r]})$ are compatible with known permanents \cite{Crump:ExtendedPermanent} and $c_2$-invariants \cite{HuSchnetzShawYeats:Further,BrownSchnetz:ModularForms} via \eqref{eq:Martin-Perm2} and \eqref{eq:c2-martin}.

Using this code, we computed the first few entries of the Martin sequences of all 3- and 4-regular graphs with small numbers of vertices. These results are provided in text files \Filename{Martin3.txt} and \Filename{Martin4.txt}, respectively, that form part of this paper. Below we give a summary of those computational results and highlight several observations.

\subsection{4-regular graphs}\label{sec:phi4}
Due to \cref{prop:edge-cuts} and the product identity \eqref{eq:martin vertex product}, we only consider cyclically 6-connected 4-regular graphs without a 3-vertex cut. The isomorphism classes of such graphs have been enumerated in \cite{Schnetz:Census} up to $n\leq 13$ vertices. They are labelled $P_{\ell,1}$, $P_{\ell,2}$, \ldots where $\ell=n-2\leq 11$ denotes the loop number of any decompletion. Explicit definitions of these graphs are given in a file provided with \cite{PanzerSchnetz:Phi4Coaction}.

\begin{table}\centering
\begin{tabular}{rccccccccc}
	\toprule
	loop order $\ell$ & 3 & 4 & 5 & 6 & 7 & 8 & 9 & 10 & 11 \\
	\midrule
	$\abs{\set{\text{graphs $G$}}}$ & 
	1 & 1 & 1 & 4 & 11 & 41 & 190 & 1182 & 8687 \\
	$\abs{\set{\Martin(G)}}$ &
	1 & 1 & 1 & 4 & 9 & 25 & 100 & 409 & 1622 \\
	$|\{\Martin(G^{[2]})\}|$ &
	1 & 1 & 1 & 4 & 9 & 29 & 129 & 776 & 6030 \\
 	$\abs{\set{\text{Hepp bounds $\Hepp(G\setminus v)$}}}$ &
	1 & 1 & 1 & 4 & 9 & 29 & 129 & 776 & 6030 \\
	\bottomrule
\end{tabular}
\caption{The number of graphs, Martin invariants, and Hepp bounds up to $11$ loops.}%
\label{tab:martin-stats}%
\end{table}

\Cref{tab:martin-stats} summarizes our findings for the first two entries $\Martin(G)$ and $\Martin(G^{[2]})$ of the Martin sequences. Explicit values are given for up to $\ell\leq 8$ loops in \cref{tab:first-martins}. We note:
\begin{itemize}
    \item Up to $7$ loops, the Martin invariant $\Martin(G)$ is a perfect period invariant: By this we mean that $\Martin(G_1)=\Martin(G_2)$ if and only if $\Period(G_1\setminus v)=\Period(G_2\setminus v)$. This follows from \cref{tab:first-martins} because at 7 loops, all periods are known \cite{PanzerSchnetz:Phi4Coaction} and the twist $P_{7,4}\leftrightarrow P_{7,7}$ and the duality $P_{7,5}\leftrightarrow P_{7,10}$ provide the only two relations.

    \item Starting at $8$ loops, there are pairs of graphs with the same Martin invariant, but different period. At 8 loops, there are 4 such pairs (highlighted in \cref{tab:first-martins}).

    \item If we consider instead the Martin invariants of the doubled graphs, then we find that for all graphs with $\ell\leq 11$ loops, $\Martin(G_1^{[2]})=\Martin(G_2^{[2]})$ holds if and only if $\Hepp(G_1\setminus v)=\Hepp(G_2\setminus v)$. Since the Hepp bound is expected to be a perfect period invariant for 4-regular graphs \cite[Conjecture~1.2]{Panzer:HeppBound}, we thus expect that $\Martin(G^{[2]})$ is a perfect period invariant up to 11 loops.

    \item For higher powers, we computed: $\Martin(G^{[3]})$ up to $\ell\leq 11$, $\Martin(G^{[4]})$ up to $\ell\leq 10$, $\Martin(G^{[5]})$ and $\Martin(G^{[6]})$ up to $\ell\leq 9$, $\Martin(G^{[7]})$ and $\Martin(G^{[8]})$ up to $\ell\leq 8$, and $\Martin(G^{[9]})$ and $\Martin(G^{[10]})$ up to $\ell \leq 7$. For all these higher powers we find, as for $\Martin(G^{[2]})$, that they coincide precisely when the Hepp bounds (and hence conjecturally the periods) coincide.
    All this data constitutes our evidence for \cref{con:Martin-perfect}.
\end{itemize}

\begin{figure}
    \centering
    \begin{tabular}{ccccccc}
    $\Graph[0.46]{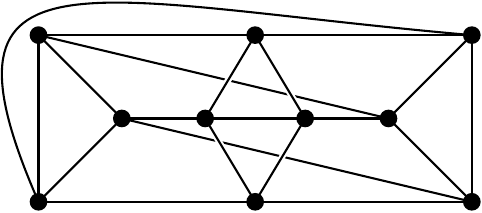}$ & $\leftrightarrow$ & $\Graph[0.57]{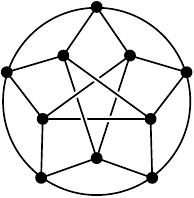}$
    & \quad\quad &
    $\Graph[0.43]{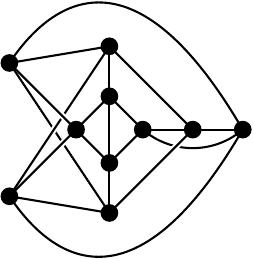}$ & $\leftrightarrow$ & $\Graph[0.48]{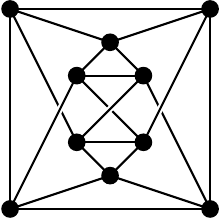}$
    \\
    $P_{8,30}$ & & $P_{8,36}$ & & $P_{8,31}$ & & $P_{8,35}$
    \\
    \end{tabular}
    \caption{Two unexplained pairs of graphs with apparently equal Martin sequences.}%
    \label{fig:830-836}%
\end{figure}
The most interesting instances of \cref{con:Martin-perfect} are the graphs in \cref{fig:830-836}. These graphs come in two pairs that share the same Hepp bound \cite{Panzer:HeppBound}, hence we expect that
\begin{equation}\label{eq:830-36per}
    \Period(P_{8,30}\setminus v)=\Period(P_{8,36}\setminus v)
    \quad\text{and}\quad
    \Period(P_{8,31}\setminus v)=\Period(P_{8,35}\setminus v),
\end{equation}
which does not follow from any of the known period identities. Correspondingly, the entire Martin sequences should agree, that is, we conjecture that
\begin{equation}\label{eq:830-36mar}
    \Martin(P_{8,30}^{[r]})=\Martin(P_{8,36}^{[r]})
    \quad\text{and}\quad
    \Martin(P_{8,31}^{[r]})=\Martin(P_{8,35}^{[r]})
\end{equation}
hold for all positive integers $r$. Our calculations confirm these identities explicitly in the range $1\leq r\leq 8$. Modulo primes $p=2r+1$, \eqref{eq:830-36mar} was checked much farther, up to $r\leq 50$ (via \cref{thm:Martin-Perm2} from the permanents in \cite[Appendix~A]{Crump:ExtendedPermanent}), and modulo the smaller primes $p=r+1$, \eqref{eq:830-36mar} was checked up to $r=100$ (via \cref{thm:c2-martin} from the $c_2$ invariants calculated in \cite{BrownSchnetz:ModularForms}).
We have thus accumulated overwhelming evidence for \eqref{eq:830-36mar}. In our view, this strengthens considerably the credibility of the conjecture \eqref{eq:830-36per}.

We also calculated examples that show that several of our results are best possible:
\begin{itemize}
    \item 
Feynman period identities do not extend from the Martin invariant to the full Martin polynomial. For example, $P_{7,4}$ and $P_{7,7}$ are related by a twist, and also by planar duality (\cref{fig:duality}), but their Martin polynomials
\begin{align*}
    \MartinPol(P_{7,4},x) &= x(\phantom{3}7x^3+210x^2+1054x+1320)
    \quad\text{and}
    \\
    \MartinPol(P_{7,7},x) &= x(13x^3+200x^2+1030x+1320)
    \quad\text{differ.}
\end{align*}

\item The divisibility of $\Martin(G^{[r]})$ by $(r!)^{2(n-3)}$ for $r$-fold duplications of 4-regular graphs (\cref{martin-power-divisibility}) does not persist for other $4r$-regular graphs. For example, the complete graph $K_9$ is an 8-regular graph on $n=9$ vertices, but
\begin{equation}
    \Martin(K_9) = 2^9\cdot 3^5\cdot 5^2 \cdot 7\cdot 17\cdot 167
\end{equation}
is not divisible by $2^{12}$.

\item
The Martin invariant is determined by the cycle matroid, but it is not determined by the Tutte polynomial $\Tutte{G}(x,y)$. For example, $P_{11,8403}$ and $P_{11,8404}$ have the same Tutte polynomial, but $\Martin(P_{11,8403})=12438$ and $\Martin(P_{11,8404})=12442$ differ.
\end{itemize}

Finally, we comment on the quantitative relation between the Martin invariant and the period. Using the known periods from \cite{PanzerSchnetz:Phi4Coaction}, we find a negative approximate correlation (\cref{fig:phi4-martin-period}): graphs with smaller $\Martin(G)$ tend to have larger periods $\Period(G\setminus v)$.

We find in particular that, at each loop order $\ell\leq 11$, the Martin invariant is minimized by a unique graph: the circulant $P_{\ell,1}=C^{\ell+2}_{1,2}$. Its decompletion is called zigzag graph, whose period is known and expected to be the largest \cite{BrownSchnetz:ZigZag}. While the period maximization is a conjecture, the Martin minimization was shown in \cite[Theorem~6.2]{BouchetGhier:BetaIso4}:
\begin{proposition}\label{thm:zigzag-min}
    If $G$ is a cyclically 6-connected 4-regular graph with $n\geq 5$ vertices, then $\Martin(G)\geq \Martin(C^n_{1,2})$. Furthermore, equality holds only if $G\cong C^n_{1,2}$, and the value is
\begin{equation}\label{eq:Martin(C12)}
    \Martin(C^n_{1,2})
    = \frac{(3n-2)\cdot 2^{n-3}-2\cdot (-1)^n}{9}.
\end{equation}
\end{proposition}
The expansion \eqref{eq:Martin-recursion-4} gives the recurrence $\Martin(C^n_{1,2})= \Martin(C^{n-1}_{1,2})+2\Martin(C^{n-2}_{1,2}) + 2^{n-4}$ where $2^{n-4}$ comes from a totally decomposable graph, leading to the closed form \eqref{eq:Martin(C12)}. That this is the minimum among cyclically 6-connected graphs is related to the following properties.
\begin{itemize}
    \item If $G$ has a 3-vertex cut, then $\Martin(G)=\Martin(G_1)\Martin(G_2)\geq \Martin(C^{k}_{1,2}) \Martin(C^{r}_{1,2})>\Martin(C^n_{1,2})$ cannot be minimal. Here $k+r=n+3$ and $k,r\geq 5$.
    \item If $G$ has no 3-vertex cut, then at each vertex, there is at least one transition such that $G_{\tau}$ is cyclically 6-connected \cite[Lemma~3.2]{FleischnerGenestJackson:CC4}.
    \item If at each vertex there is only one transition such that $G_{\tau}$ is cyclically 6-connected, then $G\cong C^n_{1,2}$ by \cite[Theorem~6.1]{BouchetGhier:BetaIso4}.
\end{itemize}

At the other extreme, our data shows that at each loop order $\ell\leq 11$, the Martin invariant is also maximized by a unique graph, see \cref{tab:Martin4-max}. We did not identify a general pattern, but the non-trivial automorphism groups and the prevalence of circulants among these graphs suggest some underlying structure. The same graphs were recently noticed to maximize the number of connected sets \cite{CambieGoedgebeurJokken:MaxConReg}.
\begin{table}
    \centering
    \begin{tabular}{rrlrr}
    \toprule
        loops & $\max\Martin(G)$ & $G$ & HoG & $|\Aut(G)|$ \\
    \midrule
         3 & 6 & $P_{3,1}=K_5$ & \HoG{462} & 120 \\
         4 & 14 & $P_{4,1}=C^6_{1,2}$ & \HoG{226} & 48 \\ 
         5 & 36 & $P_{3,1}^2$ & \HoG{912} & 48 \\
		6  & 108 & $P_{6,4}=K_{4,4}=C_{1,3}^8$ & \HoG{570} & 1152 \\
		7  & 256 & $P_{7,8}$ & \HoG{50432} & 16 \\
		8  & 728 & $P_{8,40}=C^{10}_{1,4}=C^{10}_{2,3}$ & \HoG{45705} & 320 \\
		9  & 1894 & $P_{9,190}=C^{11}_{1,3}=C^{11}_{1,4}=C^{11}_{2,3}$ & \HoG{50433} & 22 \\
		10 & 5300 & $P_{10,1182}=C^{12}_{2,3}$ & \HoG{33319} & 48 \\
		11 & 14376 & $P_{11,8687}=C^{13}_{1,5}=C^{13}_{2,3}$ & \HoG{21065} & 52 \\
    \bottomrule
    \end{tabular}
    \caption{The unique primitive $\phi^4$ graphs that maximize $\Martin(G)$ at each loop order $\leq 11$. Complete (bipartite) and circulant graphs are highlighted, and the column HoG links to their entries in the House of Graphs \cite{HoG}.}%
    \label{tab:Martin4-max}%
\end{table}
\begin{table}
    \centering
    \begin{tabular}{rrr}
    \toprule
        Graph $G$ & $\Martin(G)$ & $\Martin(G^{[2]})/4^{\ell-1}$ \\
    \midrule
        $P_{3,1}$ &   6 &    126\\
    \midrule
        $P_{4,1}$ &  14 &   1314\\
    \midrule
        $P_{5,1}$ &  34 &  14706\\
    \midrule
        $P_{6,1}$ &  78 & 147546\\
        $P_{6,2}$ &  86 & 180594\\
        $P_{6,3}$ &  92 & 212436\\
        $P_{6,4}$ & 108 & 296676\\
    \midrule
        $P_{7,1}$ &   178 &  1453914\\
        $P_{7,2}$ &  202 &   1891314\\
        $P_{7,3}$ &  210 &  2083770\\
        $P_{7,4},P_{7,7}$ & 220 & 2313900\\
        $P_{7,6}$ &  226 & 2454426\\
        $P_{7,5},P_{7,10}$ & 228 & 2577204\\
        $P_{7,9}$ &  240 & 2929680\\
        $P_{7,11}$ & 246 & 3116286\\
        $P_{7,8}$ &  256 & 3358656\\
    \bottomrule
    \end{tabular}
    \begin{tabular}{rrr}
    \toprule
        Graph $G$ & $\Martin(G)$ & $\Martin(G^{[2]})/4^{\ell-1}$ \\
    \midrule
    $P_{8,1}$ & 398 & 13881906\\ \rowcolor{Goldenrod}
$P_{8,2}$ & 470 & 19288170\\ \rowcolor{Goldenrod}
$P_{8,3}$ & 470 & 19560330\\
$P_{8,4}$ & 494 & 21875634\\
$P_{8,6},P_{8,9}$ & 510 & 23224770\\
$P_{8,5}$ & 516 & 24331644\\
$P_{8,7},P_{8,8}$ & 518 & 24330906\\
$P_{8,11},P_{8,15}$ & 524 & 25080084\\
$P_{8,14}$ & 534 & 26486154\\
$P_{8,13},P_{8,21}$ & 542 & 26900226\\ \rowcolor{Goldenrod}
$P_{8,10},P_{8,22}$ & 548 & 27340956\\ \rowcolor{Goldenrod}
$P_{8,12}$ & 548 & 28399356\\
$P_{8,18},P_{8,25}$ & 564 & 30075084\\
$P_{8,20}$ & 566 & 30153834\\
$P_{8,19},P_{8,27}$ & 572 & 31573476\\
$P_{8,17},P_{8,23}$ & 582 & 32301306\\ \rowcolor{Goldenrod}
$P_{8,16}$ & 584 & 31092984\\ \rowcolor{Goldenrod}
$P_{8,29}$ & 584 & 33515064\\
$P_{8,30},P_{8,36}$ & 602 & 36055206\\
$P_{8,26},P_{8,28}$ & 608 & 36570816\\
$P_{8,33}$ & 618 & 38238966\\
$P_{8,32},P_{8,34}$ & 620 & 38026260\\
$P_{8,31},P_{8,35}$ & 624 & 38998224\\
$P_{8,37}$ & 638 & 41602626\\ \rowcolor{Goldenrod}
$P_{8,24}$ & 656 & 42769584\\ \rowcolor{Goldenrod}
$P_{8,38}$ & 656 & 44586864\\
$P_{8,39}$ & 660 & 45058860\\
$P_{8,41}$ & 684 & 50848884\\
$P_{8,40}$ & 728 & 54288936\\
    \bottomrule
    \end{tabular}
    \caption{The first two Martin invariants of $\phi^4$ primitives $P_{\ell,i}$ with $n=\ell+2\leq 13$ vertices. All identities (two graphs in the same row) are explained by twists or dualities, except for the two unexplained identities $P_{8,30}\leftrightarrow P_{8,36}$ and $P_{8,31}\leftrightarrow P_{8,35}$.}%
    \label{tab:first-martins}%
\end{table}

\subsection{3-regular graphs}
\label{sec:phi3}

A 3-regular graph has an even number $n$ of vertices, and we denote by $\ell=n/2-1$ the loop number of its decompletions. The cyclically 4-connected 3-regular graphs with $\ell\leq 9$ (that is, up to $n\leq 20$ vertices) have been enumerated in \cite{BorinskySchnetz:RecursivePhi3}, where also many of their periods were computed. We use the notation $P_{\ell,k}$ from that paper; the definition of these graphs as edge lists can be found in the file \Filename{PeriodsPhi3.txt} from \cite{BorinskySchnetz:RecursivePhi3}.

\begin{table}\centering
\begin{tabular}{rccccccccc}
	\toprule
	loop order $\ell$ & 1 & 2 & 3 & 4 & 5 & 6 & 7 & 8 & 9\\
	\midrule
	$\abs{\set{\text{graphs $G$}}}$ & 
	1 & 1 & 2 & 5 & 18 & 84 & 607 & 6100 & 78824\\
	$|\{\Martin(G^{[2]})\}|$ &
	1 & 1 & 2 & 5 & 17 & 72 & 441 & 4015 & 47074 \\
	$|\{\Martin(G^{[4]})\}|$ &
	1 & 1 & 2 & 5 & 17 & 73 & 472 & 4534 & 58432 \\
	$\abs{\set{\text{periods $\Period(G\setminus v)$}}}$ &
	1 & 1 & 2 & 5 & 17 & 73 & ? & ? & ? \\
	$\abs{\set{\text{Hepp bounds $\Hepp(G\setminus v)$}}}$ &
	1 & 1 & 2 & 5 & 16 & 72 & 470 & 4522 & 58409 \\
	\bottomrule
\end{tabular}
\caption{The number of Martin invariants, periods and Hepp bounds among cyclically $4$-connected 3-regular graphs with $2\ell+2$ vertices.}%
\label{tab:martin-stats-phi3}%
\end{table}
We computed $\Martin(G^{[2]})$ and $\Martin(G^{[4]})$ for all graphs $P_{\ell,k}$ from \cite{BorinskySchnetz:RecursivePhi3} with $\ell\leq 9$, and $\Martin(G^{[6]})$ for $\ell\leq 8$. A summary of our findings is given in \cref{tab:martin-stats-phi3}. We observe in particular:
\begin{itemize}
    \item Up to $\ell\leq 5$, $\Martin(G^{[2]})$ is a perfect period invariant: $\Period(G\setminus v)$ is the same for two graphs if and only if $\Martin(G^{[2]})$ is the same. This is not true for the Hepp bound: $P_{5,5}$ and $P_{5,9}$ have different periods but equal Hepp bound \cite[\S 1.6]{BorinskySchnetz:RecursivePhi3}.
    \item At $\ell=6$ loops, there is a single pair of graphs $\set{P_{6,59},P_{6,68}}$ with the property that their periods differ \cite{BorinskySchnetz:RecursivePhi3} while $\Martin(P_{6,59}^{[2]})=\Martin(P_{6,68}^{[2]})=2^{29}\cdot 5^2$ agree. However, the next term $\Martin(G^{[4]})$ in the Martin sequence takes different values on this pair, and we find that indeed $\Martin(G^{[4]})$ is a perfect period invariant at $\ell=6$.
    \item We conjecture that $\Martin(G^{[4]})$ remains a perfect period invariant also for $7\leq \ell \leq 9$. In this range, however, not all periods are known, and the Hepp bounds are only expected to give a lower bound on the number of periods.\footnote{We do not know that the Hepp bounds give a lower bound on the number of periods in this context, but we do expect it.  Specifically, we expect that $\Period(G_1)=\Period(G_2)$ implies $\Hepp(G_1)=\Hepp(G_2)$.}
\end{itemize}
Because the Martin invariant is defined only for graphs with even degree, we define the Martin sequence of a 3-regular graph as
\begin{equation*}
    \Martin(G^{\bullet})=\left( \Martin(G^{[2]}), \Martin(G^{[4]}), \Martin(G^{[6]}),\ldots\right).
\end{equation*}
All available data is compatible with the following variant of \cref{con:Martin-perfect}, which gives a purely combinatorial characterization of 3-regular graphs with equal periods:
\begin{conjecture}\label{con:Martin-perfect-phi3}
    Two cyclically 4-connected 3-regular graphs $G_1$ and $G_2$ with the same number of vertices have equal period $\Period(G_1\setminus v_1)=\Period(G_2\setminus v_2)$ if and only if they have equal Martin sequences $\Martin(G_1^{\bullet})=\Martin(G_2^{\bullet})$.
\end{conjecture}
\begin{lemma}
    For a cyclically 4-edge connected, 3-regular graph with $2\ell+2$ vertices, the Martin invariant $\Martin(G^{[2r]})$ is divisible by $r!\cdot[(2r)!]^{3\ell-2}$.
\end{lemma}
\begin{proof}
    This follows similarly to \cref{martin-power-divisibility}, by adapting \cref{diag-multinomial-multiple} to the case of half-integer $k=3/2$.
\end{proof}

Turning to the quantitative relation between periods and Martin invariants, we find the same kind of behaviour as for 4-regular graphs; a plot of the known periods from \cite{BorinskySchnetz:RecursivePhi3} versus $\Martin(G^{[2]})$ shows an approximate power law correlation, similar to \cref{fig:phi4-martin-period}.

We also find that at any fixed number of vertices $8\leq 2\ell+2\leq 20$, there is a unique graph that minimizes $\Martin(G^{[2]})$: the prism $Y_{\ell+1}=K_2\times C_{\ell+1}$ over a base  polygon $C_{\ell+1}$ with $\ell+1$ sides (see \cref{fig:prisms}). We expect that this persists for higher $\ell$:
\begin{figure}
    \centering
    \begin{tabular}{ccccccc}
    \includegraphics[height=14mm]{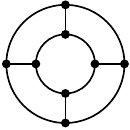} &
    \includegraphics[height=14mm]{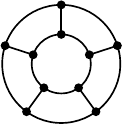} &
    \includegraphics[height=14mm]{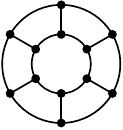} &
    \includegraphics[height=14mm]{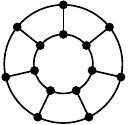} &
    \includegraphics[height=14mm]{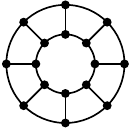} &
    \includegraphics[height=14mm]{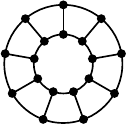} &
    \includegraphics[height=14mm]{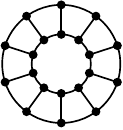}
    \\
    $P_{3,1}$ & $P_{4,2}$ & $P_{5,1}$ & $P_{6,8}$ & $P_{7,1}$ & $P_{8,79}$ & $P_{9,1}$
    \\
    \end{tabular}
    \caption{Prism graphs $Y_{\ell+1}=K_2\times C_{\ell+1}$ and their labels in \cite{BorinskySchnetz:RecursivePhi3}.}%
    \label{fig:prisms}%
\end{figure}
\begin{conjecture}\label{conj:prism-min}
    If $G$ is a cyclically 4-connected 3-regular graph with $2\ell+2\geq 8$ vertices, then $\Martin(G^{[2]})\geq \Martin(Y_{\ell+1}^{[2]})$ and equality holds only if $G\cong Y_{\ell+1}$.
\end{conjecture}
We can give a closed formula for these minimal Martin invariants: For all $\ell\geq 2$,
\begin{equation}\label{eq:Martin(prism)}
    \Martin( Y_{\ell+1}^{[2]})=4^{2\ell-1}[3^{\ell-2}(4\ell-1)-1].
\end{equation}
Our proof uses the approach of \cref{martin-families}, namely we obtain a transfer matrix of rank 3 to derive a recurrence relation in the following manner:  Beginning with one rung of prism visualized as a cyclic ladder, apply the Martin recurrence to the two vertices of the rung.  This results in the three terms illustrated in \cref{fig:baseblock}.
\begin{figure}
    \centering
    \includegraphics{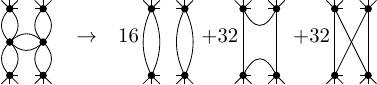}
    \caption{The result of applying the Martin recurrence to two vertices of a doubled rung.  The remainder of the graph connects at top and bottom where indicated and is identical in all the terms.}%
    \label{fig:baseblock}%
\end{figure}
For each of the graphs on the right-hand side, consider how the next rung below interacts with the part drawn.  For the first term on the right-hand side we have simply returned to the left-hand side, but with the entire graph having one rung fewer.  For the second and third terms on the right-hand side, we have different graphs to consider but in both cases we can, as before, reduce the two vertices where the next rung connected.  The result in both cases will include some graphs we've already seen on the right-hand side and some new ones.  Ultimately, we need four more equations, as illustrated in \cref{fig:otherblocks}.
\begin{figure}
    \centering
    \includegraphics{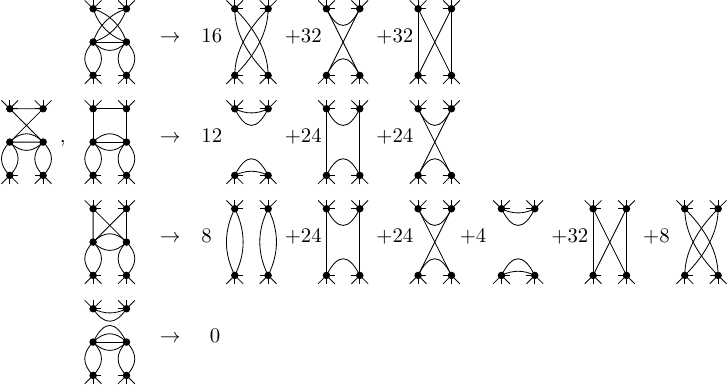}
    \caption{The result of applying the Martin recurrence to the other graphs needed to get a system of recurrences.}%
    \label{fig:otherblocks}%
\end{figure}
Letting $a_n$ represent the Martin invariant of the original ladder with $n$ rungs, and $b_n$, $c_n$, $d_n$, and $e_n$ likewise for the graphs of \cref{fig:otherblocks}, the calculations of the figures can be rephrased as
\begin{align*}
    a_n & = 16a_{n-1} + 32c_{n-1} + 32d_{n-1} \\
    b_n & = 16b_{n-1} + 32c_{n-1} + 32d_{n-1} \\
    c_n & = 12e_{n-1} + 48c_{n-1} \\
    d_n & = 8a_{n-1} + 48c_{n-1} + 4e_{n-1} + 32d_{n-1} + 8b_{n-1} \\
    e_n & = 0
\end{align*}
Observe that $d_n = \frac{1}{2}a_n + \frac{1}{2}b_n + \frac{1}{3}c_n$ (this is most apparent after reducing just one of the rung edges, but can also be verified from the equations above).  Using this fact along with $e_n=0$ we can simplify to a system of three recurrences
\begin{equation*}
\begin{pmatrix} a_n \\ b_n \\ c_n \end{pmatrix}
=
\begin{pmatrix}
32 & 16 & \frac{128}{3} \\
16 & 32 & \frac{128}{3} \\
0 & 0 & 48
\end{pmatrix}
\begin{pmatrix} a_{n-1} \\ b_{n-1} \\ c_{n-1} \end{pmatrix}
.
\end{equation*}
Solving the system gives the claimed formula.

\begin{remark}
    As for $\phi^4$, our data shows that also that at each loop order $\ell\leq 9$, there is a unique $\phi^3$ primitive graph that \emph{maximizes} the Martin invariant $\Martin(G^{[2]})$. Those graphs are collected in \cref{tab:Martin3-max}. Again, the same graphs were recently noticed to maximize the number of connected sets \cite{CambieGoedgebeurJokken:MaxConReg}.
\end{remark}

\begin{table}\centering
    \begin{tabular}{rrlrr}
    \toprule
        loops & $\max\Martin(G^{[2]})$ & $G$ & HoG & $|\Aut(G)|$ \\
\midrule
1 & $1\cdot 2^{3\phantom{1}}$ 
    & $P_{1,1}$ & \HoG{74} & 24 \\
2 & $4\cdot 2^{7\phantom{1}}$ 
    & $P_{2,1}$ & \HoG{84} & 72 \\
3 & $17\cdot 2^{11}$ 
    & $P_{3,2}$ & \HoG{640} & 16 \\
4 & $92\cdot 2^{15}$ 
    & $P_{4,5}$ & \HoG{660} & 120 \\
5 & $460\cdot 2^{19}$ 
    & $P_{5,7}$ & \HoG{27409} & 16 \\
6 & $2760\cdot 2^{23}$ 
    & $P_{6,7}$ & \HoG{1154} & 336 \\
7 & $14820\cdot 2^{27}$ 
    & $P_{7,21}$ & \HoG{1229} & 96 \\
8 & $88536\cdot 2^{31}$ 
    & $P_{8,3165}$ & \HoG{50424} & 12 \\
9 & $530093\cdot 2^{35}$ 
    & $P_{9,54494}$ & \HoG{50425} & 8 \\
\bottomrule
\end{tabular}
\caption{The unique graphs that maximize $\Martin(G^{[2]})$ among $\phi^3$ primitives at a fixed loop order. The column HoG links to their entries in the House of Graphs \cite{HoG}.}%
\label{tab:Martin3-max}%
\end{table}

\bibliography{refs}

\end{document}